\documentclass[11pt, a4paper, oneside]{amsart}

\usepackage[english]{babel}
\usepackage{amsmath, amsthm, amsfonts, mathrsfs, amssymb}
\usepackage{mathtools}
\mathtoolsset{centercolon}
\usepackage{booktabs}
\usepackage[shortlabels]{enumitem}
\setlist[itemize]{leftmargin=20pt}
\usepackage[colorlinks, citecolor = blue]{hyperref}
\usepackage[hmargin=3cm,vmargin=3cm]{geometry}
\usepackage[color=green!40]{todonotes}
\usepackage{comment}

\usepackage{color}
\usepackage{graphicx}
\usepackage{tikz-cd}
\usetikzlibrary{arrows}

\newcommand{\R}{\ensuremath{\mathbf{R}}}

\newcommand{\mc}{\mathcal}

\DeclarePairedDelimiter{\nrm}\lVert\rVert

\DeclareMathOperator{\loc}{loc}
\DeclareMathOperator{\supp}{supp}
\DeclareMathOperator{\ind}{\mathbf{1}}

\DeclareMathOperator*{\esssup}{ess\,sup}
\DeclareMathOperator*{\essinf}{ess\,inf}

\renewcommand{\emptyset}{\varnothing}
\def\avint_#1{\mathchoice{\mathop{\kern 0.2em\vrule width 0.6em height 0.69678ex depth -0.58065ex \kern -0.8em \intop}\nolimits_{\kern -0.4em#1}}{\mathop{\kern 0.1em\vrule width 0.5em height 0.69678ex depth -0.60387ex \kern -0.6em \intop}\nolimits_{#1}} {\mathop{\kern 0.1em\vrule width 0.5em height 0.69678ex depth -0.60387ex \kern -0.6em \intop}\nolimits_{#1}} {\mathop{\kern 0.1em\vrule width 0.5em height 0.69678ex depth -0.60387ex \kern -0.6em \intop}\nolimits_{#1}}}
\newtheorem{TheoremLetter}{Theorem}
{}
\newtheorem{theorem}{Theorem}
\newtheorem*{theorem*}{Theorem}
\newtheorem{corollary}[theorem]{Corollary}
\newtheorem{lemma}[theorem]{Lemma}
\newtheorem{proposition}[theorem]{Proposition}

\newtheorem{conjecture}[theorem]{Conjecture}
\newtheorem{question}[theorem]{Question}
\newtheorem*{question*}{Question}
\newtheorem{ConjectureLetter}[TheoremLetter]{Conjecture}
{}

{}
\newtheorem{QuestionLetter}[TheoremLetter]{Question}
{}

\theoremstyle{remark}
\newtheorem{remark}[theorem]{Remark}
\newtheorem{example}[theorem]{Example}

\theoremstyle{definition}
\newtheorem{definition}[theorem]{Definition}

\numberwithin{theorem}{section}
\numberwithin{equation}{section}
\title{The Muckenhoupt condition}
\author{Zoe Nieraeth}
\thanks{Z. N. is supported by the grant Juan de la Cierva formación 2021 FJC2021-046837-I, the Basque Government through the BERC 2022-2025 program, by the Spanish State Research Agency project PID2020-113156GB-I00/AEI/10.13039/501100011033 and through BCAM Severo Ochoa excellence accreditation SEV-2023-2026.}
\address{Zoe Nieraeth (she/her)\hfill\break\indent BCAM\textendash  Basque Center for Applied Mathematics, Bilbao, Spain}
\email{zoe.nieraeth@gmail.com}
\allowdisplaybreaks

\begin{document}
\begin{abstract}
The goal of this paper is to unify the theory of weights beyond the setting of weighted Lebesgue spaces in the general setting of quasi-Banach function spaces. We prove new characterizations for the boundedness of singular integrals, pose several conjectures, and prove partial results related to the duality of the Hardy-Littlewood maximal operator. Furthermore, we give an overview of the theory applied to weighted variable Lebesgue, Morrey, and Musielak-Orlicz spaces.
\end{abstract}

\keywords{Banach function space, Hardy-Littlewood maximal operator, Muckenhoupt weights, Calder\'on-Zygmund operators}

\subjclass[2020]{Primary: 42B25; Secondary: 46E30}

%42B25  Maximal functions, Littlewood-Paley theory
%46E30  Spaces of measurable functions (Lp-spaces, Orlicz spaces, Köthe function spaces, Lorentz spaces, rearrangement invariant spaces, ideal spaces, etc.)

\maketitle

\section{Introduction}
\subsection{The Muckenhoupt condition beyond weighted Lebesgue spaces}
A fundamental problem in the theory of singular integrals is:
\begin{equation}\tag{P}\label{P}
\begin{array}{c}
\text{\emph{Given a class of spaces $X$, under which conditions do we have $T:X\to X$ for}} \\ 
\text{\emph{all Calder\'on-Zygmund operators $T$?}}
\end{array}
\end{equation}
When the class of spaces is the class of weighted Lebesgue spaces $L^p_w(\R^d)$ for exponents $1\leq p\leq\infty$ and weights $w$, a very satisfying answer to \eqref{P} can be given through the Muckenhoupt $A_p$ condition 
\[
[w]_p:=\sup_Q\Big(\frac{1}{|Q|}\int_Q\!w^p\,\mathrm{d}x\Big)^{\frac{1}{p}}\Big(\frac{1}{|Q|}\int_Q\!w^{-p'}\,\mathrm{d}x\Big)^{\frac{1}{p'}}<\infty,
\]
where the supremum is taken over all cubes $Q\subseteq\R^d$, $p'$ denotes the H\"older conjugate of $p$, and the average is interpreted as an essential supremum when the corresponding exponent is infinite. Moreover, this condition is intricately linked to the Hardy-Littlewood maximal operator 
\[
Mf:=\sup_Q\Big(\frac{1}{|Q|}\int_Q\!|f|\,\mathrm{d}x\Big)\ind_Q.
\]
In this paper we will normalize our weights using the multiplier approach in the definition of weighted Lebesgue spaces
\[
\|f\|_{L^p_w(\R^d)}:=\|wf\|_{L^p(\R^d)}=\begin{cases}
\displaystyle\Big(\int_{\R^d}\!|wf|^p\,\mathrm{d}x\Big)^{\frac{1}{p}} & \text{if $p<\infty$;} \\
\displaystyle\esssup_{\R^d}|wf| & \text{if $p=\infty$.}
\end{cases}
\]
Introducing weights as multipliers as opposed to introducing them through a change of measure dates back at least to \cite{MW74}. An overview and comparison of these approaches can be found in \cite[Section~3.4]{LN23b}, see also \cite{CFN12}. We have the following classical characterizations.
\begin{theorem}\label{thm:Awlequiv}
Let $1\leq p\leq\infty$ and let $w$ be a weight. The following are equivalent:
\begin{enumerate}[(i)]
\item\label{it:wl1} $T:L^p_w(\R^d)\to L^p_w(\R^d)$ for all Calder\'on-Zygmund operators $T$;
\item\label{it:wl2} $R_j:L^p_w(\R^d)\to L^p_w(\R^d)$ for all Riesz transforms, $R_j$, $j=1,\ldots,d$;
\item\label{it:wl3} $M:L^p_w(\R^d)\to L^p_w(\R^d)$ and $M:L^{p'}_{w^{-1}}(\R^d)\to L^{p'}_{w^{-1}}(\R^d)$;
\item\label{it:wl4} $1<p<\infty$ and $w\in A_p$.
\end{enumerate}
\end{theorem}
See, e.g., \cite{Gr14a}. Note that we need to assume both bounds of $M$ in \ref{it:wl3} to exclude the exponents $p=1$ and $p=\infty$. If we had assumed initially that $1<p<\infty$, then it would suffice to only have $M:L^p_w(\R^d)\to L^p_w(\R^d)$. We remark here that when $p=\infty$, the bound
\[
M:L^\infty_w(\R^d)\to L^\infty_w(\R^d)
\]
is characterized by the condition $w\in A_\infty$ (not to be confused with the Fujii-Wilson condition), and is equivalent to the condition $w^{-1}\in A_1$. This result was first obtained by Muckenhoupt in \cite{Mu72}.

To consider \eqref{P} beyond the setting of weighted Lebesgue spaces, we consider the general setting of quasi-Banach function spaces. For the precise definition, we refer the reader to Subsection~\ref{subsec:qBFS}. If $X$ is a quasi-Banach function space, we define its K\"othe dual $X'$ through
\[
\|g\|_{X'}:=\sup_{\|f\|_X=1}\int_{\R^d}\!|fg|\,\mathrm{d}x.
\]
We say that a quasi-Banach function space $X$ satisfies the Muckenhoupt condition if $\ind_Q\in X$ and $\ind_Q\in X'$ for all cubes $Q$, and there exists a constant $C\geq 1$ such that for all cubes $Q$ we have
\[
\|\ind_Q\|_X\|\ind_Q\|_{X'}\leq C|Q|.
\]
In this case we write $X\in A$ and denote the smallest possible constant $C$ by $[X]_A$. Note that for $1\leq p\leq\infty$ and $w$ a weight we have $L^p_w(\R^d)\in A$ if and only if $w\in A_p$, with
\[
[L^p_w(\R^d)]_A=[w]_p.
\]

We say that a quasi-Banach function space $X$ satisfies the Fatou property if: 
\begin{itemize}
    \item For every sequence $0\leq f_n\uparrow f$ with $\sup_{n\geq 1}\|f_n\|_X<\infty$, we have $f\in X$ with $\|f\|_X=\sup_{n\geq 1}\|f_n\|_X$.
\end{itemize}
It was shown in \cite{Ru18} that \ref{it:wl1}-\ref{it:wl3} remain equivalent when replacing $L^p_w(\R^d)$ by a Banach function space $X$ with the Fatou property in Theorem~\ref{thm:Awlequiv}. However, rather than \ref{it:wl4} being equivalent, the Muckenhoupt condition is now a strictly weaker condition. To formulate the result, we say that a collection of cubes $\mc{S}$ is called sparse if there exists a pairwise disjoint collection $(E_Q)_{Q\in\mc{S}}$ of subsets $E_Q\subseteq Q$ satisfying $|E_Q|\geq\tfrac{1}{2}|Q|$, and write
\[
A_{\mc{S}}f:=\sum_{Q\in\mc{P}}\langle |f|\rangle_Q\ind_Q,\quad \langle f\rangle_Q:=\frac{1}{|Q|}\int_Q\!f\,\mathrm{d}x.
\]
Then we have the following result.
\begin{theorem}\label{thm:Bbfs}
Let $X$ be a Banach function space satisfying the Fatou property. The following are equivalent:
\begin{enumerate}[(i)]
\item\label{it:bfs1} For all Calder\'on-Zygmund operators $T$ there is a $C>0$ such that  
\[
\|Tf\|_X\leq C\|f\|_X,\quad f\in X\cap L^\infty_c(\R^d);
\]
\item\label{it:bfs2} For all Riesz transforms $R_j$, $j=1,\ldots,d$, there is a $C>0$ such that
\[
\|R_jf\|_X\leq C\|f\|_X,\quad f\in X\cap L^\infty_c(\R^d);
\]
\item\label{it:bfs3} $M:X\to X$ and $M:X'\to X'$.
\item\label{it:bfssparse} $A_{\mc{S}}:X\to X$ uniformly with respect to all sparse collections $\mc{S}$;
\end{enumerate}
Moreover, any of these equivalent properties imply
\begin{enumerate}[(i)]
\setcounter{enumi}{4} 
\item\label{it:bfs4} $X\in A$,
\end{enumerate}
but the converse does not hold in general.
\end{theorem}

\begin{remark}
Suppose $T$ is a (sub)linear operator for which there is a $C>0$ such that 
\begin{equation}\label{eq:introbddcompact}
\|Tf\|_X\leq C\|f\|_X,\quad  f\in X\cap L^\infty_c(\R^d).
\end{equation}
When $X$ is order-continuous (for example, when $X=L^p_w(\R^d)$ for $p<\infty$ or, more generally, when $X$ is $s$-concave for some $s<\infty$), this is equivalent to the boundedness of $T:X\to X$. The reason that the full boundedness of $M$ appears in \ref{it:bfs3} without this extra assumption, is because in the case that $T=M$, the Fatou property of $X$ suffices to show that \eqref{eq:introbddcompact} is equivalent to $M:X\to X$.
\end{remark}

The most difficult implication is \ref{it:bfs2}$\Rightarrow$\ref{it:bfs3}, a proof of which can be found in \cite{Ru18}. The implication \ref{it:bfs3}$\Rightarrow$\ref{it:bfssparse} follows from a direct computation which can be found, e.g., in \cite[Lemma~3.4]{LN23a}. The implication \ref{it:bfssparse}$\Rightarrow$\ref{it:bfs1} follows from sparse domination of Calder\'on-Zygmund operators, see \cite{Le13a}. 

Alternatively, one can also prove the implications \ref{it:bfs3}$\Rightarrow$\ref{it:bfs1}\&\ref{it:bfssparse} through Rubio de Francia extrapolation. Early works on extrapolation in Banach function spaces date back to \cite{CGMP06}. The argument for our implication can already be found in \cite{CMP11}, and was stated explicitly in \cite[Theorem~10.1]{Cr17b}, albeit for function spaces satisfying the definition of \cite{BS88} (see \cite{LN23b} for a thorough discussion on the appropriate assumptions to make in the definition of a Banach function space). Further extrapolation results and comparisons can be found in \cite{CMM22, Ni23}.

The implication \ref{it:bfs3}$\Rightarrow$\ref{it:bfs4} only requires $M:X\to X_{\text{weak}}$; see \cite[Proposition~4.21]{Ni23}. Finally, a counterexample to \ref{it:bfs4}$\Rightarrow$\ref{it:bfs3} can be found in \cite[Theorem~5.3.4]{DHHR11}, with an explicit construction of an exponent function $\tfrac{3}{2}\leq p(\cdot)\leq 3$ for which the variable Lebesgue space $X=L^{p(\cdot)}(\R)$ satisfies the Muckenhoupt condition, but $M$ is not bounded on $X$.

\bigskip

Rutsky proves in \cite{Ru14, Ru19} that the conditions \ref{it:bfs1}-\ref{it:bfssparse} are also equivalent to the assertion
\begin{itemize}
    \item $T:X\to X$ for some Calder\'on-Zygmund operator $T$ defined on $L^2(\R^d)$ for which both $T$ and $T^\ast$ are non-degenerate,
\end{itemize}
as long as we have the additional assumption that $X$ is $r$-convex and $s$-concave for some $1<r<s<\infty$, i.e., there is a constant $C>0$ such that for all finite $\mc{F}\subseteq X$ we have
\[
\Big\|\Big(\sum_{f\in\mc{F}}|f|^r\Big)^{\frac{1}{r}}\Big\|_X\leq C\Big(\sum_{f\in\mc{F}}\|f\|_X^r\Big)^{\frac{1}{r}},\quad \Big(\sum_{f\in\mc{F}}\|f\|_X^s\Big)^{\frac{1}{s}}\leq C\Big\|\Big(\sum_{f\in\mc{F}}|f|^s\Big)^{\frac{1}{s}}\Big\|_X.
\]
We elaborate on these conditions in Subsection~\ref{subsec:dualityofM} below. Here, we say that an operator $T$ is non-degenerate if there is a constant $C>0$ such that for all $\ell>0$ there is an $x_{\ell}\in\R^d$ such that for all cubes $Q$ with $\ell(Q)=\ell$, all $0\leq f\in L^1(Q)$, and all $x\in Q+x_\ell\cup Q-x_\ell$, we have
\[
|Tf(x)|\geq C \langle f\rangle_Q.
\]
This is the non-degeneracy condition introduced by Stein in \cite[p.211]{St93}, and the example he had in mind are Calder\'on-Zygmund operators of convolution-type with kernel $K$ for which there is a direction $u\in\R^d$ with $|u|=1$ such that for all $t\in\R$  we have
\[
|K(tu)|\gtrsim|t|^{-d},
\]
see \cite[p.210]{St93}, or \cite{CM21} for an extension of this condition to the non-convolution case. This includes, e.g., any single Riesz transform. Thus, under these convexity and concavity assumptions, Theorem~\ref{thm:Bbfs}\ref{it:bfs2} can be weakened to only requiring the boundedness of a single one of the Riesz transforms rather than all of them.

We improve on Rutsky's result by showing that the equivalence holds without the convexity and concavity assumption, as long as $X$ satisfies the weaker order-continuity assumption. A quasi-Banach function space $X$ is called order-continuous if:
\begin{itemize}
\item For every sequence $(f_n)_{n\geq 1}$ in $X$ satisfying $0\leq f_n\downarrow 0$ a.e., we have $\|f_n\|_X\downarrow 0$.
\end{itemize}
Any quasi-Banach function space that is $s$-concave for some $s<\infty$ is order-continuous. We will show that we only need $T$ itself to be non-degenerate, and no further assumption needs to be made on $T^\ast$.
\begin{TheoremLetter}\label{thm:A}
Let $X$ be an order-continuous Banach function space with the Fatou property. Then any of the following statements are equivalent to any of the equivalent statements \ref{it:bfs1}-\ref{it:bfssparse} in Theorem~\ref{thm:Bbfs}:
\begin{enumerate}[(i)]\setcounter{enumi}{5}
    \item\label{it:bfs5} $T:X\to X$ for some non-degenerate linear operator $T$;
    \item\label{it:bfs6} there is a $C>0$ such that for all finite collections of cubes $\mc{F}$ and all sequences $(f_Q)_{Q\in\mc{F}}$ in $X$, we have
    \[
    \Big\|\Big(\sum_{Q\in\mc{F}}\langle f_Q\rangle_Q^2\ind_Q\Big)^{\frac{1}{2}}\Big\|_X\leq C\Big\|\Big(\sum_{Q\in\mc{F}}|f_Q|^2\Big)^{\frac{1}{2}}\Big\|_X.
    \]
\end{enumerate}
Moreover, in this case we have
\[
\sup_{\mc{S}\text{ is sparse}}\|A_{\mc{S}}\|_{X\to X}\lesssim_d\|T\|_{X\to X}^2,
\]
and $C$ can be chosen such that
\[
\sup_{\mc{S}\text{ is sparse}}\|A_{\mc{S}}\|_{X\to X}\lesssim_d C^2\leq\|M\|_{X\to X}\|M\|_{X'\to X'}.
\]
\end{TheoremLetter}
The equivalence of both \ref{it:bfs5} and \ref{it:bfs6} with the other statements are based on a deep result in the theory of Euclidean structures in operator theory from \cite{KLW23}, and we present it in Section~\ref{sec:rutsky}. To prove characterization \ref{it:bfs5} we also make use of the Grothendieck theorem on $\ell^2$-valued extensions.

Condition \ref{it:bfs6} is a square function estimate of the averaging operators in $X$. The fact that no kind of convexity assumption on $X$ needs to be made for this equivalence to hold is a notably non-trivial result. If instead of order-continuity we assume that $X$ is $r_0$-convex for some $r_0>1$, a much more elementary proof allows us to obtain the following additional characterization.
\begin{itemize}
    \item There is an $1<r\leq r_0$ and a $C>0$ such that for all finite collections of cubes $\mc{F}$ and all sequences $(f_Q)_{Q\in\mc{F}}$ in $X$, we have
    \[
    \Big\|\Big(\sum_{Q\in\mc{F}}\langle f_Q\rangle_Q^r\ind_Q\Big)^{\frac{1}{r}}\Big\|_X\leq C\Big\|\Big(\sum_{Q\in\mc{F}}|f_Q|^r\Big)^{\frac{1}{r}}\Big\|_X.
    \] 
\end{itemize}
This result is contained in \cite{Ru15}, and follows from Theorem~\ref{thm:Xconvexsingleexponentlinearizedm} below.

Square function estimates like \ref{it:bfs6} already appeared in the classical works on vector-valued extrapolation of Rubio de Francia \cite{Ru86}. It is closely linked to probabilistic notions in vector-valued analysis such as the UMD property of a Banach space. Indeed, under the assumption that $X$ is $s$-concave for some $s<\infty$, the square function estimate \ref{it:bfs6} can be related to random sums through the Khintchine-Maurey inequality \cite[Theorem~7.2.13]{HNVW17}. More precisely, this inequality says that if $\mc{F}$ is a finite collection of cubes and $(\varepsilon_Q)_{Q\in\mc{F}}$ is a Rademacher sequence over a probability space $(\Omega,\mathbb{P})$, then for all $0<p<\infty$ we have
\[
\Big\|\Big(\sum_{Q\in\mc{F}}|f_Q|^2\Big)^{\frac{1}{2}}\Big\|_X\eqsim_{X,p,s}\Big\|\sum_{Q\in\mc{F}}\varepsilon_Q f_Q\Big\|_{L^p(\Omega;X)}.
\]
Thus, to prove \ref{it:bfs6} when $X$ is $s$-concave, it suffices to show that for some (or, equivalently, all) $0<p,q<\infty$ we have
\[
\Big\|\sum_{Q\in\mc{F}}\varepsilon_Q \langle f_Q\rangle_Q\ind_Q\Big\|_{L^q(\Omega;X)}\lesssim \Big\|\sum_{Q\in\mc{F}}\varepsilon_Q f_Q\Big\|_{L^p(\Omega;X)}.
\]
This condition is known as the \emph{$R$-boundedness} in $X$ of the family of averaging operators, and this characterization can be found in \cite[Proposition~8.1.3]{HNVW17} in the case that $X$ is a Lebesgue (or Bochner) space.

\subsection{Generalizing the Muckenhoupt condition}

The counterexample we gave above to the question of whether or not the condition $X\in A$ characterizes the boundedness $M:X\to X$ was a variable Lebesgue space $X=L^{p(\cdot)}(\R^d)$ with an exponent function satisfying
\begin{equation}\label{eq:introvarlebesgue}
1<p_-\leq p_+<\infty,
\end{equation}
where
\[
p_-:=\essinf_{x\in\R^d} p(x),\quad p_+:=\esssup_{x\in\R^d} p(x).
\]
In \cite{Di05}, Diening was concerned with finding a stronger condition than $X\in A$ to characterize when exactly the maximal operator was bounded under the condition \eqref{eq:introvarlebesgue}. He showed that $M:L^{p(\cdot)}(\R^d)\to L^{p(\cdot)}(\R^d)$ precisely when $X=L^{p(\cdot)}(\R^d)$ satisfies the \emph{strong Muckenhoupt condition} $X\in A_{\text{strong}}$, i.e., there is a $C>0$ such that for every pairwise disjoint collection of cubes $\mc{P}$ and all $f\in X$ we have
\begin{equation}\label{eq:introstrongmuckenhoupt}
\Big\|\sum_{Q\in\mc{P}}\langle f\rangle_Q\ind_Q\Big\|_X\leq C\|f\|_X.
\end{equation}
We let $[X]_{A_{\text{strong}}}$ denote the smallest possible $C$ in the above inequality. This condition reduces back to the Muckenhoupt condition $X\in A$ if one only considers collections $\mc{P}$ consisting of single cubes. Moreover, as the operator inside the norm on the left-hand side of \eqref{eq:introstrongmuckenhoupt} is dominated by $Mf$, this condition is weaker than $M:X\to X$. Nevertheless, Diening shows that if $X=L^{p(\cdot)}(\R^d)$ and \eqref{eq:introvarlebesgue} holds, then $X\in A_{\text{strong}}$ is equivalent to $M:X\to X$.

This implication is not true in general: the space $X=L^1(\R^d)$ satisfies the strong Muckenhoupt condition, but not boundedness of $M$. However, the maximal operator is \emph{weakly} bounded on $L^1(\R^d)$. In general, one can show that the strong Muckenhoupt condition of $X$ implies the weak-type bound $M:X\to X_{\text{weak}}$, where
\[
\|f\|_{X_{\text{weak}}}:=\sup_{\lambda>0}\|\lambda\ind_{\{x\in\R^d:|f(x)|>\lambda\}}\|_X.
\]
This result was already shown for variable Lebesgue spaces in \cite{DHHR11}. Additionally, for spaces satisfying a certain structural property introduced in \cite{Be99}, one can show that $X\in A$ is equivalent to $X\in A_{\text{strong}}$. We write $X\in\mc{G}$ if there is a $C>0$ such that for every pairwise disjoint collection of cubes $\mc{P}$ and every $f\in X$, $g\in X'$ we have
\[
\sum_{Q\in\mc{P}}\|f\ind_Q\|_X\|g\ind_Q\|_{X'}\leq C\|f\|_X\|g\|_{X'}.
\]
The smallest possible constant is denoted by $[X]_{\mc{G}}$. This property holds for (weighted) Lebesgue spaces by H\"older's inequality. Further examples include (weighted) variable Lebesgue spaces with exponent functions satisfying global $\log$-H\"older continuity. We refer the reader to Section~\ref{sec:applications} for an overview. 

The following result captures the relations between the various notions of boundedness of the maximal operator and Muckenhoupt conditions.
\begin{TheoremLetter}\label{thm:introrelations}
Let $X$ be a Banach function space over $\R^d$ with the Fatou property. Consider the following statements:
\begin{enumerate}[(a)]
\item\label{it:relations1} $M:X\to X$;
\item\label{it:relations2} $X\in A_{\text{strong}}$;
\item\label{it:relations3} $M:X\to X_{\text{weak}}$;
\item\label{it:relations4} $X\in A$.
\end{enumerate}
Then \ref{it:relations1}$\Rightarrow$\ref{it:relations2}$\Rightarrow$\ref{it:relations3}$\Rightarrow$\ref{it:relations4} with
\[
[X]_A\leq\|M\|_{X\to X_{\text{weak}}}\lesssim_d [X]_{A_{\text{strong}}}\leq\|M\|_{X\to X}.
\]
Moreover, if $X\in\mc{G}$, then \ref{it:relations2}-\ref{it:relations4} are equivalent, with
\[
[X]_{A_{\text{strong}}}\leq[X]_{\mc{G}}[X]_A.
\]
\end{TheoremLetter}
This result is proved in Theorem~\ref{thm:ApropsHL} and Corollary~\ref{cor:aequivastrongweaktype} below. We also give a new characterization of the property $X\in\mc{G}$, generalizing the ideas of \cite[Section~7.3]{DHHR11}.
\begin{TheoremLetter}\label{thm:C}
Let $X$ be a Banach function space over $\R^d$ with the Fatou property. Then the following statements are equivalent.
\begin{enumerate}[(i)]
    \item $X\in\mc{G}$;
    \item\label{it:propG2} There are $C_2,\widetilde{C}_2>0$ such that for for all pairwise disjoint collections of cubes $\mc{P}$ and all $f\in X$ supported in $\bigcup_{Q\in\mc{P}} Q$ we have
    \[
    \widetilde{C}_2^{-1}\|f\|_X\leq \Big\|\sum_{Q\in\mc{P}}\frac{\|f\ind_Q\|_X}{\|\ind_Q\|_X}\ind_Q\Big\|_X\leq C_2\|f\|_X;
    \]
\end{enumerate}
Moreover, the optimal constants satisfy
\[
\max\{C_2,\widetilde{C}_2\}\leq [X]_{\mc{G}}\leq C_2\widetilde{C}_2.
\]
\end{TheoremLetter}
Property~\ref{it:propG2} can be interpreted as being able to reconstitute the norm of $f$ through only its local norm behavior. The proof of this result is inspired by the work of Kopaliani in \cite{Ko04}, and can be found as part of Theorem~\ref{thm:aequivastrong} below.

\subsection{Duality of the Hardy-Littlewood maximal operator}\label{subsec:dualityofM}

When $1<p<\infty$, the bounds $M:L^p_w(\R^d)\to L^p_w(\R^d)$ and $M:L^{p'}_{w^{-1}}(\R^d)\to L^{p'}_{w^{-1}}(\R^d)$ are equivalent. Thus, only one of them needs to be assumed in  Theorem~\ref{thm:Awlequiv}\ref{it:wl3}. This opens up the problem of characterizing the spaces satisfying this duality property:
\begin{equation}\tag{Q}\label{Q}
\begin{array}{c}
\text{\emph{Given a class of spaces $X$, when is it true that if $M:X\to X$, then}} \\ 
\text{\emph{also $M:X'\to X'$?}}
\end{array}
\end{equation}
Several characterizations of this problem are scattered throughout the literature, and we provide a list of some of them in Theorem~\ref{thm:Eintext} below.

Partial answers to \eqref{Q} do exist outside of weighted Lebesgue spaces: it was shown by Diening in \cite{Di05} that if $X=L^{p(\cdot)}(\R^d)$ is a variable Lebesgue space with exponent function $p(\cdot)$ satisfying \eqref{eq:introvarlebesgue}, then we have $M:X\to X$ if and only if $M:X'\to X'$. This is an immediate consequence of his characterization of the bound $M:X\to X$ with the condition $X\in A_{\text{strong}}$ combined with the fact that this latter property holds if and only if $X'\in A_{\text{strong}}$.

Generally, by the symmetry of the definition of the (strong) Muckenhoupt condition, for any Banach function space $X$ with the Fatou property, it is true that $X\in A$ or $X\in A_{\text{strong}}$ if and only if, respectively, $X'\in A$ or $X'\in A_{\text{strong}}$. Thus, any class of spaces $X$ for which $X\in A_{\text{strong}}$ implies $M:X\to X$ satisfies a version of Diening's duality result.

This duality result was extended by Lerner \cite{Le16b} to weighted variable Lebesgue spaces $X=L^{p(\cdot)}_w(\R^d)$ under the additional assumption that $w^{p(\cdot)}\in A_{\text{FW}}$, where we write $v\in A_{\text{FW}}$ when
\[
[v]_{\text{FW}}:=\sup_Q\frac{1}{v(Q)}\int_Q\!M(v\ind_Q)\,\mathrm{d}x<\infty.
\]
It is unclear whether this condition on the weight can be removed or not. Moreover, Lerner left open the question of whether or not the boundedness $M:X\to X$ is equivalent to $X\in A_{\text{strong}}$ in the weighted setting, see Subsection~\ref{subsec:variablelebesgue} below.

\bigskip

In general, to answer \eqref{Q}, we need to impose some kind of reflexivity condition on $X$. Looking at the above examples, one could expect that the condition that $X$ is $r$-convex and $s$-concave for some $1<r\leq s<\infty$ is a suitable choice. We note here that all quasi-Banach function spaces are $\infty$-concave by the ideal property, and a quasi-Banach function space is $1$-convex if and only if it is isomorphic to a Banach function space, i.e., there is an equivalent quasinorm on the space that satisfies the triangle inequality. A version of the definition of $r$-convexity and $s$-concavity in the setting of Banach lattices dates back to Krivine and Maurey \cite{Kr74, Ma74b}, see \cite{CT86,Ka84} for the case of quasi--Banach lattices, and many of the properties of these spaces are explored in these works. See also Subsection~\ref{subsec:convexconcave} below.

If a space is $r$-convex and $s$-concave for some $1<r\leq s<\infty$, then it is uniformly convex and, thus, superreflexive (see \cite{En72}). This class of spaces can be considered a natural extension of the class of weighted Lebesgue spaces in the following sense: by the Kolmogorov-Nagumo theorem, the only space that is both $r$-convex and $r$-concave is the space $L^r_w(\R^d)$ for some weight $w$ (see \cite{MN75} and \cite[Section~3]{BBS02})).

We list several examples of spaces with these convexity and concavity conditions. Precise definitions and further properties related to these spaces can be found in Section~\ref{sec:applications}. 
\begin{itemize}
    \item \emph{Weighted variable Lebesgue spaces}: If $X=L^{p(\cdot)}_w(\R^d)$, then $X$ is $p_-$-convex and $p_+$-concave. Thus, the spaces we are considering are those satisfying \eqref{eq:introvarlebesgue}.
    \item\emph{Musielak-Orlicz spaces}: The space $L^{\phi(\cdot)}(\R^d)$ is $s$-concave for some $s<\infty$ if and only if $\phi$ satisfies the $\Delta_2$ condition, see Theorem~\ref{thm:delta2conditionequivalence} below. Thus, we are considering the spaces for which both $\phi$ and $\phi^\ast$ satisfy the $\Delta_2$ condition.
    \item\emph{Weighted Morrey spaces}: If $1<p<q<\infty$, the weighted Morrey spaces $M_w^{p,q}(\R^d)$ is $p$-convex, but not $s$-concave for any $s<\infty$.
\end{itemize}

Our conjecture is as follows.
\begin{ConjectureLetter}\label{con:mduality1}
Let $1<r\leq s<\infty$ and let $X$ be an $r$-convex and $s$-concave Banach function space. Then $M:X\to X$ if and only if $M:X'\to X'$.
\end{ConjectureLetter}
Additionally, we can ask if in this case we have $M:X\to X$ if and only if $X\in A_{\text{strong}}$. If Conjecture~\ref{con:mduality1} is true, then this removes the necessity of the condition $w^{p(\cdot)}\in A_\infty$ in \cite{Le16b}. Not only would it simplify checking the conditions for Theorem~\ref{thm:Bbfs}, but it would also simplify the assumptions required on the space where both $M:X\to X$ and $M:X'\to X'$ are needed, such as in Rubio de Francia extrapolation results (see, e.g., \cite{Cr17b, CMM22, Ni23}), or as in the extrapolation of compactness theorem \cite[Theorem~A]{LN23a}.

This conjecture does not apply when dealing with spaces that are either not $r$-convex, or not $s$-concave. The former is the case, e.g., for Musielak-Orlicz spaces $L^{\phi(\cdot)}(\R^d)$ where $\phi^\ast$ does not satisfy the $\Delta_2$ condition, and the latter is the case, e.g., for weighted Morrey spaces $M_w^{p,q}(\R^d)$ with $1<p<q<\infty$. To be able to tackle spaces such as these, we formulate the following conjecture more closely related to the formulation of \eqref{Q}.
\begin{ConjectureLetter}\label{con:mduality2}
Let $1<s<\infty$, and let $X$ be an $s$-concave Banach function space. If $M:X\to X$, then $M:X'\to X'$.
\end{ConjectureLetter}
As it turns out, Conjecture~\ref{con:mduality2} is equivalent to Conjecture~\ref{con:mduality1}. Indeed, when $X$ satisfies the Fatou property, $X$ is $r$-convex or $s$-concave if and only if $X'$ is $s'$-convex or $r'$-concave, showing that the validity of Conjecture~\ref{con:mduality2} implies the validity of Conjecture~\ref{con:mduality1} (noting that any $r$-convex and $s$-concave space is reflexive, and, hence, satisfies the Fatou property by \cite[Corollary~3.16.]{LN23b}). As for the (harder) converse implication, this can be found in Corollary~\ref{cor:equivalentconjectures} below.

As exemplified by $X=L^\infty(\R^d)$, Conjecture~\ref{con:mduality2} is false if the concavity assumption is removed. Another example in the class of Morrey spaces is given in Example~\ref{ex:morreydoesnotimplyblock} below.

\bigskip

Finally, we present some partial results. As a special case of \cite[Theorem~4.22]{Ni23}, if an operator $T$ satisfies
\[
\|T\|_{L^1_w(\R^d)\to L^{1,\infty}_w(\R^d)}\leq\phi([w]_1)
\]
for some increasing function $\phi$ and all $w\in A_1$, then for all Banach function spaces $X$ with the Fatou property for which $M:X\to X$, we have
\[
\|T\|_{X'\to (X')_{\text{weak}}}\leq 2\phi(2\|M\|_{X\to X}).
\]
In particular, this implies that if $M:X\to X$, then we also have
\begin{equation}\label{eq:intromweakbound}
\|M\|_{X'\to (X')_{\text{weak}}}\lesssim_d\|M\|_{X\to X},
\end{equation}
and, uniformly in all sparse collections $\mc{S}$,
\begin{equation}\label{eq:introsparseweakbound}
\|A_{\mc{S}}\|_{X'\to (X')_{\text{weak}}}\lesssim_d(1+\log\|M\|_{X\to X})\|M\|_{X\to X}
\end{equation}
by the weak-type $A_1$ bound of \cite{DLR16}.

Since \eqref{eq:intromweakbound} already holds under the weaker assumption $X\in A_{\text{strong}}$ by Theorem~\ref{thm:introrelations}, one might wonder if \eqref{eq:introsparseweakbound} also holds under a weaker assumption. We show that this is indeed the case under the assumption that there is an $r>1$ for which $X^r\in A_{\text{strong}}$, where
\[
\|f\|_{X^r}:=\||f|^{\frac{1}{r}}\|_X^r.
\]
To see that this condition is weaker than $M:X\to X$, we note that it was shown in \cite{LO10} that if $M:X\to X$, then there is an $r>1$ (with $r'\eqsim_d\|M\|_{X\to X}$, see \cite[Theorem~2.34]{Ni23}) for which also $M:X^r\to X^r$. Thus, $M:X\to X$ implies that $X^r\in A_{\text{strong}}$. 

\begin{TheoremLetter}\label{prop:D}
Let $X$ be a quasi-Banach function space for which there is an $r>1$ such that $X^r\in A_{\text{strong}}$. Then $A_{\mc{S}}:X'\to (X')_{\text{weak}}$ uniformly in all sparse collections $\mc{S}$, with
\[
\sup_{\mc{S}}\|A_{\mc{S}}\|_{(X')\to (X')_{\text{weak}}}\lesssim_d r'(1+\log r')\|M\|_{X'\to (X')_{\text{weak}}}[X^r]^{\frac{1}{r}}_{A_{\text{strong}}}.
\]
\end{TheoremLetter}
The proof follows along the same lines of the one in \cite{DLR16}, but with a modification at the end of the proof inspired by the one used in \cite{Le20}. It can be found below as Theorem~\ref{lem:sparsetomweak}. 

A self-improvement condition like the one for the bound $M:X\to X$ does not exist for the condition $X\in A_{\text{strong}}$. This is exemplified by the fact that $L^1(\R^d)\in A_{\text{strong}}$, but $L^1(\R^d)^r=L^{\frac{1}{r}}(\R^d)\notin A_{\text{strong}}$ for any $r>1$. However, a self-improvement result was established by Diening for Musielak-Orlicz spaces in \cite[Theorem~5.7]{Di05}. Given a $\Phi$-function $\phi$ satisfying
\[
\lim_{t\downarrow 0}\tfrac{\phi(t)}{t}=0,\quad \lim_{t\to\infty}\tfrac{\phi(t)}{t}=\infty
\]
and for which both $\phi$ and $\phi^\ast$ satisfy the $\Delta_2$ condition, he proved that if $L^{\phi(\cdot)}(\R^d)\in A_{\text{strong}}$, then $\phi(\cdot)$ satisfies a certain reverse H\"older condition which generalizes the reverse H\"older condition for Muckenhoupt $A_p$ weights. For Muckenhoupt $A_p$ weights, the reverse H\"older condition allows one to deduce that if $1<p<\infty$ and $w\in A_p$, then there is an $1<r<p$ such that also $w^r\in A_{\frac{p}{r}}$. This can be equivalently written as
\[
L^p_w(\R^d)\in A\quad\Rightarrow\quad L^p_w(\R^d)^r\in A\text{ for some $1<r<p$.}
\]
This leaves us with the following question.
\begin{QuestionLetter}\label{que:G}
For $\phi$ as above, does $L^{\phi(\cdot)}(\R^d)\in A_{\text{strong}}$ imply the existence of an $r>1$ such that $L^{\phi(\cdot)}(\R^d)^r\in A_{\text{strong}}$?
\end{QuestionLetter}
If yes, then Theorem~\ref{prop:D} (applied to $X=L^{\phi^\ast(\cdot)}(\R^d)$) implies that for all Calder\'on-Zygmund operators $T$ we have
\[
T:L^{\phi(\cdot)}(\R^d)\to L^{\phi(\cdot)}(\R^d)_{\text{weak}}
\]
if $L^{\phi(\cdot)}(\R^d)\in A_{\text{strong}}$.

\subsection{Organization}
This paper is organized as follows.
\begin{itemize}
    \item In Section~\ref{sec:prelims} we define the required notions needed to understand the rest of the paper. It contains subsections on quasi-Banach function spaces, the notions of $r$-convexity and $s$-concavity, factorization of spaces, mixed-norm spaces, and weak-type spaces. In this context, we discuss vector-valued bounds related to a linearization of $M$.
    \item In Section~\ref{sec:muckenhoupt} we give an overview of several variants of the Muckenhoupt condition and their properties. In the first two subsections we introduce the conditions $X\in A$, $X\in A_{\text{strong}}$, and $X\in A_{\text{sparse}}$. We establish their basic properties, establish a hierarchy of the conditions, prove a self-improvement property for $A_{\text{sparse}}$, and prove Theorems \ref{thm:C} and \ref{prop:D}. In the next subsection we establish their relation to the boundedness of the Hardy-Littlewood maximal operator and prove Theorem~\ref{thm:introrelations}. In the final subsection we introduce Rutsky's notion of $A_p$-regularity and establish its connection to vector-valued estimates, resulting in a further characterization of the properties in Theorem~\ref{thm:Bbfs} under a convexity assumption. Finally, we prove Theorem~\ref{thm:A}.
    \item In Section~\ref{sec:dualityofm} we discuss the question \eqref{Q}. We establish several characterizations of this problem including through a self-improvement property related to $r$-convexity and $s$-concavity, and highlight its connection to other classical estimates such as the Fefferman-Stein inequality. Finally, we also provide a criterion for the failure of \eqref{Q}.
    \item In Section~\ref{sec:applications} we give an overview of the theory applied to (weighted) variable Lebesgue, Musielak-Orlicz spaces, Morrey, and block spaces. For (weighted) variable Lebesgue spaces and Musielak-Orlicz spaces we pose several conjectures, including on the the boundedness of the maximal operator and their relation to the classes $A$, $A_{\text{strong}}$, and $\mc{G}$. In weighted Morrey and block spaces we deduce several new weighted bounds of singular integrals, pose various further questions and conjectures, and give a new counterexample to \eqref{Q}.
    \item In Appendix~\ref{app:A} we prove two lemmas on sparse collections of cubes necessary for some of the proofs of the main results.
\end{itemize}

\section{Preliminaries}\label{sec:prelims}
\subsection{Notation}
We write $A\lesssim B$ or $B\gtrsim A$ when there is a constant $C>0$ such that $A\leq CB$. If the constant $C$ depends on certain parameters $\alpha_1,\alpha_2,\ldots$, then we sometimes write $A\lesssim_{\alpha_1,\alpha_2,\ldots} B$ to signify this. We write $A\eqsim B$ when $A\lesssim B$ and $B\lesssim A$. A similar convention holds for the notation $A\eqsim_{\alpha_1,\alpha_2,\ldots} B$.

Throughout this paper $d\geq 1$ is an integer signifying the dimension of $\R^d$. By a cube in $\R^d$ we mean a cube with sides parallel to the coordinate axes.

\subsection{Quasi-Banach function spaces}\label{subsec:qBFS}
Let $(\Omega,\mu)$ be a $\sigma$-finite measure space, and let $L^0(\Omega)$ denote the space of measurable functions on $\Omega$.
\begin{definition}
Let $X\subseteq L^0(\Omega)$ be a complete quasi-normed vector-space. We say that $X$ is a \emph{quasi-Banach function space over $\Omega$} if it satisfies:
\begin{itemize}
    \item \emph{Ideal property:} If $f\in X$ and $g\in L^0(\Omega)$ with $|g|\leq|f|$, then $g\in X$ with $\|g\|_X\leq\|f\|_X$;
    \item \emph{Saturation property:} For every $E\subseteq\Omega$ with $\mu(E)>0$ there is an $F\subseteq E$ for which $\mu(F)>0$ and $\ind_F\in X$.
\end{itemize}
\end{definition}
We denote the optimal constant $K\geq 1$ for which
\[
\|f+g\|_X\leq K(\|f\|_X+\|g\|_X)
\]
for all $f,g\in X$ by $K_X$. If $K_X=1$, we say that $X$ is a \emph{Banach function space over $\Omega$}.

By the ideal property we have $f\in X$ if and only if $|f|\in X$, with 
\[
\||f|\|_X=\|f\|_X.
\]
The saturation property is equivalent to various other properties, such as the existence of a \emph{weak order unit}, i.e., a function $u\in X$ satisfying $u>0$ a.e., or the property that the seminorm
\[
\|g\|_{X'}:=\sup_{\|f\|_X=1}\int_{\Omega}\!|fg|\,\mathrm{d}\mu
\]
is a norm, see \cite[Proposition~2.5]{LN23b}. The space
\[
X':=\{g\in L^0(\Omega):fg\in L^1(\Omega)\text{ for all $f\in X$}\}
\]
equipped with $\|\cdot\|_{X'}$ is called the \emph{K\"othe dual} of $X$. If $X$ is a Banach function space over $\Omega$, then so is $X'$. However, $X'$ might not satisfy the saturation property if $X$ is a \emph{quasi}-Banach function space over $\Omega$, e.g., $X=L^p(\R^d)$ with $0<p<1$ satisfies $X'=\{0\}$.

In the case that $\Omega=\R^d$, an often used sufficient condition for the saturation property is as follows.
\begin{proposition}\label{prop:ballbfs}
Let $X\subseteq L^0(\R^d)$ be a complete quasi-normed vector-space satisfying the ideal property. If $\ind_Q\in X$ for all cubes $Q\subseteq \R^d$, then $X$ satisfies the saturation property and, hence, is a quasi-Banach function space over $\R^d$.
\end{proposition}
\begin{proof}
By \cite[Proposition~2.5(ii)]{LN23b}, the saturation property is equivalent to the existence of a sequence of sets increasing to $\R^d$ whose indicator functions belong to $X$. Thus, we can pick any sequence of cubes that increases to $\R^d$ to satisfy this.
\end{proof}

We recall the following convergence properties from the introduction.
\begin{itemize}
    \item \emph{Fatou property:} For every sequence $(f_n)_{n\geq 1}$ in $X$ and $f$ in $L^0(\Omega)$ satisfying $0\leq f_n\uparrow f$ a.e. and $\sup_{n\geq 1}\|f_n\|_X<\infty$, we have $f\in X$ with $\|f\|_X=\sup_{n\geq 1}\|f_n\|_X$;
    \item \emph{Order-continuity:} For every sequence $(f_n)_{n\geq 1}$ in $X$ satisfying $0\leq f_n\downarrow 0$ a.e., we have $\|f_n\|_X\downarrow 0$.
\end{itemize}
The K\"othe dual of a quasi-Banach function space satisfies the Fatou property by the monotone convergence theorem. By the Lorentz-Luxemburg theorem, a Banach function space $X$ over $\Omega$ satisfies the Fatou property if and only if it is reflexive in the sense of K\"othe duality. This result can be found, e.g., in \cite[Theorem~71.1]{Za67}.
\begin{theorem*}[Lorentz-Luxemburg]
Let $X$ be a Banach function space over $\Omega$. Then $X$ satisfies the Fatou property if and only if $X''=X$.
\end{theorem*}
Order-continuity is equivalent to the canonical embedding $X'\hookrightarrow X^\ast$ being an isomorphism. In particular, $X$ is reflexive if and only if $X$ satisfies the Fatou property and $X$ and $X'$ are order-continuous, see \cite[Corollary~3.16]{LN23b}. Both the Fatou property and order-continuity are sufficient conditions for the equality
\[
\|f\|_X=\sup_{\|g\|_{X'}=1}\int_\Omega\!|fg|\,\mathrm{d}\mu,
\]
respectively due the Lorentz-Luxemburg theorem and the Hahn-Banach theorem.

Proofs and further details related to these spaces and their properties can be found in the survey \cite{LN23b} and the book \cite{Za67}.

\subsection{Convexity and concavity}\label{subsec:convexconcave}
Let $X$ be a quasi-Banach function space over $\Omega$ and let $0<r\leq s\leq\infty$. Then $X$ is called $r$-convex if there is a constant $M^{(r)}(X)\geq 1$ such that for all finite $\mc{F}\subseteq X$ we have
\[
\Big\|\Big(\sum_{f\in\mc{F}}|f|^r\Big)^{\frac{1}{r}}\Big\|_X\leq M^{(r)}(X)\Big(\sum_{f\in\mc{F}}\|f\|_X^r\Big)^{\frac{1}{r}}
\]
and $s$-concave if there is a constant $M_{(s)}(X)\geq 1$ such that for all finite $\mc{F}\subseteq X$ we have
\[
\Big(\sum_{f\in\mc{F}}\|f\|_X^s\Big)^{\frac{1}{s}}\leq M_{(s)}(X)\Big\|\Big(\sum_{f\in\mc{F}}|f|^s\Big)^{\frac{1}{s}}\Big\|_X,
\]
where the sums are replaced by a supremum when $r=\infty$ or $s=\infty$. Any quasi-Banach function space $X$ is $\infty$-concave with $M_{(\infty)}(X)=1$ by the ideal property, and $X$ is a Banach function space if and only if it $1$-convex with $M^{(1)}(X)=1$.

If $X$ is $r$-convex or $s$-concave, then $X'$ is respectively $r'$-concave or $s'$-convex, with 
\[
M_{(r')}(X')\leq M^{(r)}(X),\quad  M^{(s')}(X')\leq M_{(s)}(X).
\]

If a quasi-Banach function space $X$ is $r$-convex and $s$-concave, then there exists an equivalent quasi-norm on $X$ for which $M^{(r)}(X)=M_{(s)}(X)=1$, see \cite[Theorem~1.d.8]{LT79}.

If $X$ is $s$-concave for some $s<\infty$, then $X$ is order-continuous. If $X$ is $r$-convex and $s$-concave for some $1<r\leq s<\infty$ and $M^{(r)}(X)=M_{(s)}(X)=1$, then $X$ is (super)reflexive, and, hence, has the Fatou property, and $X$ and $X'$ are order-continuous.

For $0<p<\infty$, the $p$-concavification of a quasi-Bananach function space $X$ over $\Omega$ is defined as
\[
X^p:=\{f\in L^0(\Omega):|f|^{\frac{1}{p}}\in X\},\quad \|f\|_{X^p}:=\||f|^{\frac{1}{p}}\|_X^p.
\]
Then $X$ is $r$-convex or $s$-concave if and only if $X^p$ is respectively $\tfrac{r}{p}$-convex or $\tfrac{s}{p}$-concave with
\[
M^{\frac{r}{p}}(X^p)=M^{(r)}(X)^p,\quad M_{(\frac{s}{p})}(X^p)=M_{(s)}(X)^p.
\]

\subsection{Factorization}
Given two quasi-Banach function spaces $X$, $Y$ over $\Omega$, we define $X\cdot Y$ as the space of $f\in L^0(\Omega)$ for which there exist $0\leq h\in X$, $0\leq k\in Y$ such that $|f|\leq hk$, with the seminorm
\[
\|f\|_{X\cdot Y}:=\inf\|h\|_X\|k\|_Y,
\]
where the infimum is taken over all $0\leq h\in X$, $0\leq k\in Y$ satisfying $|f|\leq hk$. This is again a quasi-Banach function space over $\Omega$ with $K_{X\cdot Y}\leq 2K_XK_Y$.

For any $0<\theta<1$ and pair of quasi-Banach function spaces $X_0$, $X_1$ over $\Omega$, we call 
\[
X_0^{1-\theta}\cdot X_1^\theta
\]
the \emph{Calder\'on-Lozanovskii product} of $X_0$ and $X_1$ with parameter $\theta$. Using Young's inequality, one can show that
\[
K_{X_0^{1-\theta}\cdot X_1^\theta}\leq K_{X_0}^{1-\theta}K_{X_1}^\theta.
\]
In particular, if $X_0$ and $X_1$ are Banach function spaces, then so is $X_0^{1-\theta}\cdot X_1^\theta$. Moreover, in this case the Lozanovskii duality theorem (see \cite[Theorem~2]{Lo69}, \cite[Appendix~7]{CNS03}) states that
\[
(X_0^{1-\theta}\cdot X_1^\theta)'=[(X_0)']^{1-\theta}\cdot [(X_1)']^\theta.
\]
In particular, setting $X_0=L^\infty(\Omega)$, this implies that for any Banach function space $X$ over $\Omega$ we have
\begin{equation}\label{eq:lozdualitylebesgue}
(X^\theta)'=(X')^\theta\cdot L^1(\Omega)^{1-\theta}=(X')^\theta\cdot L^{\frac{1}{1-\theta}}(\Omega).
\end{equation}

\subsection{Mixed-norm spaces}
Let $X$ be a quasi-Banach function space over $\R^d$, and let $Y$ be a quasi-Banach function space over $\Omega$, where $(\Omega,\mu)$ is a $\sigma$-finite measure space. Equipping $\R^d\times\Omega$ with the product measure, we define $X[Y]$ as the space of those $f\in L^0(\R^d\times\Omega)$ for which $y\mapsto f(x,y)\in Y$ for a.e. $x\in\R^d$ and $x\mapsto\|y\mapsto f(x,y)\|_Y\in X$, with
\[
\|f\|_{X[Y]}:=\big\|\|f\|_Y\big\|_X:=\big\|x\mapsto \|y\mapsto f(x,y)\|_Y\big\|_X.
\]
This space is a quasi-Banach function space over $\R^d\times\Omega$ with $K_{X[Y]}\leq K_XK_Y$.

If $X$ and $Y$ are order-continuous or have the Fatou property, then $X[Y]$ also has these respective properties. If $X$ and $Y$ are Banach function spaces, it was shown by Bukhvalov \cite{Bu75} that
\[
X[Y]'=X'[Y'].
\]
Moreover, he showed in \cite{Bu87} (see also \cite{Ma89}) that if $X_0,X_1,Y_0,Y_1$ are Banach function spaces and $Y_0$, $Y_1$ have the Fatou property, then for all $0<\theta<1$
\begin{equation}\label{eq:mixednorminterpolation}
X_0[Y_0]^{1-\theta}\cdot X_1[Y_1]^\theta=(X_0^{1-\theta}\cdot X_1^\theta)[Y_0^{1-\theta}\cdot Y_1^\theta].
\end{equation}

The bounds of the Hardy-Littlewood maximal operator $M$ correspond to the bounds of a certain linear operator $\mc{M}$ defined on a mixed-norm space. Indeed, let $\mc{Q}$ denote the (countable) collection of cubes in $\R^d$ with rational corners and define the map
\[
\mathcal{M}((f_Q)_{Q\in\mc{Q}}):=(\langle f_Q\rangle_Q\ind_Q)_{Q\in\mc{Q}}.
\]
Then we have $M:X\to X$ if and only if $\mathcal{M}:X[\ell^\infty]\to X[\ell^\infty]$, where we have indexed $\ell^\infty$ over $\mc{Q}$. As a matter of fact, we have the following result.
\begin{proposition}\label{prop:mlinearization}
Let $X$ be a quasi-Banach function space over $\R^d$. Then the following are equivalent:
\begin{enumerate}[(i)]
\item\label{it:mlinearization1} $M:X\to X$;
\item\label{it:mlinearization2} $\mathcal{M}:X[\ell^\infty]\to X[\ell^\infty]$.
\end{enumerate}
In this case we have
\[
\|M\|_{X\to X}=\|\mathcal{M}\|_{X[\ell^\infty]\to X[\ell^\infty]}.
\]
Furthermore, if $X$ is a Banach function space with the Fatou property, then the following are equivalent:
\begin{enumerate}[(i)]\setcounter{enumi}{2}
\item\label{it:mlinearization3} $M:X'\to X'$;
\item\label{it:mlinearization4} $\mathcal{M}:X[\ell^1]\to X[\ell^1]$.
\end{enumerate}
In this case we have
\[
\|M\|_{X'\to X'}=\|\mc{M}\|_{X[\ell^1]\to X[\ell^1]}.
\]
\end{proposition}
\begin{proof}
For \ref{it:mlinearization1}$\Rightarrow$\ref{it:mlinearization2}, let $f=(f_Q)_{Q\in\mathcal{Q}}\in X[\ell^\infty]$. Then
\begin{align*}
\|\mathcal{M}((f_Q)_{Q\in\mathcal{Q}})\|_{X[\ell^\infty]}
&\leq\big\|M\big(\sup_{Q\in\mathcal{Q}}|f_Q|\big)\|_X\leq\|M\|_{X\to X}\big\|\sup_{Q\in\mathcal{Q}}|f_Q|\big\|_X\\
&=\|M\|_{X\to X}\|f\|_{X[\ell^\infty]}.
\end{align*}
proving that $\mathcal{M}: X[\ell^\infty]\to X[\ell^\infty]$ with $\|\mc{M}\|_{X[\ell^\infty]\to X[\ell^\infty]}\leq\|M\|_{X\to X}$. 

Conversely, noting that for $f\in X$ we have $(|f|)_{Q\in\mc{Q}}\in X[\ell^\infty]$ with $\|(|f|)_{Q\in\mc{Q}}\|_{X[\ell^\infty]}=\|f\|_X$, we have
\[
\|Mf\|_X=\|\mathcal{M}((|f|)_{Q\in\mc{Q}})\|_{X[\ell^\infty]}\leq\|\mc{M}\|_{X[\ell^\infty]\to X[\ell^\infty]}\|f\|_X,
\]
proving that \ref{it:mlinearization1}$\Leftrightarrow$\ref{it:mlinearization2} with the stated equality.

For \ref{it:mlinearization3}$\Rightarrow$\ref{it:mlinearization4}, let $f=(f_Q)_{Q\in\mathcal{Q}}\in X[\ell^1]$ and $g\in X'$ with $\|g\|_{X'}=1$. Then, since
\[
\int_{\R^d}\!\langle f_Q\rangle_{1,Q}\ind_Q|g|\,\mathrm{d}x=\int_{\R^d}\!|f_Q|\langle g\rangle_{1,Q}\ind_Q\,\mathrm{d}x
\]
for all $Q\in\mathcal{Q}$, we have
\begin{align*}
\int_{\R^d}\!|g|\sum_{Q\in\mc{Q}}|\langle f_Q\rangle_{Q}|\ind_Q\,\mathrm{d}x&\leq\int_{\R^d}\!\sum_{Q\in\mathcal{Q}}|f_Q|\langle g\rangle_{1,Q}\ind_Q\,\mathrm{d}x
\leq\int_{\R^d}\!\|f\|_{\ell^1}Mg\,\mathrm{d}x\\
&\leq\|M\|_{X'\to X'}\|f\|_{X[\ell^1]}.
\end{align*}
Thus, $\|\mc{M}f\|_{\ell^1}\in X$ with 
\[
\|\mc{M}f\|_{X[\ell^1]}\leq\|M\|_{X'\to X'}\|f\|_{X[\ell^1]},
\]
as desired.

Finally, for \ref{it:mlinearization4}$\Rightarrow$\ref{it:mlinearization3}, we use the fact that $X'[\ell^\infty]=X[\ell^1]'$ so that, for $g\in X'$,
\begin{align*}
\|Mg\|_{X'} &= \|(\langle g\rangle_{1,Q}\ind_Q)_{Q\in\mathcal{Q}}\|_{X'[\ell^\infty]}\\
&=\sup_{\|f\|_{X[\ell^1]}=1} \int_{\R^d}\sum_{Q\in\mathcal{Q}}\!|f_Q|\langle g\rangle_{1,Q}\ind_Q\,\mathrm{d}x\\
&=\sup_{\|f\|_{X[\ell^1]}= 1}\int_{\R^d}\!\sum_{Q\in\mathcal{Q}} \langle f_Q\rangle_{1,Q}\ind_Q|g|\,\mathrm{d}x\\
&\leq \sup_{\|f\|_{X[\ell^1]}= 1} \|\mathcal{M}(|f|)\|_{X[\ell^1]}\|g\|_{X'}   \\
&\leq \|\mathcal{M}\|_{X[\ell^1]\to X[\ell^1]} \,\nrm{g}_{X'}.
\end{align*}
This proves the result.
\end{proof}
By the interpolation formula \eqref{eq:mixednorminterpolation} and H\"older's inequality, this yields the following corollary:
\begin{corollary}\label{cor:mlinearization}
Let $X$ be a Banach function space over $\R^d$ with the Fatou property. Then the following are equivalent:
\begin{enumerate}[(i)]
\item $M:X\to X$, $M:X'\to X'$
\item $\mc{M}:X[\ell^r]\to X[\ell^r]$ for all $1\leq r\leq\infty$.
\end{enumerate}
Moreover, we have
\[
\|\mc{M}\|_{X[\ell^r]\to X[\ell^r]}\leq\|M\|_{X\to X}^{\frac{1}{r'}}\|M\|_{X'\to X'}^{\frac{1}{r}}.
\]
\end{corollary}
\begin{proof}
Let $f\in X[\ell^r]$. By \eqref{eq:mixednorminterpolation} we have
\begin{equation}\label{eq:mlinearizationinterpolation1}
X[\ell^r]=X[\ell^\infty]^{\frac{1}{r'}}\cdot X[\ell^1]^{\frac{1}{r}},
\end{equation}
showing that we can pick $0\leq g\in X[\ell^\infty]$, $0\leq h\in X[\ell^1]$ for which $|f_Q|\leq g_Q^{\frac{1}{r'}}h_Q^{\frac{1}{r}}$ for all $Q\in\mc{Q}$. Thus, by H\"older's inequality,
\[
\Big(\sum_{Q\in\mc{Q}}|\langle f_Q\rangle_Q|^r\ind_Q\Big)^{\frac{1}{r}}\leq \Big(\sum_{Q\in\mc{Q}}\langle g_Q\rangle_Q^{\frac{r}{r'}}\langle h_Q\rangle_Q\ind_Q\Big)^{\frac{1}{r}}\leq \|\mc{M}g\|_{\ell^\infty}^{\frac{1}{r'}}\|\mc{M}h\|_{\ell^1}^{\frac{1}{r}}.
\]
Hence, by \eqref{eq:mlinearizationinterpolation1},
\begin{align*}
\|\mc{M}f\|_{X[\ell^r]}
&\leq\|\mc{M}g\|_{X[\ell^\infty]}^{\frac{1}{r'}}\|\mc{M}h\|_{X[\ell^1]}^{\frac{1}{r}}\\
&\leq\|\mc{M}\|_{X[\ell^\infty]\to X[\ell^\infty]}^{\frac{1}{r'}}\|\mc{M}\|_{X[\ell^1]\to X[\ell^1]}^{\frac{1}{r}}\|g\|_{X[\ell^\infty]}^{\frac{1}{r'}}\|h\|_{X[\ell^1]}^{\frac{1}{r}}.
\end{align*}
Taking an infimum over all possible $|f|\leq g^{\frac{1}{r'}}h^{\frac{1}{r}}$, the result follows from Proposition~\ref{prop:mlinearization}.
\end{proof}

\subsection{Weak-type spaces}
In the same way the weak-type space $L^{p,\infty}(\Omega)$ can be obtained from $L^p(\Omega)$, one can obtain a weak-type space $X_{\text{weak}}$ from a quasi-Banach function space $X$.
\begin{definition}
Let $X$ be a quasi-Banach function space over $\Omega$. Then the \emph{weak-type space} $X_{\text{weak}}$ is defined as the space of those $f\in L^0(\Omega)$ for which $\ind_{\{x\in\Omega:|f(x)|>\lambda\}}\in X$ for all $\lambda>0$ and
\[
\|f\|_{X_{\text{weak}}}:=\sup_{\lambda>0}\|\lambda\ind_{\{x\in\Omega:|f(x)|>\lambda\}}\|_X<\infty.
\]
\end{definition}
Since 
\[
\ind_{\{x\in\Omega:|f(x)+g(x)|>\lambda\}}\leq\ind_{\{x\in\Omega:|f(x)|>\frac{\lambda}{2}\}}+\ind_{\{x\in\Omega:|g(x)|>\frac{\lambda}{2}\}},
\]
we have that $X_{\text{weak}}$ is a quasi-normed space with $K_{X_{\text{weak}}}\leq 2K_X$. As a matter of fact, it is also a quasi-Banach function space.
\begin{proposition}\label{prop:weakqbfsdefinition}
Let $X$ be a quasi-Banach function space over $\Omega$. Then $X_{\text{weak}}$ is also a quasi-Banach function space over $\Omega$ for which $X\hookrightarrow X_{\text{weak}}$ with
\[
\|f\|_{X_{\text{weak}}}\leq\|f\|_X
\]
for all $f\in X$. If $X$ satisfies the Fatou property, then so does $X_{\text{weak}}$. Moreover, for all measurable sets $E\subseteq\Omega$ we have $\ind_E\in X$ if and only if $\ind_E\in X_{\text{weak}}$ with
\[
\|\ind_E\|_{X_{\text{weak}}}=\|\ind_E\|_X.
\]
\end{proposition}
\begin{proof}
The first assertion can be found in \cite[Proposition~4.20]{Ni23}. For the final assertion, note that
\[
\ind_{\{x\in\Omega:\ind_E(x)>\lambda\}}=\begin{cases}
\ind_E & \text{if $\lambda\in(0,1)$;}\\
0 & \text{if $\lambda\geq 1$.}
\end{cases}
\]
Thus, if $\ind_E\in X_{\text{weak}}$, then $\ind_E\in X$ and
\[
\|\ind_E\|_{X_{\text{weak}}}=\sup_{\lambda\in(0,1)}\lambda\|\ind_E\|_X=\|\ind_E\|_X,
\]
as desired.
\end{proof}

\section{A zoo of Muckenhoupt conditions}\label{sec:muckenhoupt}
\subsection{The Muckenhoupt condition}
The Muckenhoupt condition in a quasi-Banach function space $X$ over $\R^d$ is closely related to the boundedness of averaging operators over cubes. Indeed, for a cube $Q$, the operator
\[
T_Qf:=\langle f\rangle_{1,Q}\ind_Q,
\]
where $\langle f\rangle_{1,Q}:=\left(\frac{1}{|Q|}\int_Q\!|f|\,\mathrm{d}x\right)$, satisfies $T_Q:X\to X$ if and only if $\ind_Q\in X$ and $\ind_Q\in X'$. As a matter of fact, we have the following result.
\begin{proposition}\label{prop:averagingoperatornorm}
Let $X$ be a quasi-Banach function space over $\R^d$ and let $Q$ be a cube. Then $T_Q:X\to X$ if and only if $\ind_Q\in X$ and $\ind_Q\in X'$, with
\[
\|T_Q\|_{X\to X}=|Q|^{-1}\|\ind_Q\|_X\|\ind_Q\|_{X'}.
\]
\end{proposition}
This equivalence is already contained in, e.g., \cite{Be99}, but we include a proof here.
\begin{proof}[Proof of Proposition~\ref{prop:averagingoperatornorm}]
First suppose that $\ind_Q\in X$ and $\ind_Q\in X'$. Then
\begin{align*}
\|T_Qf\|_X&=\langle f\rangle_{1,Q}\|\ind_Q\|_X=\big(|Q|^{-1}\|\ind_Q\|_X\|\ind_Q\|_{X'}\big)\|\ind_Q\|_{X'}^{-1}\|f\ind_Q\|_{L^1(\R^d)}\\
&\leq\big(|Q|^{-1}\|\ind_Q\|_X\|\ind_Q\|_{X'}\big)\|f\|_X,
\end{align*}
proving that $T_Q:X\to X$ with
\begin{equation}\label{eq:AX1}
\|T_Q\|_{X\to X}\leq|Q|^{-1}\|\ind_Q\|_X\|\ind_Q\|_{X'}.
\end{equation}
For the converse, let $0<u\in X$ be a weak order unit. Fix a cube $Q$. Then we have
\[
\langle u\rangle_{1,Q}\ind_Q= T_Q u\in X.
\]
Since $\langle u\rangle_{1,Q}>0$, it follows that $\ind_Q\in X$. Next, we note for all $f\in X$ we have
\begin{align*}
\|f\ind_Q\|_{L^1(\R^d)}
&=\langle f\rangle_{1,Q}|Q|=\frac{|Q|}{\|\ind_Q\|_X}\|\langle f\rangle_{1,Q}\ind_Q\|_X
\leq\|T_Q\|_{X\to X}\frac{|Q|}{\|\ind_Q\|_X}\|f\|_X.
\end{align*}
Hence, $\ind_Q\in X'$ with
\[
\|\ind_Q\|_{X'}\leq\|T_Q\|_{X\to X}\frac{|Q|}{\|\ind_Q\|_X},
\]
proving the result. Moreover, combining this last inequality with \eqref{eq:AX1} proves the desired norm equality.
\end{proof}
Note that we can just as well have defined the averaging operators with respect to the linearized averages
\[
\widetilde{T}_Qf=\langle f\rangle_Q\ind_Q,\quad \langle f\rangle_Q=\frac{1}{|Q|}\int_Q\!f\,\mathrm{d}x.
\]
Indeed, this follows from the observation that $|\widetilde{T}_Qf|\leq T_Qf=\widetilde{T}_Q(|f|)$, and the fact that $\||f|\|_X=\|f\|_X$ by the ideal property.

Proposition~\ref{prop:averagingoperatornorm} motivates the following definition of the Muckenhoupt condition:
\begin{definition}
Let $X$ be a quasi-Banach function space over $\R^d$. We say that $X$ satisfies the \emph{Muckenhoupt condition} $A$ and write $X\in A$ when for all cubes $Q$ we have $\ind_Q\in X$, $\ind_Q\in X'$, and
\[
[X]_{A}:=\sup_Q|Q|^{-1}\|\ind_Q\|_X\|\ind_Q\|_{X'}<\infty.
\]
\end{definition}
Since
\[
|Q|=\|\ind _Q\ind_Q\|_{L^1(\R^d)}\leq\|\ind_Q\|_X\|\ind_Q\|_{X'}
\]
for any cube $Q$, we have $[X]_{A}\geq 1$. Moreover, since the Muckenhoupt condition requires that $\ind_Q\in X$ and $\ind_Q\in X'$ for all cubes $Q$, any space $X\in A$ has the property that $X'$ is saturated and, hence, is a Banach function space over $\R^d$ by Proposition~\ref{prop:ballbfs}.

By Proposition~\ref{prop:averagingoperatornorm}, the condition $X\in A$ is characterized by the uniform (weak) boundedness of the averaging operators $T_Q$.
\begin{proposition}\label{prop:AXcondition}
Let $X$ be a quasi-Banach function space over $\R^d$. Then the following are equivalent:
\begin{enumerate}[(i)]
\item\label{it:AX1} $X\in A$;
\item\label{it:AX2} $T_Q:X\to X$ uniformly over all cubes $Q$;
\item\label{it:AX3} $T_Q:X\to X_{\text{weak}}$ uniformly over all cubes $Q$.
\end{enumerate}
Moreover, in this case we have
\begin{equation}\label{eq:AX}
[X]_{A}=\sup_{Q}\|T_Q\|_{X\to X}=\sup_{Q}\|T_Q\|_{X\to X_{\text{weak}}}.
\end{equation}
\end{proposition}
\begin{proof}
The equivalence \ref{it:AX1}$\Leftrightarrow$\ref{it:AX2} and the first equality in \eqref{eq:AX} follow from the operator norm equality in Proposition~\ref{prop:averagingoperatornorm}. The equivalence \ref{it:AX2}$\Leftrightarrow$\ref{it:AX3} and the second equality in \eqref{eq:AX} follows from the fact that by Proposition~\ref{prop:weakqbfsdefinition} we have
\[
\|T_Q f\|_{X}=\langle f\rangle_{1,Q}\|\ind_Q\|_X=\langle f\rangle_{1,Q}\|\ind_Q\|_{X_{\text{weak}}}=\|T_Q f\|_{X_{\text{weak}}}
\]
for all cubes $Q$ and $f\in X$. This proves the result.
\end{proof}

Note that the equivalence \ref{it:AX1}$\Leftrightarrow$\ref{it:AX2} was already shown shown in \cite{Be99}.

We say that a (sub)linear operator $T$ is non-degenerate if there is a constant $C>0$ such that for all $\ell>0$ there is an $x_{\ell}\in\R^d$ such that for all cubes $Q$ with $\ell(Q)=\ell$, all $0\leq f\in L^1(Q)$, and all $x\in Q+x_\ell\cup Q-x_\ell$, we have
\[
|Tf(x)|\geq C \Big(\frac{1}{|Q|}\int_Q\!f\,\mathrm{d}x\Big).
\]
Note that this includes, e.g., the Hardy-Littlewood maximal operator $M$ with $x_\ell=0$, $C=1$, but also linear operators such as the Hilbert transform, or any of the Riesz transforms. The proof of the following result follows the lines of the proof of \cite[p.211]{St93}.
\begin{proposition}\label{prop:nondegimpliesmuckenhoupt}
Let $X$ be a quasi-Banach function space over $\R^d$ with the Fatou property and let $T$ be a non-degenerate operator for which $T:X\to X_{\text{weak}}$. Then $X\in A$ with
\[
[X]_A\leq C^{-2}\|T\|^2_{X\to X_{\text{weak}}},
\]
where $C$ is the constant in the definition of non-degeneracy.
\end{proposition}
\begin{proof}
Let $f\in X\cap L^\infty_c(\R^d)$, let $Q$ be a cube, and let $0<\lambda<C\langle f\rangle_{1,Q}$. Then for 
\[
P\in\{Q+x_{\ell(Q)},Q-x_{\ell(Q)}\}
\]
we have
\[
P\subseteq \{x\in\R^d:|T(|f|\ind_Q)(x)|>\lambda\},
\]
which, by the ideal property of $X$, means that $\ind_{P}\in X$ with
\[
\lambda\|\ind_{P}\|_X\leq\|T\|_{X\to X_{\text{weak}}}\|f\ind_Q\|_X.
\]
Taking a supremum over $0<\lambda<C\langle f\rangle_{1,Q}$, we conclude that
\begin{equation}\label{eq:nondegimpliesmuckenhoupt1}
\|\langle f\rangle_{1,Q}\ind_{P}\|_X\leq C^{-1}\|T\|_{X\to X_{\text{weak}}}\|f\ind_Q\|_X.
\end{equation}
Note that since $\ell(P)=\ell(Q)$, applying this with $Q$ replaced by $Q_+:=Q+x_{\ell(Q)}$ and $f=\ind_{Q_+}\in X$ proves that
\[
\|\ind_Q\|_X\leq C^{-1}\|T\|_{X\to X_{\text{weak}}}\|\ind_{Q_+}\|_X,
\]
where we used that $Q_+-x_{\ell(Q)}=Q$. Thus, by \eqref{eq:nondegimpliesmuckenhoupt1},
\[
\|T_Q f\|_X\leq C^{-1}\|T\|_{X\to X_{\text{weak}}}\|\langle f\rangle_{1,Q}\ind_{Q_+}\|_X\leq C^{-2}\|T\|^2_{X\to X_{\text{weak}}}\|f\ind_Q\|_X.
\]
As for any $f\in X$ we can pick a sequence $f_n\in X\cap L^\infty_c(\R^d)$ with $0\leq f_n\uparrow |f|$, we can extend this bound to all $f\in X$ by the Fatou property of $X$. By Proposition~\ref{prop:AXcondition}, we conclude that $X\in A$ with
\[
[X]_A\leq C^{-2}\|T\|^2_{X\to X_{\text{weak}}},
\]
as desired.
\end{proof}

\subsection{Generalized Muckenhoupt conditions}
Two other more general averaging operators related to the boundedness of $M$ and the Muckenhoupt condition have appeared in the literature. Given a collection of cubes $\mc{P}$ and $f\in L^1_{\text{loc}}(\R^d)$, we define
\[
A_{\mc{P}}f:=\sum_{Q\in\mc{P}}T_Qf=\sum_{Q\in\mc{P}}\langle f\rangle_{1,Q}\ind_Q.
\]
For $0<\eta<1$ we say that a collection of cubes $\mc{S}$ is \emph{$\eta$-sparse} if there exists a pairwise disjoint collection $(E_Q)_{Q\in\mc{S}}$ of subsets $E_Q\subseteq Q$ for which $|E_Q|\geq\eta|Q|$. We will call $\mc{S}$ \emph{sparse} if it is $\tfrac{1}{2}$-sparse.

\begin{definition}
Let $X$ be a quasi-Banach function space over $\R^d$.
\begin{itemize}
    \item We say that $X$ satisfies the \emph{strong Muckenhoupt condition} and write $X\in A_{\text{strong}}$ if there is a $C>0$ such that for every pairwise disjoint collection of cubes $\mathcal{P}$ and all $f\in X$ we have
\[
\|A_\mathcal{P}f\|_X\leq C\|f\|_X.
\]
The smallest possible constant $C$ is denoted by $[X]_{A_{\text{strong}}}$.
\item We say that $X$ satisfies the \emph{sparse Muckenhoupt condition} and write $X\in A_{\text{sparse}}$ if there is a $C>0$ such that for every sparse collection of cubes $\mathcal{S}$ and all $f\in X$ we have
\[
\|A_\mathcal{S}f\|_X\leq C\|f\|_X.
\]
The smallest possible constant $C$ is denoted by $[X]_{A_{\text{sparse}}}$.
\end{itemize}
\end{definition}
Since for pairwise disjoint collections $\mc{P}$ and $0<\theta<1$ we have 
\[
(A_{\mc{P}}f)^{\frac{1}{\theta}}=\sum_{Q\in\mc{P}}\langle f\rangle_{1,Q}^{\frac{1}{\theta}}\ind_Q\leq \sum_{Q\in\mc{P}}\langle |f|^{\frac{1}{\theta}}\rangle_{1,Q}\ind_Q
\]
by H\"older's inequality, we conclude that
\[
[X^\theta]_A\leq[X]_A^\theta,\quad [X^\theta]_{A_{\text{strong}}}\leq[X]^\theta_{\text{strong}}.
\]
By combining Theorem~\ref{cor:bfsdualweak} and Theorem~\ref{thm:ApropsHL} below, a variant of this result also holds for the condition $X\in A_{\text{sparse}}$. More precisely, if $X\in A_{\text{sparse}}$ and $0<\theta<1$, then $X^\theta\in A_{\text{sparse}}$ with
\[
[X^\theta]_{A_{\text{sparse}}}\lesssim_d\frac{1}{\theta}\frac{1}{1-\theta}[X]_{A_{\text{sparse}}}^{1+\theta}.
\]

The strong Muckenhoupt condition was originally introduced by Diening in the context of Musielak-Orlicz spaces -- in particular for variable Lebesgue spaces -- in \cite{Di05}. The sparse Muckenhoupt condition was used in the recent paper \cite{LN23a} by Lorist and the author to prove a dual self-improvement result related to the Hardy-Littlewood maximal operator.

Since for each cube $Q$ the collection $\mathcal{P}=\{Q\}$ is pairwise disjoint with $A_\mathcal{F}=T_Q$, it follows from Proposition~\ref{prop:AXcondition} that if $X\in A_{\text{strong}}$, then $X\in A$. Moreover, since any pairwise disjoint collection of cubes is sparse (with $E_Q=Q$), any $X\in A_{\text{sparse}}$ also satisfies $X\in A_{\text{strong}}$. Thus, we have the chain of implications
\[
X\in A_{\text{sparse}} \Rightarrow X\in A_{\text{strong}} \Rightarrow X\in A,
\]
with
\[
[X]_{A}\leq[X]_{A_{\text{strong}}}\leq[X]_{A_{\text{sparse}}}.
\]

In the definition of $A_{\text{sparse}}$, it suffices to check the result for sparse collections contained in a dyadic grid as follows from the $3^d$-lattice theorem combined with the fact that the sparsity constant $\tfrac{1}{2}$ could be replaced by any other $0<\eta<1$. For the definition of a dyadic grid and related properties we refer the reader to \cite{LN15}. This latter statement about the sparsity constant follows from the following result.
\begin{proposition}\label{prop:sparsedomsparse}
Let $\mc{D}$ be a dyadic grid, let $0<\eta<1$, and let $\mc{S}\subseteq\mc{D}$ be a finite $\eta$-sparse collection. Then for all $0<\nu<1$ and all $f\in L^1_{\text{loc}}(\R^d)$ there exists a $\nu$-sparse collection $\mc{E}\subseteq\mc{D}$ such that
\[
A_{\mc{S}}f\lesssim_d\frac{1}{\eta(1-\nu)} A_{\mc{E}}f.
\]
\end{proposition}
We point out that this is certainly known, but as we could not find a direct proof in the literature, we have added a proof as part of Lemma~\ref{lem:sparsedomofsparse} below for completeness.

We have the following duality relations of the generalized Muckenhoupt conditions.
\begin{proposition}\label{prop:bfspropertiesduality}
Let $X$ be a quasi-Banach function space over $\R^d$. If $X\in A$, then $X'$ has the saturation property and, hence, is a Banach function space. Moreover, we have $X'\in A$ with
\[
[X']_{A}\leq[X]_{A}.
\]
If $X$ is a Banach function space over $\R^d$ with the Fatou property, then $X\in A$ if and only if $X'\in A$ with
\[
[X']_{A}=[X]_{A}.
\]
Analogous statements hold if $A$ is replaced by $A_{\text{strong}}$ or $A_{\text{sparse}}$.
\end{proposition}
\begin{proof}
The saturation property follows from Proposition~\ref{prop:ballbfs}. For the next assertion, let $f\in X$, $g\in X'$, and $\mathcal{P}$ a collection of cubes. Then
\[
\|f(A_\mathcal{P}g)\|_{L^1(\R^d)}=\sum_{Q\in\mathcal{P}}\langle f\rangle_{1,Q}\langle g\rangle_{1,Q}|Q|=\|(A_\mathcal{P}f)g\|_{L^1(\R^d)}\leq\|A_{\mathcal{P}}\|_{X\to X}\|f\|_X\|g\|_{X'}
\]
so that
\begin{equation}\label{eq:dualaveragingbound1}
\|A_{\mathcal{P}}\|_{X'\to X'}=\sup_{\substack{\|f\|_X=1\\ \|g\|_{X'}=1}}\|f(A_\mathcal{P}g)\|_{L^1(\R^d)}\leq\|A_{\mathcal{P}}\|_{X\to X}.
\end{equation}
From this we conclude the first assertion. The final assertion follows from the fact that the inequality \eqref{eq:dualaveragingbound1} is now an equality by the Lorentz-Luxemburg theorem.
\end{proof}

The strong Muckenhoupt condition $A_{\text{strong}}$ appears to be a strictly stronger condition than the Muckenhoupt condition $A$ outside of weighted Lebesgue spaces. An example where this is the case is given in \cite[Theorem~5.3.4]{DHHR11}, where an exponent function $p(\cdot)$ is given for which the variable Lebesgue space $X=L^{p(\cdot)}(\R)$ satisfies $X\in A$, but not $X\in A_{\text{strong}}$.

\begin{definition}
Let $X$ be a quasi-Banach function space over $\R^d$. We say that $X$ satisfies the structural property $\mc{G}$ and write $X\in\mc{G}$ if there is a constant $C>0$ such that for all pairwise disjoint collections of cubes $\mc{P}$ and all $f\in X$, $g\in X'$ we have
\begin{equation}\label{eq:propertygb}
\sum_{Q\in\mc{P}}\|f\ind_Q\|_X\|g\ind_Q\|_{X'}\leq C\|f\|_X\|g\|_{X'}.
\end{equation}
We define $[X]_{\mc{G}}$ as the smallest possible constant $C$ in the above inequality.
\end{definition}
It is shown in \cite{Be99} that if $X$ is a Banach function space over $\R^d$ with the Fatou property satisfying $X\in\mc{G}$, then $X\in A$ and $X\in A_{\text{strong}}$ are equivalent. Moreover, in this case we have
\[
[X]_{A_{\text{strong}}}\leq [X]_{\mc{G}}[X]_A.
\]
Using H\"older's inequality, \eqref{eq:propertygb} holds if $X$ is both $r$-convex and $r$-concave for some $r\geq 1$, with $C=M^{(r)}(X)M_{(r)}(X)$. However, this example is very limited, as the Kolmogorov-Nagumo theorem asserts that the only possible space that is both $r$-convex and $r$-concave is $X=L^r_w(\R^d)$ for some weight $w$. 

In the setting of $\R_+=[0,\infty)$ rather than $\R^d$, it was shown in \cite{Ko04} that if \eqref{eq:propertygb} holds with the collection $\mc{P}$ replaced by any collection of pairwise disjoint sets, then we must be in the above situation where $X=L^r_w(\R_+)$ for some weight $w$. However, non-trivial examples of when \eqref{eq:propertygb} holds were also provided. Indeed, under the condition \eqref{eq:propertygb} specialized to the space $[0,1]$ rather than $\R^d$, this was shown to hold for the variable Lebesgue space $X=L^{p(\cdot)}([0,1])$ with exponents $p(\cdot)$ satisfying
\[
|p(x)-p(y)||\log|x-y||\leq C.
\]
An explicit example of a space in this setting that does not satisfy \eqref{eq:propertygb} is also constructed. We discuss further examples in Section~\ref{subsec:variablelebesgue}. The proof uses a Banach function space analogue of the average $\langle f\rangle_{1,Q}$. Indeed, defining
\[
\|f\|_{X_Q}:=\frac{\|f\ind_Q\|_X}{\|\ind_Q\|_X}
\]
for a cube $Q$, we note that for $X=L^1(\R^d)$ we have $\|f\|_{X_Q}=\langle f\rangle_{1,Q}$. In general, there is a constant $C>0$ such that for all cubes $Q$ and all $f\in X$ we have
\[
\langle f\rangle_{1,Q}\leq C\|f\|_{X_Q}
\]
if and only if $X\in A$, where the smallest possible $C$ is given by $[X]_A$. Indeed, this follows from the observation that the above inequality is just a rewrite of the boundedness of the averaging operator $T_Q$, combined with Proposition~\ref{prop:AXcondition}.

We have the following sufficient conditions for the equivalence of $X\in A$ and $X\in A_{\text{strong}}$, generalizing the ideas of \cite[Section~7.3]{DHHR11}.
\begin{theorem}\label{thm:aequivastrong}
Let $X$ be a Banach function space over $\R^d$ with the Fatou property. Consider the following statements:
\begin{enumerate}[(a)]
    \item\label{it:aequivastrong1} $X\in\mc{G}$;
    \item\label{it:aequivastrong2} There are $C_2,\widetilde{C}_2>0$ such that for for all pairwise disjoint collections of cubes $\mc{P}$ and all $f\in X$ supported in $\bigcup_{Q\in\mc{P}} Q$ we have
    \[
    \widetilde{C}_2^{-1}\|f\|_X\leq \Big\|\sum_{Q\in\mc{P}}\|f\|_{X_Q}\ind_Q\Big\|_X\leq C_2\|f\|_X;
    \]
    \item\label{it:aequivastrong3} There is a $C_3>0$ such that for for all pairwise disjoint collections of cubes $\mc{P}$ and all $f\in X$ we have
    \[
    \Big\|\sum_{Q\in\mc{P}}\|f\|_{X_Q}\ind_Q\Big\|_X\leq C_3\|f\|_X.
    \]
\end{enumerate}
Then \ref{it:aequivastrong1}$\Leftrightarrow$\ref{it:aequivastrong2}$\Rightarrow$\ref{it:aequivastrong3} with optimal constants satisfying
\[
C_2=C_3,\quad \max\{C_2,\widetilde{C}_2\}\leq [X]_{\mc{G}}\leq C_2\widetilde{C}_2.
\]
If $X\in A$, then \ref{it:aequivastrong3}$\Rightarrow$ $X\in A_{\text{strong}}$, with
\[
[X]_{A_{\text{strong}}}\leq C_3[X]_A.
\]
\end{theorem}
\begin{proof}
For \ref{it:aequivastrong1}$\Rightarrow$\ref{it:aequivastrong2}, we define $\ell_X$ as the Banach function space over $\mc{P}$ with the counting measure as the space of sequences $(a_Q)_{Q\in\mc{P}}$ satisfying
\[
\|(a_Q)_{Q\in\mc{P}}\|_{\ell_X}:=\sup_{\|h\|_{\ell^\infty[X]}\leq 1}\Big\|\sum_{Q\in\mc{P}}a_Qh_Q\ind_Q\Big\|_X<\infty.
\]
For any $f\in X$ supported in $\bigcup_{Q\in\mc{P}}Q$, the sequence $h_Q:=\tfrac{f\ind_Q}{\|f\ind_Q\|_X}$ satisfies $\|h\|_{\ell^\infty[X]}=1$, proving that 
\[
\|f\|_X=\Big\|\sum_{Q\in\mc{P}}\|f\ind_Q\|_Xh_Q\ind_Q\Big\|_X\leq\|(\|f\ind_Q\|_X)_{Q\in\mc{P}}\|_{\ell_X}.
\]
Conversely, note that for any $h\in\ell^\infty[X]$ with $\|h\|_{\ell^\infty[X]}\leq 1$ and any $g\in X'$ we have
\begin{align*}
\int_{\R^d}\!\sum_{Q\in\mc{P}}\|f\ind_Q\|_Xh_Q\ind_Q|g|\,\mathrm{d}x
&\leq\sum_{Q\in\mc{P}}\|f\ind_Q\|_X\|h_Q\|_X\|g\ind_Q\|_{X'}\\
&\leq [X]_{\mc{G}}\|f\|_X\|g\|_{X'}.
\end{align*}
Taking a supremum over $g\in X'$ with $\|g\|_{X'}=1$ and $h\in\ell^\infty[X]$ with $\|h\|_{\ell^\infty[X]}=1$, we conclude that any $f\in X$ supported in $\bigcup_{Q\in\mc{P}}Q$ satisfies
\[
\|f\|_X\leq \|(\|f\ind_Q\|_X)_{Q\in\mc{P}}\|_{\ell_X}\leq [X]_{\mc{G}}\|f\|_X.
\]
Setting $F:=\sum_{Q\in\mc{P}}\|f\|_{X_Q}\ind_Q$ and noting that $\|F\ind_Q\|_X=\|f\ind_Q\|_X$, we have
\[
\|f\|_X\leq \|(\|f\ind_Q\|_X)_{Q\in\mc{P}}\|_{\ell_X}
=\|(\|F\ind_Q\|_X)_{Q\in\mc{P}}\|_{\ell_X}\leq [X]_{\mc{G}}\|F\|_X,
\]
and
\[
\|F\|_X\leq \|(\|F\ind_Q\|_X)_{Q\in\mc{P}}\|_{\ell_X}
=\|(\|f\ind_Q\|_X)_{Q\in\mc{P}}\|_{\ell_X}\leq [X]_{\mc{G}}\|f\|_X,
\]
proving \ref{it:aequivastrong2} with $\max\{C_2,\widetilde{C}_2\}\leq [X]_{\mc{G}}$, as desired.

For \ref{it:aequivastrong2}$\Rightarrow$\ref{it:aequivastrong1}, fix $f\in X$, $g\in X'$, and let $\varepsilon>0$. For each $Q\in\mc{P}$ we pick $h_Q\in X$ with $\|h_Q\|_X\leq 1$ such that
\[
\|g\ind_Q\|_{X'}\leq(1+\varepsilon)\int_{Q}\!|h_Q||g|\,\mathrm{d}x.
\]
Then
\begin{align*}
\sum_{Q\in\mc{P}}\|f\ind_Q\|_X\|g\ind_Q\|_{X'}&\leq(1+\varepsilon)\sum_{Q\in\mc{P}}\|f\ind_Q\|_X\int_Q\!|h_Q||g|\,\mathrm{d}x\\
&\leq(1+\varepsilon)\Big\|\sum_{Q\in\mc{P}}\|f\ind_Q\|_X h_Q\ind_Q\Big\|_X\|g\|_{X'}\\
&\leq(1+\varepsilon)\widetilde{C}_2\Big\|\sum_{Q\in\mc{P}}\|f\|_{X_Q}\|h_Q\|_X\ind_Q\Big\|_X\|g\|_{X'}\\
&\leq(1+\varepsilon)C_2\widetilde{C}_2\|f\|_X\|g\|_{X'}.
\end{align*}
Letting $\varepsilon\downarrow 0$ proves the assertion.

As \ref{it:aequivastrong2}$\Rightarrow$\ref{it:aequivastrong3} is immediate, it remains to prove the final assertion. Suppose $X\in A$ and let $f\in X$. Then, since $\langle f\rangle_{1,Q}\leq[X]_A\|f\|_{X_Q}$ for any cube $Q$, we have
\[
\|A_{\mc{P}}f\|_X\leq[X]_A\Big\|\sum_{Q\in\mc{P}}\|f\|_{X_Q}\ind_Q\Big\|_X\leq C_3[X]_A\|f\|_X.
\]
Thus, $X\in A_{\text{strong}}$ with $[X]_{A_{\text{strong}}}\leq C_3[X]_A$, as desired.
\end{proof}

\begin{remark}
In the literature, the condition \ref{it:aequivastrong1} has been written as $X\in \mc{G}(B)$ (see \cite{Be99}), $X\in G(\Pi_\ast)$ (see \cite{Ko04}), or just $X\in \mc{G}$ (see, e.g., \cite[Section~7.3]{DHHR11}, \cite{CDH11}).
\end{remark}

\begin{remark}
Theorem~\ref{thm:aequivastrong}\ref{it:aequivastrong3} gives a sufficient condition under which $A$ and $A_{\text{strong}}$ are equivalent. This leads one to wonder if there is such a condition that characterizes when $A$ and $A_{\text{strong}}$ are equivalent. More precisely, one can ask whether there is a suitably nice condition $\mc{H}$ for which $X\in A\cap\mc{H}$ is equivalent to $X\in A_{\text{strong}}$ for all Banach function spaces $X$ over $\R^d$ with the Fatou property. We leave this problem as an open question.
\end{remark}

We will see below in Corollary~\ref{cor:aequivastrongweaktype} that if $X\in A_{\text{strong}}$ and $X\in A$ are equivalent in a class of Banach function spaces with the Fatou property, then both of these conditions are characterized by the weak-type bound $M:X\to X_{\text{weak}}$, and
\[
\|M\|_{X\to X_{\text{weak}}}\lesssim_d[X]_{A_{\text{strong}}}.
\]
Turning to the proof of Theorem~\ref{prop:D}, we observe that if $X^r\in A_{\text{strong}}$ for some $r>1$, then also $X\in A_{\text{strong}}$ and thus $X'\in A_{\text{strong}}$ by Proposition~\ref{prop:bfspropertiesduality}. Hence,
\[
\|M\|_{X'\to (X')_{\text{weak}}}\lesssim_d[X']_{A_{\text{strong}}}\leq[X]_{A_{\text{strong}}}\leq[X^r]^{\frac{1}{r}}_{A_{\text{strong}}}.
\]
Theorem~\ref{prop:D} is as follows.
\begin{theorem}\label{lem:sparsetomweak}
Let $X$ be a quasi-Banach function space over $\R^d$. Suppose there is an $r>1$ for which $X^r\in A_{\text{strong}}$. Then for all sparse collections $\mc{S}$ we have $A_{\mc{S}}:X'\to (X')_{\text{weak}}$ with
\[
\sup_{\mc{S}}\|A_{\mc{S}}\|_{(X')\to (X')_{\text{weak}}}\lesssim_d r'(1+\log r')\|M\|_{X'\to (X')_{\text{weak}}}[X^r]^{\frac{1}{r}}_{A_{\text{strong}}}.
\]
\end{theorem}
The proof follows along the same lines of the one in \cite{DLR16}, but with a modification at the end of the proof inspired by the one used in \cite{Le20}.
\begin{proof}[Proof of Theorem~\ref{lem:sparsetomweak}]
By the $3^d$ lattice theorem and Lemma~\ref{lem:sparsedomofsparse} we may assume that $\mc{S}\subseteq\mc{D}$ for some dyadic grid $\mc{D}$, and
\begin{equation}\label{eq:sparsetomweak1}
\sum_{Q'\in\text{ch}_{\mc{S}}(Q)}|Q'|\leq\frac{1}{2}|Q|.
\end{equation}
Moreover, by the Fatou property of $(X')_{\text{weak}}$, we may assume that $\mc{S}$ is finite. Let $g\in X'$ and $f\in X$ with $\|f\|_{X}=1$. Then we want to show that for all $\lambda>0$ we have
\[
\lambda\int_{\R^d}\ind_{\{A_{\mc{S}}g>\lambda\}}|f|\,\mathrm{d}x\lesssim r'(1+\log r')\|M\|_{X'\to (X')_{\text{weak}}}[X^r]^{\frac{1}{r}}_{A_{\text{strong}}}\|g\|_{X'}.
\]
By replacing $g$ by $\tfrac{2g}{\lambda}$, it suffices to consider the case $\lambda=2$. Set
\[
E:=\{A_{\mc{S}}g>2\}\backslash\{M^\mc{D} g>\tfrac{1}{4}\}.
\]
Then
\begin{align*}
\int_{\R^d}\ind_{\{A_{\mc{S}}g>2\}}|f|\,\mathrm{d}x
&\leq \int_{E}\!|f|\,\mathrm{d}x+\int_{\R^d}\!\ind_{\{M^\mc{D} g>\tfrac{1}{4}\}}|f|\,\mathrm{d}x\\
&\leq \int_{E}\!|f|\,\mathrm{d}x+4\|M^{\mc{D}}g\|_{(X')_{\text{weak}}}.
\end{align*}
Since $M^{\mc{D}}:X'\to (X')_{\text{weak}}$, it remains to estimate the other term.

Using the notation from Lemma~\ref{lem:sparsetoweak}, we denote the maximal elements of $\mc{S}_m$ by $\mc{S}_m^\ast$, and note that by Kolmogorov's Lemma we have
\[
\sum_{Q\in\mc{S}_m(Q_0)}\langle f\rangle_{r,Q}|Q|\leq 2\int_{Q_0}M^{\mc{D}(Q_0)}_{r}g\,\mathrm{d}x\lesssim r'\langle f\rangle_{r,Q_0}|Q_0|.
\]
Thus, it follows from Lemma~\ref{lem:sparsetoweak}, H\"older's inequality, and the fact that 
\[
\bigcup_{Q_0\in\mc{S}_m^\ast}Q_0\subseteq\{M^{\mc{D}}g>4^{-(m+1)}\},
\]
that
\begin{align*}
\int_E\!|f|\,\mathrm{d}x
&\leq\sum_{m=1}^\infty 4^{-m}\sum_{Q\in\mc{S}_m}\int_{F_m(Q)}\!|f|\,\mathrm{d}x\\
&\leq \sum_{m=1}^\infty 4^{-m}2^{-\frac{2^m}{r'}}\sum_{Q_0\in\mc{S}^\ast_m}\sum_{Q\in\mc{S}_m(Q_0)}\langle f\rangle_{r,Q}|Q|\\
&\lesssim r'\sum_{m=1}^\infty 4^{-m}2^{-\frac{2^m}{r'}}\sum_{Q_0\in\mc{S}^\ast_m}\langle f\rangle_{r,Q_0}|Q_0|\\
&\leq r'\sum_{m=1}^\infty 4^{-m}2^{-\frac{2^m}{r'}}\int_{\R^d}\!\ind_{\{M^{\mc{D}}g>4^{-(m+1)}\}}\!A_{\mc{S}_m^\ast}(|f|^r)^{\frac{1}{r}}\,\mathrm{d}x\\
&\leq 4r'\sum_{m=1}^\infty 2^{-\frac{2^m}{r'}}\|Mg\|_{(X')_{\text{weak}}}\|A_{\mc{S}_m^\ast}(|f|^{r})\|_{X^r}^{\frac{1}{r}}\\
&\lesssim r'(1+\log r')\|M\|_{X'\to (X')_{\text{weak}}}[X^r]_{A_{\text{strong}}}^{\frac{1}{r}}\|g\|_{X'}.
\end{align*}
The assertion follows.
\end{proof}

It is not true that if $X\in A_{\text{strong}}$, then there is an $r>1$ for which $X^r\in A_{\text{strong}}$. A counterexample is the space $X=L^1(\R^d)$. However, this self-improvement is true for the condition $X\in A_{\text{sparse}}$.
\begin{theorem}\label{thm:sparseselfimprovement}
Let $X$ be a quasi-Banach function space over $\R^d$. If $X\in A_{\text{sparse}}$, then there is an $r_0>1$ such that for every $1<r\leq r_0$ we have $X^r\in A_{\text{sparse}}$, and
\[
[X^r]^{\frac{1}{r}}_{A_{\text{sparse}}}\lesssim_{d,K_X}[X]_{A_{\text{sparse}}}.
\]
\end{theorem}
We will use the following result based on the sharp reverse H\"older theorem of \cite{HP13}.
\begin{lemma}\label{lem:RHrregularity}
Let $X$ be a quasi-Banach function space over $\R^d$ for which $M:X\to X$. Then there is a dimensional constant $C_d>1$ such that for any $1<r<\infty$ satisfying
\[
r'\geq C_dK_X\|M\|_{X\to X}
\]
and each $f\in X$, there exists a weight $w\geq |f|$ satisfying $\|w\|_X\leq 2K_X^2\|M\|_{X\to X}\|f\|_X$, and
\[
\langle w\rangle_{r,Q}\lesssim_d \langle w\rangle_{1,Q}
\]
for all cubes $Q$.
\end{lemma}
\begin{proof}
As we will see in Proposition~\ref{prop:rutskya1reg} below, for every $f\in X$ there exists a $w\in A_1$ for which
\[
\|w\|_X\leq 2K_X^2\|f\|_X,\quad [w]_1\leq 2K_X\|M\|_{X\to X}.
\]
By the sharp reverse H\"older theorem of \cite{HP13}, any such $w$ satisfies
\[
\langle w\rangle_{r,Q}\lesssim_d \langle w\rangle_{1,Q}
\]
as long as $r'\geq C_d[w]_{\text{FW}}$ for a certain $C_d>1$. In particular, since
\[
[w]_{\text{FW}}\leq[w]_1\leq 2K_X\|M\|_{X\to X},
\]
this is the case for $r$ satisfying
\[
r'\geq 2C_d K_X \|M\|_{X\to X},
\]
proving the result.
\end{proof}
\begin{proof}[Proof of Theorem~\ref{thm:sparseselfimprovement}]
Let $\mc{S}$ be a sparse collection, and let $f\in X^r$, where $r$ satisfies 
\[
r'\geq C_dK_X \|M\|_{X\to X}=:r_0'.
\]
Since $M$ satisfies sparse domination, the assertion $X\in A_{\text{sparse}}$ implies that $M:X\to X$. Hence, since $|f|^{\frac{1}{r}}\in X$, by Lemma~\ref{lem:RHrregularity} we can pick a weight $w\geq |f|^{\frac{1}{r}}$ satisfying
\[
\|w\|_X\lesssim_{K_X} \|f\|_{X^r}^{\frac{1}{r}},\quad \langle w\rangle_{r,Q}\lesssim\langle w\rangle_{1,Q}
\]
for all cubes $Q$. Since $\|\cdot\|_{\ell^r}\leq\|\cdot\|_{\ell^1}$, we conclude that
\begin{align*}
\|A_{\mathcal{S}}f\|^{\frac{1}{r}}_{X^r}
&=\Big\|\Big(\sum_{Q\in\mathcal{S}}\langle |f|^{\frac{1}{r}}\rangle_{r,Q}^r\ind_Q\Big)^{\frac{1}{r}}\Big\|_{X}
\leq\Big\|\sum_{Q\in\mathcal{S}}\langle w\rangle_{r,Q}\ind_Q\Big\|_{X}\\
&\lesssim_d \|A_{\mc{S}}w\|_X
\leq [X]_{A_{\text{sparse}}}\|w\|_{X}\\
&\lesssim_{K_X}[X]_{A_{\text{sparse}}}\|f\|_{X^r}^{\frac{1}{r}}.
\end{align*}
Thus, $X^r\in A_{\text{sparse}}$ with
\[
[X^r]^{\frac{1}{r}}_{A_{\text{sparse}}}\lesssim_{d,K_X}[X]_{A_{\text{sparse}}},
\]
as desired.
\end{proof}

\subsection{Bounds of the Hardy-Little maximal operator}
The Hardy-Littlewood maximal operator $M$ is defined by 
\[
Mf:=\sup_Q T_Qf.
\]
Moreover, for a collection of cubes $\mc{P}$ we set $M^{\mc{P}}f:=\sup_{Q\in\mc{P}}T_Qf$. The generalized Muckenhoupt conditions are directly related to boundedness properties of $M$.
\begin{theorem}\label{thm:ApropsHL}
Let $X$ be a quasi-Banach function space over $\R^d$. Then we have the following assertions:
\begin{enumerate}[(a)]
\item\label{it:mmuckenhoupt1} If
$
M:X\to X_{\text{weak}},
$
then $X\in A$ with
\[
[X]_{A}\leq\|M\|_{X\to X_{\text{weak}}}.
\]
\item\label{it:mmuckenhoupt2} If $X\in A_{\text{strong}}$ and $X$ has the Fatou property, then $M:X\to X_{\text{weak}}$ with
\[
\|M\|_{X\to X_{\text{weak}}}\lesssim_d[X]_{A_{\text{strong}}}.
\]
\item\label{it:mmuckenhoupt3} If 
$
M:X\to X,
$
then $X\in A_{\text{strong}}$ with
\[
[X]_{A_{\text{strong}}}\leq\|M\|_{X\to X}.
\]
\item\label{it:mmuckenhoupt4} If $X\in A_{\text{sparse}}$ and $X$ has the Fatou property, then $M:X\to X$ with
\[
\|M\|_{X\to X}\lesssim_d[X]_{A_{\text{sparse}}}.
\]
\item\label{it:mmuckenhoupt5} If $X$ is a Banach function space with the Fatou property, then
\[
M:X\to X,\quad M:X'\to X'
\]
if and only if $X\in A_{\text{sparse}}$, with
\begin{equation}\label{eq:sparsedualmbounds1}
\max\big\{\|M\|_{X\to X},\|M\|_{X'\to X'}\big\}\lesssim_d[X]_{A_{\text{sparse}}}\lesssim\|M\|_{X\to X}\|M\|_{X'\to X'}.
\end{equation}
\end{enumerate}
\end{theorem}
The Fatou property is assumed in \ref{it:mmuckenhoupt2} and \ref{it:mmuckenhoupt4} so that
\[
\sup_{\mc{F}}\|M^{\mc{F}}f\|_X=\|Mf\|_X,
\]
where the supremum is over all finite collections of cubes $\mc{F}$. It can be removed if the conclusion is replaced by the boundedness of $M^{\mc{F}}$ uniformly over all finite $\mc{F}$.

\begin{proof}[Proof of Theorem~\ref{thm:ApropsHL}]
Assertion \ref{it:mmuckenhoupt1} is \cite[Proposition~4.21]{Ni23}. For a slick proof, note that by Proposition~\ref{prop:AXcondition} and the ideal property of $X_{\text{weak}}$ we have
\[
[X]_{A}=\sup_{Q}\|T_Q\|_{X\to X_{\text{weak}}}\leq\|\sup_{Q}T_Q\|_{X\to X_{\text{weak}}}=\|M\|_{X\to X_{\text{weak}}},
\]
as desired.

For \ref{it:mmuckenhoupt2}, by the $3^d$-lattice theorem, it suffices to prove the result for $M^\mathcal{D}$ for a dyadic grid $\mathcal{D}$. Moreover, by the Fatou property of $X_{\text{weak}}$, it suffices to prove the result for $M^\mathcal{F}$ for a finite collection $\mathcal{F}\subseteq\mathcal{D}$. Let $\lambda>0$, $f\in X$, and let $\mc{P}$ denote the collection of maximal cubes $P\in\mc{F}$ satisfying $\langle f\rangle_{1,P}>\lambda$, such that we have the Calder\'on-Zygmund decomposition
\[
\{x\in\R^d:M^\mathcal{F}f(x)>\lambda\}=\bigcup_{P\in\mathcal{P}}P.
\]
As $\mc{P}$ is pairwise disjoint, we have $\ind_{\{x\in\R^d:M^\mathcal{F}f(x)>\lambda\}}=\sum_{P\in\mathcal{P}}\ind_P$, which implies that
\begin{equation}\label{eq:distoweak1}
\lambda\|\ind_{\{x\in\R^d:M^\mathcal{F}f(x)>\lambda\}}\|_X=\Big\|\sum_{P\in\mathcal{P}}\lambda\ind_P\Big\|_X\leq\|T_{\mathcal{P}}f\|_X\leq[X]_{A_{\text{strong}}}\|f\|_X.
\end{equation}
Hence, taking a supremum over $\lambda>0$ yields
\[
\|M^\mathcal{F} f\|_{X_{\text{weak}}}\leq[X]_{A_{\text{strong}}}\|f\|_X,
\]
as desired.

Assertion \ref{it:mmuckenhoupt3} follows from the fact that for each pairwise disjoint collection of cubes $\mathcal{F}$ and $f\in L^0(\R^d)$ we have $T_{\mathcal{F}}f\leq Mf$.

To prove \ref{it:mmuckenhoupt4}, we note that for any $f\in L^1_{\loc}(\R^d)$, each dyadic grid $\mathcal{D}$, and each finite collection $\mathcal{F}\subseteq\mathcal{D}$, there exists a sparse collection $\mathcal{S}\subseteq\mathcal{F}$ so that
\[
M^\mathcal{F}f\leq 2 T_\mathcal{S}f.
\]
Hence, the result follows from the ideal and Fatou properties of $X$ and the $3^d$-lattice theorem.

For the final assertion \ref{it:mmuckenhoupt5}, we note that the first inequality in \eqref{eq:sparsedualmbounds1} follows from \ref{it:mmuckenhoupt4} and Proposition~\ref{prop:bfspropertiesduality}. For the second inequality, note that
\[
\|(A_{\mc{S}}f)g\|_{L^1(\R^d)}=\sum_{Q\in\mc{S}}\langle f\rangle_{1,Q}\langle g\rangle_{1,Q}|Q|\leq 2\int_{\R^d}\!(Mf)(Mg)\,\mathrm{d}x\lesssim \|Mf\|_X\|Mg\|_{X'},
\]
so that by the Lorentz-Luxemburg theorem we have
\[
[X]_{A_{\text{sparse}}}\lesssim \|M\|_{X\to X}\|M\|_{X'\to X'}.
\]
This proves the result.
\end{proof}

\begin{remark}
From the proof of Theorem \ref{thm:ApropsHL}\ref{it:mmuckenhoupt2} we actually find that, if $X\in A_{\text{strong}}$, then for every dyadic grid $\mathcal{D}$ in $\R^d$ and $\lambda>0$, we have
\begin{equation}\label{eq:distoweak2}
\lambda\|\ind_{\{x\in\R^d:M^\mathcal{D} f(x)>\lambda\}}\|_X\leq[X]_{A_{\text{strong}}}\|f\ind_{\{x\in\R^d:M^\mathcal{D} f(x)>\lambda\}}\|_X.
\end{equation}
To see this, note that since $T_{\mathcal{P}}f=T_{\mathcal{P}}(f\ind_{\{x\in\R^d:M^\mathcal{F} f(x)>\lambda\}})$ in \eqref{eq:distoweak1}, we can replace the last estimate here with the more precise bound
\[
\|T_{\mathcal{P}}f\|_X\leq[X]_{A_{\text{strong}}}\|f\ind_{\{x\in\R^d:M^\mathcal{F} f(x)>\lambda\}}\|_X.
\]
The assertion \eqref{eq:distoweak2} then follows from the Fatou property of $X$. In the case $X=L^1(\R^d)$, \eqref{eq:distoweak2} is similar to a more classical estimate (with $M^{\mc{D}}f$ replaced by $2f$ in the upper level set on the right-hand side) which is part of an equivalence that can be found, e.g., in \cite[p.~144, Theorem~2.1]{GR85}. This opens up the question of whether this same equivalence also holds for general $X$.
\end{remark}

Combined with Theorem~\ref{thm:aequivastrong} we obtain the following corollary:
\begin{corollary}\label{cor:aequivastrongweaktype}
Let $X$ be a Banach function space over $\R^d$ with the Fatou property. If $X\in\mc{G}$, then the following are equivalent:
\begin{enumerate}[(i)]
    \item $X\in A$;
    \item $X\in A_{\text{strong}}$;
    \item $M:X\to X_{\text{weak}}$,
\end{enumerate}
with
\[
[X]_A\leq\|M\|_{X\to X_{\text{weak}}}\lesssim_d[X]_{A_{\text{strong}}}\leq [X]_{\mc{G}}[X]_A.
\]
\end{corollary}
In fact, by Theorem~\ref{thm:aequivastrong}, for this equivalence to hold it suffices to assume the weaker condition that there is a $C>0$ such that for every pairwise disjoint collection of cubes $\mc{P}$ and all $f\in X$ we have
\[
\Big\|\sum_{Q\in\mc{P}}\|f\|_{X_Q}\ind_Q\Big\|_X\leq C\|f\|_X.
\]

\subsection{\texorpdfstring{$A_p$}{Ap}-regularity}\label{sec:rutsky}
The notion of $A_p$-regularity of a quasi-Banach function space was widely studied in the works of Rutsky, see, e.g., \cite{Ru14, Ru19, Ru15, Ru18}. In this section we will do a quantitative study of some of his results.

To streamline our notation, we will call a collection of weights $B$ a \emph{weight class} if there is an associated constant $[\cdot]_B$ with the property that $w\in B$ if and only if $[w]_B<\infty$.
\begin{definition}
Let $B$ be a weight class and let $X$ be a quasi-Banach function space $X$ over $\R^d$. We say that $X$ is \emph{$B$-regular} if there exist constants $C_1,C_2>0$ such that for all $f\in X$ there exists a weight $w\in B$ with $w\geq|f|$, $w\in X$, and
\[
\|w\|_X\leq C_1\|f\|_X,\quad [w]_B\leq C_2.
\]
\end{definition}
We point out that our notation varies from that of Rutsky: since we have introduced weights using the multiplier approach instead of the change of measure approach, a space $X$ is $A_p$-regular in the notation of Rutsky if and only if $X^{\frac{1}{p}}$ is $A_p$-regular in our current notation.

The notion of $A_p$-regularity can be seen as a generalization of the boundedness $M:X\to X$. Indeed, we have the following result.
\begin{proposition}\label{prop:rutskya1reg}
Let $X$ be a quasi-Banach function space over $\R^d$. Then 
\[
M:X\to X
\]
if and only if $X$ is $A_1$-regular. In this case, we can take
\[
C_1=2K_X^2,\quad C_2=2K_X\|M\|_{X\to X}.
\]
\end{proposition}
The proof uses the Rubio de Francia iteration algorithm, which is essentially the main tool used in modern extrapolation theory.
\begin{proof}[Proof of Proposition~\ref{prop:rutskya1reg}]
For the direct implication, set
\[
w:=\sum_{n=0}^\infty \frac{M^n f}{(2K_X)^n\|M\|_{X\to X}^n},
\]
where we have recursively defined $M^0f:=|f|$ and $M^{n+1}f:=M(M^n f)$ for $n\geq 1$. Note that by the Riesz-Fischer property of $X$, see \cite[Section~2.1]{LN23b}, we have $w\in X$ with
\[
\|w\|_X\leq K_X\sum_{n=0}^\infty \frac{K_X^{n+1}}{(2K_X)^n}\|f\|_X=2K_X^2\|f\|_X.
\]
Thus, the result follows with $C_1=2K_X^2$, $C_2=2K_X\|M\|_{X\to X}$.

For the converse, note that for any $f\in X$ we have $Mf\leq Mw\leq C_2w$ so that by the ideal property of $X$ we have $Mf\in X$ with
\[
\|Mf\|_X\leq C_2\|w\|_X\leq C_1C_2\|f\|_X.
\]
Thus, $M:X\to X$ with $\|M\|_{X\to X}\leq C_1C_2$. The result follows.
\end{proof}

We also have the following result relating $A_p$-regularity to bounds of $M$.
\begin{proposition}\label{prop:apregularimpliesmbounds}
Let $1\leq p<\infty$ and let $X$ be a $p$-convex Banach function space over $\R^d$. If $X$ is $A_p$-regular, then we have:
\begin{enumerate}[(a)]
\item\label{it:propapregularimpliesmbounds1} $M:[(X^p)']^{\frac{1}{p}}\to [(X^p)']^{\frac{1}{p}}$ if $p>1$;
\item\label{it:propapregularimpliesmbounds2} $M:X\cdot L^{p'}(\R^d)\to X\cdot L^{p'}(\R^d)$.
\end{enumerate}
Moreover, if $C_1,C_2>0$ are the $A_p$-regularity constants, then
\[
\|M\|_{[(X^p)']^{\frac{1}{p}}\to [(X^p)']^{\frac{1}{p}}}\lesssim_d p'C_1 C_2^{p'},\quad \|M\|_{X\cdot L^{p'}(\R^d)\to X\cdot L^{p'}(\R^d)}\lesssim_d p C_1C_2^p.
\]
\end{proposition}
Note that
\[
(X\cdot L^{p'}(\R^d))'=[(X^p)']^{\frac{1}{p}}.
\]
\begin{proof}[Proof of Proposition~\ref{prop:apregularimpliesmbounds}]
For \ref{it:propapregularimpliesmbounds1}, let $g\in [(X^p)']^{\frac{1}{p}}$ and $f\in X^p$ with $\|f\|_{X^p}=1$. Picking a weight $w\in A_p$ with $w\geq|f|^{\frac{1}{p}}$, $\|w\|_X\leq C_1$, $[w]_p\leq C_2$, we have
\begin{align*}
\Big(\int_{\R^d}\!(Mg)^p|f|\,\mathrm{d}x\Big)^{\frac{1}{p}}&\leq\|Mg\|_{L^p_w(\R^d)}\lesssim_d p'[w]_p^{p'}\|g\|_{L_w^p(\R^d)}
\leq p'C_2^{p'}\|g\|_{[(X^p)']^{\frac{1}{p}}}\|w\|_X\\
&\leq p'C_1C_2^{p'}\|g\|_{[(X^p)']^{\frac{1}{p}}},
\end{align*}
where we used Buckley's bound \cite{Bu93}, so that
\[
\|Mg\|_{[(X^p)']^{\frac{1}{p}}}=\sup_{\|f\|_{X^p}=1}\Big(\int_{\R^d}\!(Mg)^p|f|\,\mathrm{d}x\Big)^{\frac{1}{p}}\lesssim_d p'C_1C_2^{p'}\|g\|_{[(X^p)']^{\frac{1}{p}}},
\]
as desired.

For \ref{it:propapregularimpliesmbounds2}, let $f\in X\cdot L^{p'}(\R^d)$, pick $0\leq h\in X$, $0\leq k\in L^{p'}(\R^d)$ such that $|f|\leq hk$, and pick $w\geq h$ as in the definition of $A_p$ regularity. Then,
using Buckley's bound \cite{Bu93} and the fact that $[w^{-1}]_{p'}=[w]_p$, we have
\begin{align*}
\|Mf\|_{L_{w^{-1}}^{p'}(\R^d)}\lesssim_d p[w^{-1}]_{p'}^p\|f\|_{L^{p'}_{w^{-1}}(\R^d)}\leq p C_2^p\|k\|_{L^{p'}(\R^d)},
\end{align*}
where in the last estimate we used $|f|w^{-1}\leq hw^{-1} k\leq k$. This implies that
\begin{align*}
\|Mf\|_{X\cdot L^{p'}(\R^d)}\leq\|w\|_{X}\|Mf\|_{L^{p'}_{w^{-1}}(\R^d)}\lesssim_d p C_1C_2^p\|h\|_X\|k\|_{L^{p'}(\R^d)}.
\end{align*}
Taking an infimum over all $0\leq h\in X$, $0\leq k\in L^{p'}(\R^d)$ such that $|f|\leq h k$ now proves the assertion.
\end{proof}

The notion of $A_p$-regularity is closely related to the linearization of $M$. Indeed, let $\mc{Q}$ denote the (countable) collection of cubes in $\R^d$ with rational corners and define the map
\[
\mathcal{M}((f_Q)_{Q\in\mc{Q}}):=(\langle f\rangle_Q\ind_Q)_{Q\in\mc{Q}}.
\]
By Proposition~\ref{prop:mlinearization} we have $M:X\to X$ if and only if $\mathcal{M}:X[\ell^\infty]\to X[\ell^\infty]$, where we have indexed $\ell^\infty$ over $\mc{Q}$, and, if $X$ is a Banach function space with the Fatou property, $M:X'\to X'$ if and only if $\mc{M}:X[\ell^1]\to X[\ell^1]$. Moreover, by Corollary~\ref{cor:mlinearization} we have
\[
M:X\to X,\quad M:X'\to X'
\]
if and only if $\mc{M}:X[\ell^r]\to X[\ell^r]$ for all $1\leq r\leq\infty$. When $X$ is $r$-convex, we can actually expand this characterization to only requiring this bound for a single exponent.
\begin{theorem}\label{thm:Xconvexsingleexponentlinearizedm}
Let $r_0>1$ and let $X$ be an $r_0$-convex Banach function space over $\R^d$ with the Fatou property. Then the following are equivalent:
\begin{enumerate}[(i)]
    \item\label{it:Xconvexsingleexponentlinearizedm1} $M:X\to X$, $M:X'\to X'$;
    \item\label{it:Xconvexsingleexponentlinearizedm2} $\mc{M}:X[\ell^r]\to X[\ell^r]$ for some $1<r\leq r_0$;
\end{enumerate}
\end{theorem}
\begin{remark}
Using Theorem~\ref{thm:klwrdfnoconvexity} below, we point out that the implication \ref{it:Xconvexsingleexponentlinearizedm2}$\Rightarrow$\ref{it:Xconvexsingleexponentlinearizedm1} is true for $r=2$ without any convexity assumption on $X$.
\end{remark}

Theorem~\ref{thm:Xconvexsingleexponentlinearizedm} follows directly from a characterization proved by Rutsky in \cite{Ru15}. We prove a sharp version of it.
\begin{proposition}\label{prop:apregandmlinearization}
Let $r>1$ and let $X$ be an $r$-convex Banach function space over $\R^d$ with the Fatou property. Then the following are equivalent:
\begin{enumerate}[(i)]
\item\label{it:apregandmlinearization1} $\mc{M}:X[\ell^r]\to X[\ell^r]$;
\item\label{it:apregandmlinearization2} $[(X^r)']^{\frac{1}{r}}$ is $A_r$-regular.
\item\label{it:apregandmlinearization3} $M:X\to X$ and $M:X'\to X'$.
\end{enumerate}
In this case we can take $C_1=2^{\frac{1}{r}}$ in \ref{it:apregandmlinearization2}, and
\[
C_2\leq 2^{\frac{1}{r}}M^{(r)}(X)\|\mc{M}\|_{X[\ell^r]\to X[\ell^r]},\quad \|\mc{M}\|_{X[\ell^r]\to X[\ell^r]}\leq 2^{\frac{1}{r}}M^{(r)}(X) C_2.
\]
\end{proposition}
The main ingredient in the proof is a version of a classical result by Rubio de Francia \cite{Ru86}. The following version is proved in \cite[Lemma~3.4]{ALV17} (see also \cite[Corollary~6.1.4]{Lo16}).
\begin{theorem}\label{thm:rdfapregularity}
Let $r\geq 1$ and let $X$ be an $r$-convex Banach function space over $\R^d$ with the Fatou property. Let $\Gamma$ be a family of (sub)linear operators on $X$. Then the following are equivalent:
\begin{enumerate}[(i)]
\item\label{it:rdfapregularity1} There exists a $C>0$ such that for any finite index set $\mc{I}$ and $(T_n)_{n\in\mc{I}}\subseteq\Gamma$, $(f_n)_{n\in\mc{I}}\subseteq X$, we have
\[
\|(T_nf_n)_{n\in I}\|_{X[\ell^r(\mc{I})]}\leq C\|(f_n)_{n\in\mc{I}}\|_{X[\ell^r(\mc{I})]};
\]
\item\label{it:rdfapregularity2} $[(X^r)']^{\frac{1}{r}}$ is $B$-regular with $C_1=2^{\frac{1}{r}}$, where
\[
[w]_B:=\sup_{T\in\Gamma}\|T\|_{L^r_w(\R^d)\to L^r_w(\R^d)},\quad B:=\{w:[w]_B<\infty\}.
\]
\end{enumerate}
Moreover, in this case we have 
\[
C_2\leq 2^{\frac{1}{r}}M^{(r)}(X) C,\quad  C\leq 2^{\frac{1}{r}}M^{(r)}(X) C_2.
\]
\end{theorem}
Note that for \ref{it:rdfapregularity2} to make sense, one needs that $X\cap L^r_w(\R^d)$ is dense in $L^r_w(\R^d)$ for any weight $w$ so that all $T\in\Gamma$ uniquely extend to $L^r_w(\R^d)$. As noted by Rubio de Francia \cite[Page~200]{Ru86}, this follows from the Hahn-Banach theorem.
\begin{remark}
Instead of the Fatou property (which is assumed by Rubio de Francia \cite{Ru86}), one can assume that $X$ is order-continuous (as is done, for example, in the versions of this result by Rutsky \cite{Ru14, Ru19}). Actually, the only requirement for the result is that $X'$ is norming for $X$, i.e., for all $f\in X$ we have $\|f\|_X=\sup_{\|g\|_{X'}=1}\int_{\R^d}\!|fg|\,\mathrm{d}x$.
\end{remark}
\begin{proof}[Proof of Proposition~\ref{prop:apregandmlinearization}]
For \ref{it:apregandmlinearization1}$\Rightarrow$\ref{it:apregandmlinearization2}, note that this follows from Theorem~\ref{thm:rdfapregularity} applied to the family
\[
\Gamma:=\{\widetilde{T}_Q:Q\in\mc{Q}\},\quad \widetilde{T}_Qf:=\langle f\rangle_Q\ind_Q.
\]
Note that in this case we have Theorem~\ref{thm:rdfapregularity}\ref{it:rdfapregularity1} with $C=\|\mc{M}\|_{X[\ell^r]\to X[\ell^r]}$, and Theorem~\ref{thm:rdfapregularity}\ref{it:rdfapregularity2} with $C_1=2^{\frac{1}{r}}$, where now $B=A_r$ by Proposition~\ref{prop:AXcondition}, and $C_2$ as in Theorem~\ref{thm:rdfapregularity}, as desired.

The implication \ref{it:apregandmlinearization2}$\Rightarrow$\ref{it:apregandmlinearization3} follows from Proposition~\ref{prop:apregularimpliesmbounds}, where the Fatou property of $X$ is used so that $Y=[(X^r)']^{\frac{1}{r}}$ satisfies $[(Y^r)']^{\frac{1}{r}}=[(X^r)'']^{\frac{1}{r}}=X$ by the Lorentz-Luxemburg theorem.

Finally, \ref{it:apregandmlinearization3}$\Rightarrow$\ref{it:apregandmlinearization1} follows from Corollary~\ref{cor:mlinearization}.
\end{proof}

We end this section with the proof of Theorem~\ref{thm:A}. This requires a version of Theorem~\ref{thm:rdfapregularity} without any convexity assumption on the space. The following is \cite[Theorem~2.3.1]{KLW23} (see also \cite[Theorem~4.6.2]{Lo16}).
\begin{theorem}\label{thm:klwrdfnoconvexity}
Let $X$ be an order-continuous Banach function space over $\R^d$ with the Fatou property, and let $\Gamma$ be a family of bounded linear operators on $X$. Then the following are equivalent:
\begin{enumerate}[(i)]
\item\label{it:klw1} There exists a $C_1>0$ such that for any finite index sets $\mc{I}$ and $(T_n)_{n\in\mc{I}}\subseteq\Gamma$, $(f_n)_{n\in\mc{I}}\subseteq X$, we have
\[
\|(T_nf_n)_{n\in I}\|_{X[\ell^2(\mc{I})]}\leq C_1\|(f_n)_{n\in\mc{I}}\|_{X[\ell^2(\mc{I})]};
\]
\item\label{it:klw2} For every $f_1,f_2\in X$ there exists a weight $w$ such that
\[
\sup_{T\in\Gamma}\|T\|_{L^2_w(\R^d)\to L^2_w(\R^d)}\leq C_2,\quad \|f_1\|_{L^2_w(\R^d)}\leq 1312\|f_1\|_X,\quad \|f_2\|_X\leq 1312\|f_2\|_{L^2_w(\R^d)}.
\]
\end{enumerate}
In this case $C_1$ and $C_2$ can be chosen such that $C_1\eqsim C_2$.
\end{theorem}
\begin{proof}[Proof of Theorem~\ref{thm:A}]
First, we show that \ref{it:bfs1}$\Rightarrow$\ref{it:bfs5}$\Rightarrow$ \ref{it:bfssparse}, proving the equivalence of \ref{it:bfs5} with \ref{it:bfs1}-\ref{it:bfssparse}.

For \ref{it:bfs1}$\Rightarrow$\ref{it:bfs5}, pick any non-degenerate Calder\'on-Zygmund operator $T$, such as one of the Riesz transforms. For \ref{it:bfs5}$\Rightarrow$\ref{it:bfssparse},
by the Grothendieck theorem (see, e.g., \cite[Theorem~3]{Kr74}), we have $\widetilde{T}:X[\ell^2]\to X[\ell^2]$, where $\widetilde{T}((f_n)_{n\geq 1})=(Tf_n)_{n\geq 1}$. Thus, by Theorem~\ref{thm:klwrdfnoconvexity} with $\Gamma:=\{T\}$, we find that \ref{it:klw2} also holds. Since $T:X\to X$, by Proposition~\ref{prop:nondegimpliesmuckenhoupt} we have $X\in A$ and, for any $w$ constructed as in \ref{it:klw2},
\[
[w]_2\lesssim \|T\|_{L^2_w(\R^d)\to L^2_w(\R^d)}\lesssim \|\widetilde{T}\|_{X[\ell^2]\to X[\ell^2]}\lesssim\|T\|_{X\to X}.
\]
Since $X\in A$, for any finite sparse collection of cubes $\mc{S}$ we have $A_{\mc{S}}f\in X$ when $f\in X$. For fixed $f\in X$ we apply \ref{it:klw2} with $f_1=f$, $f_2=A_{\mc{S}}f$. Then, by the $A_2$-bound for sparse operators,
\[
\|A_{\mc{S}}f\|_X\lesssim\|A_{\mc{S}}f\|_{L^2_w(\R^d)}\lesssim_d[w]_2^2\|f\|_{L^2_w(\R^d)}\lesssim\|T\|_{X\to X}^2\|f\|_X.
\]
By the Fatou property of $X$, we can extend this result to any sparse collection $\mc{S}$, proving that $X\in A_{\text{sparse}}$ with
\[
[X]_{A_{\text{sparse}}}\lesssim_d\|T\|_{X\to X}^2.
\]
The assertion follows.

To prove the equivalence of \ref{it:bfs6} with \ref{it:bfs1}-\ref{it:bfssparse}, we first prove \ref{it:bfs3}$\Rightarrow$\ref{it:bfs6}. By Corollary~\ref{cor:mlinearization} we have
\[
\|\mc{M}\|_{X[\ell^2]\to X[\ell^2]}\leq\|M\|^{\frac{1}{2}}_{X\to X}\|M\|^{\frac{1}{2}}_{X'\to X'}.
\]
Since we can take $C=\|\mc{M}\|_{X[\ell^2]\to X[\ell^2]}$ in \ref{it:bfs6}, this proves the assertion. Conversely, we proceed analogous to the implication \ref{it:bfs5}$\Rightarrow$\ref{it:bfssparse} above, except this time we take the family
\[
\Gamma:=\{\widetilde{T}_Q:Q\text{ a cube}\},\quad \widetilde{T}_Qf:=\langle f\rangle_Q\ind_Q.
\]
For this family, Theorem~\ref{thm:klwrdfnoconvexity}\ref{it:klw1} is precisely our assumption \ref{it:bfs6}, with $C=C_1$. Thus, since
\[
\sup_Q\|T_Q\|_{L^2_w(\R^d)\to L^2_w(\R^d)}=[w]_2
\]
by Proposition~\ref{prop:AXcondition}, proceeding as above yields
\[
[X]_{A_{\text{sparse}}}\lesssim_d C^2.
\]
The assertion follows.
\end{proof}

\section{Duality of the Hardy-Littlewood maximal operator}\label{sec:dualityofm}

The general duality question \eqref{Q} is posed for Banach function spaces rather than the more general \emph{quasi-}Banach function spaces. It is worth noting that it is not true that if $M$ is bounded on a quasi-Banach function space $X$, then $X$ can be renormed to a Banach function space. For example, the Lorentz space $X=L^{p,q}(\R^d)$ for $p\in (1,\infty)$ and $q\in(0,1)$ is a quasi-Banach function space on which $M$ is bounded that is not a Banach function space. However, we claim that solving \eqref{Q} as posed is sufficient to also conclude the result for any quasi-Banach function space $X$. 

Indeed, suppose $X$ is a quasi-Banach function space on which $M$ is bounded. Then, by a duality result originally shown in \cite[Proposition~2.27]{Ni23}, $M$ is also bounded on $X''$, which is a Banach function space. Since $X'''=X'$, we need only answer \eqref{Q} for $X''$ instead of $X$, as claimed. 

We give a brief proof of this duality result based on the linearization $\mc{M}$ here.
\begin{theorem}\label{thm:mseconddual}
Let $X$ be a quasi-Banach function space over $\R^d$ with the Fatou property. If $M:X\to X$, then also $M:X''\to X''$ with
\[
\|M\|_{X''\to X''}\leq\|M\|_{X\to X}.
\]
\end{theorem}
\begin{proof}
By the second assertion of Proposition~\ref{prop:mlinearization} and the fact that $X'''=X'$, the conclusion is equivalent to
\[
\mc{M}:X'[\ell^1]\to X'[\ell^1].
\]
But, with the same proof as the one of \ref{it:mlinearization3}$\Rightarrow$\ref{it:mlinearization4} in Proposition~\ref{prop:mlinearization}, this follows from $M:X\to X$ with
\[
\|\mc{M}\|_{X'[\ell^1]\to X'[\ell^1]}\leq\|M\|_{X\to X},
\]
proving the result.
\end{proof}

To pass bounds of $M$ from $X$ to $X'$, we note that by the original result of Fefferman and Stein \cite{FS71}, we have
\begin{equation}\label{eq:feffstein}
\Big(\int_{\R^d}\!(Mf)^p|g|\,\mathrm{d}x\Big)^{\frac{1}{p}}\lesssim_d p'\Big(\int_{\R^d)}|f|^p (Mg)\,\mathrm{d}x\Big)^{\frac{1}{p}}
\end{equation}
for all $p>1$ and all $f,g\in L^1_{\text{loc}}(\R^d)$. This can be used to establish that $M:(X')^\theta\to (X')^\theta$ for any $0<\theta<1$. As a matter of fact, we have the following general duality results.
\begin{theorem}\label{cor:bfsdualweak}
Let $X$ be a quasi-Banach function space over $\R^d$ with the Fatou property. Suppose that $M:X\to X$. Then the following assertions hold:
\begin{enumerate}[(a)]
\item\label{it:mgeneralduality2} $M:(X')^\theta\to (X')^\theta$ for all $0<\theta<1$ with
\[
\|M\|_{(X')^\theta\to (X')^\theta}\lesssim_d \frac{1}{1-\theta} \|M\|^{\theta}_{X\to X};
\]
\item\label{it:mgeneralduality3} $M:[(X')^\theta]'\to [(X')^\theta]'$ for all $0<\theta<1$ with
\[
\|M\|_{[(X')^\theta]'\to [(X')^\theta]'}\lesssim_d\frac{1}{\theta}\|M\|_{X\to X};
\]
\item\label{it:mgeneralduality4} $(X')^\theta\in A_{\text{sparse}}$ for all $0<\theta<1$ with
\[
[(X')^\theta]_{A_{\text{sparse}}}\lesssim_d\frac{1}{\theta}\frac{1}{1-\theta}\|M\|_{X\to X}^{1+\theta}.
\]
\end{enumerate}
\end{theorem}
Qualitatively, these statements follow from extrapolation, see \cite[Theorem~4.16, Remark~4.17]{Ni23}. We will give an alternative  direct proof of \ref{it:mgeneralduality2} using \eqref{eq:feffstein}, and an alternative proof of \ref{it:mgeneralduality3} using Proposition~\ref{prop:apregularimpliesmbounds}. Note that \ref{it:mgeneralduality2} also follows from Proposition~\ref{prop:apregularimpliesmbounds}, but we would get a worse dependence on the operator norm.

\begin{proof}[Proof of Theorem~\ref{cor:bfsdualweak}]
For \ref{it:mgeneralduality2} and \ref{it:mgeneralduality3}, note that by Theorem~\ref{thm:mseconddual} we may assume that $X$ is a Banach function space by replacing $X$ by $X''$. Moreover, by the $3^d$-lattice theorem, it suffices to prove the result with $M$ replaced by $M^{\mc{D}}$ for a dyadic grid $\mc{D}$. We also set $p:=\tfrac{1}{\theta}>1$. 

To prove \ref{it:mgeneralduality2}, let $g\in (X')^\theta$. Then, by \eqref{eq:feffstein}, we have
\begin{align*}
\|Mg\|_{(X')^\theta}&=\sup_{\|f\|_X=1}\Big(\int_{\R^d}\!(Mg)^p|f|\,\mathrm{d}x\Big)^{\frac{1}{p}}
\lesssim_d p'\sup_{\|f\|_X=1}\Big(\int_{\R^d}\!|g|^p(Mf)\,\mathrm{d}x\Big)^{\frac{1}{p}}\\
&\leq p'\|M\|_{X\to X}^{\frac{1}{p}}\||g|^p\|^{\frac{1}{p}}_{X'}=\frac{1}{1-\theta}\|M\|_{X\to X}^{\theta}\|g\|_{(X')^\theta},
\end{align*}
as desired. 

For \ref{it:mgeneralduality3}, note that by \eqref{eq:lozdualitylebesgue} we have $[(X')^\theta]'=X^{\frac{1}{p}}\cdot L^{p'}(\R^d)$. Thus, the result follows from Proposition~\ref{prop:apregularimpliesmbounds} applied to the space $X^{\frac{1}{p}}$, using the observation that since $X$ is $A_1$-regular with constant $C_2\lesssim \|M\|_{X\to X}$ and for any $w\in A_1$ we have
\[
[w^{\frac{1}{p}}]^p_p\leq[w]_1
\]
the space $X^{\frac{1}{p}}$ is $A_p$-regular with $C_2\lesssim \|M\|_{X\to X}^{\frac{1}{p}}$.

Statement \ref{it:mgeneralduality4} follows from Theorem~\ref{thm:ApropsHL}\ref{it:mmuckenhoupt5} applied to the space $(X')^\theta$, combined with \ref{it:mgeneralduality2} and \ref{it:mgeneralduality3}. The assertion follows.
\end{proof}

We point out that \ref{it:mgeneralduality2} and \ref{it:mgeneralduality3} in Theorem~\ref{cor:bfsdualweak} imply that if $M:X\to X$, then, for any $0<\theta<1$, the space $Y:=(X')^\theta$ satisfies
\begin{equation}\label{eq:smallpowerssparseextrapolation}
M:Y\to Y,\quad M:Y'\to Y'
\end{equation}
This observation results in the following characterizations for boundedness of $M$ on a space and its K\"othe dual, which were originally proved in \cite{LN23a}:
\begin{corollary}
Let $X$ be a Banach function space over $\R^d$ with the Fatou property. Consider the statement
\begin{enumerate}[(i)]
\item\label{it:sicriteria0} $M:X\to X$ and  $M:X'\to X'$.
\end{enumerate}
If $X$ is $s_0$-concave for some $1<s_0<\infty$, then \ref{it:sicriteria0} is equivalent to
\begin{enumerate}[(i)]\setcounter{enumi}{1}
\item\label{it:sicriteria1} $M:[(X')^{s'}]'\to [(X')^{s'}]$ for some $s_0\leq s<\infty$.
\end{enumerate}
If $X$ is $r_0$-convex for some $1<r_0<\infty$, then \ref{it:sicriteria0} is equivalent to
\begin{enumerate}[(i)]\setcounter{enumi}{2}
\item\label{it:sicriteria2} $M:(X^r)'\to (X^r)'$ for some $1<r\leq r_0$.
\end{enumerate}
Moreover, in these situations we respectively have
\[
\|M\|_{X\to X}\lesssim_d s'\|M\|_{[(X')^{s'}]'\to [(X')^{s'}]'},\quad\|M\|_{X'\to X'}\lesssim_d s\|M\|_{[(X')^{s'}]'\to [(X')^{s'}]'}^{\frac{1}{s'}}
\]
and
\[
\|M\|_{X\to X}\lesssim_d r'\|M\|^{\frac{1}{r'}}_{(X^r)'\to (X^r)'},\quad\|M\|_{X'\to X'}\lesssim_d r\|M\|_{(X^r)'\to (X^r)'}
\]
\end{corollary}
\begin{proof}
For \ref{it:sicriteria1}$\Rightarrow$\ref{it:sicriteria0}, apply Theorem~\ref{cor:bfsdualweak} to the space $[(X')^{s'}]'$ with $\theta:=\frac{1}{s'}$. For \ref{it:sicriteria2}$\Rightarrow$\ref{it:sicriteria0}, apply Theorem~\ref{cor:bfsdualweak} to the space $(X^r)'$ with $\theta:=\frac{1}{r}$.

For \ref{it:sicriteria0}$\Rightarrow$\ref{it:sicriteria2}, by Theorem~\ref{thm:ApropsHL}\ref{it:mmuckenhoupt5} we have $X\in A_{\text{sparse}}$. Then by Theorem~\ref{thm:sparseselfimprovement} we can pick an $1<r\leq r_0$ for which $X^r\in A_{\text{sparse}}$. Another application of Theorem~\ref{thm:ApropsHL}\ref{it:mmuckenhoupt5} now proves \ref{it:sicriteria2}. By symmetry, the implication \ref{it:sicriteria0}$\Rightarrow$\ref{it:sicriteria1} can be proved analogously, replacing $X$ by $X'$ and $r$ by $s'$. The result follows.
\end{proof}
Thus, in relation to \eqref{Q}, we find that if $M:X\to X$ and $X$ is $s_0$-concave for some $1<s_0<\infty$, then $M:X'\to X'$ if and only if $M:[(X')^{s'}]'\to [(X')^{s'}]'$ for some $s_0\leq s<\infty$.

Next, we provide a list featuring several known characterizations of \eqref{Q}. For this, we define the sharp maximal operator
\[
M^\sharp f=\sup_{Q}\Big(\frac{1}{|Q|}\int_Q\!|f-\langle f\rangle_Q|\,\mathrm{d}x\Big)\ind_Q,\quad \langle f\rangle_Q=\frac{1}{|Q|}\int_Q\!f\,\mathrm{d}x.
\]
\begin{theorem}\label{thm:Eintext}
Let $X$ be a Banach function space with the Fatou property. Suppose that $M:X\to X$. Then the following are equivalent:
\begin{enumerate}[(i)]
    \item\label{it:Eintext1} $M:X'\to X'$;
    \item\label{it:Eintext2} $X'$ is $A_{\text{FW}}$-regular;
    \item\label{it:Eintext3} there is a $C>0$ such that for all $f\in X$ satisfying $|\{x\in\R^d:|f(x)|>\lambda\}|<\infty$ for all $\lambda>0$ we have
    \[
    \|f\|_X\leq C\|M^{\sharp} f\|_X;
    \]
    \item\label{it:Eintext4} $M:(X^\theta)'\to (X^\theta)'$ for some $0<\theta<1$;
\end{enumerate}
\end{theorem}
The equivalence of \ref{it:Eintext1}$\Leftrightarrow$\ref{it:Eintext2} follows from \cite[Proposition~7]{Ru19}, where Rutsky proves that if a Banach function space $X$ over $\R^d$ is $A_{\text{FW}}$-regular and there is a $0<\theta<1$ for which $M:X^\theta\to X^\theta$, then also $M:X\to X$. Since $M:X\to X$ implies that $M:(X')^\theta\to (X')^\theta$ for all $0<\theta<1$ by Theorem~\ref{cor:bfsdualweak}, this proves the assertion.

The characterization \ref{it:Eintext1}$\Leftrightarrow$\ref{it:Eintext3} is contained in the work of Lerner \cite[Corollary~4.3]{Le10c}. Perhaps interestingly, one can also prove the implication \ref{it:Eintext2}$\Rightarrow$\ref{it:Eintext3} directly by using the weighted Fefferman-Stein inequality
\[
\|f\|_{L^1_w(\R^d)}\lesssim \|M^\sharp f\|_{L^1_w(\R^d)},
\]
valid for any $w\in A_{\text{FW}}$. As a matter of fact, since this bound is true for any weight $w\in C_p$ for some $p>1$, one can weaken \ref{it:Eintext2} even further.

The equivalence \ref{it:Eintext1}$\Leftrightarrow$\ref{it:Eintext4} can be shown using a different characterization of Lerner obtained in the very recent note \cite{Le24}. Here, he shows that \ref{it:Eintext1} is equivalent to the unboundedness of a certain function
\[
\phi_X:(0,1)\to[1,\infty)
\]
that satisfies $\phi_{X^p}=(\phi_X)^p$ for any $p>0$. In particular, this means that \ref{it:Eintext1}$\Leftrightarrow$ $\phi_X$ is unbounded $\Leftrightarrow$ $\phi_{X^\theta}$ is unbounded. Since $M:X\to X$ also implies $M:X^\theta\to X^\theta$, this last assertion is equivalent to \ref{it:Eintext4}, as desired.

As a consequence, we can show that Conjecture~\ref{con:mduality1} and Conjecture~\ref{con:mduality2} are equivalent.
\begin{corollary}\label{cor:equivalentconjectures}
Conjecture~\ref{con:mduality1} is true if and only if Conjecture~\ref{con:mduality2} is.
\end{corollary}
\begin{proof}
As explained in the introduction, Conjecture~\ref{con:mduality2} implies Conjecture~\ref{con:mduality1}. For the converse, suppose $X$ is an $s$-concave Banach function space over $\R^d$ for which $M:X\to X$. Since $X^{\frac{1}{2}}$ is $2$-convex, $2s$-concave, and $M:X^{\frac{1}{2}}\to X^{\frac{1}{2}}$, the validity of Conjecture~\ref{con:mduality1} implies that also $M:(X^{\frac{1}{2}})'\to (X^{\frac{1}{2}})'$. Since $X^{\frac{1}{2}}$ is reflexive, it satisfies the Fatou property. Hence, so does $X$. Thus, by Theorem~\ref{thm:Eintext}, we conclude that $M:X'\to X'$. The result follows.
\end{proof}

Finally, we give a criterion for when the implication in \eqref{Q} fails.
\begin{proposition}
Let $X$ be a quasi-Banach function space over $\R^d$ for which $X'\neq \{0\}$. If $\ind_{\R^d}\in X$, then we do not have $M:X'\to X'$.
\end{proposition}
\begin{proof}
If $M:X'\to X'$, then for all $g\in X'$ we have $Mg=(Mg)\ind_{\R^d}\in L^1(\R^d)$, which is only possible when $X'=\{0\}$. The result follows by contraposition.
\end{proof}
This result applies, in particular, when $M:X\to X$, since then $X'$ is saturated. Thus, we get the following corollary:
\begin{corollary}\label{cor:failuresufficientcondition}
Let $X$ be a quasi-Banach function space over $\R^d$ for which $M:X\to X$ and $\ind_{\R^d}\in X$. Then we do not have $M:X'\to X'$.
\end{corollary}

\section{Examples of classes of spaces}\label{sec:applications}
\subsection{Unweighted variable Lebesgue spaces} Given $1\leq p<\infty$, we define
\[
\phi_p(t):=\tfrac{1}{p}t^p,\quad \phi_\infty(t):=\infty\ind_{(1,\infty)}(t).
\]
For an exponent function $p:\R^d\to [1,\infty]$ and
\[
\rho_{p(\cdot)}(f):=\int_{\R^d}\!\phi_{p(x)}(|f(x)|)\,\mathrm{d}x,
\]
the variable Lebesgue space $L^{p(\cdot)}(\R^d)$ is defined as the space of $f\in L^0(\R^d)$ for which there is a $\lambda>0$ such that $\rho_{p(\cdot)}(\lambda^{-1} f)<\infty$, with the Luxemburg norm
\[
\|f\|_{L^{p(\cdot)}(\R^d)}:=\inf\{\lambda>0:\rho_{p(\cdot)}(\lambda^{-1} f)\leq 1\}.
\]
Then $L^{p(\cdot)}(\R^d)$ is a Banach function space over $\R^d$ with the Fatou property, and 
\[
L^{p(\cdot)}(\R^d)'=L^{p'(\cdot)}(\R^d),
\]
where $\tfrac{1}{p'(x)}:=1-\tfrac{1}{p(x)}$. Setting
\[
p_-:=\essinf_{x\in\R^d}p(x),\quad p_+:=\esssup_{x\in\R^d}p(x),
\]
the space $L^{p(\cdot)}(\R^d)$ is $p_-$-convex, and $p_+$-concave. For proofs and further properties of these spaces we refer the reader to \cite{DHHR11, CF13}.

The following is a deep result by Diening \cite[Theorem~8.1]{Di05}, which implies that Conjecture~\ref{con:mduality1} is true in the class of variable Lebesgue spaces.
\begin{theorem}\label{thm:dieningduality}
Let $p:\R^d\to [1,\infty]$ with $1<p_-\leq p_+<\infty$. Then the following are equivalent:
\begin{enumerate}[(i)]
    \item $M:L^{p(\cdot)}(\R^d)\to L^{p(\cdot)}(\R^d)$;
    \item $L^{p(\cdot)}(\R^d)\in A_{\text{strong}}$.
\end{enumerate}
\end{theorem}
Since $L^{p(\cdot)}(\R^d)\in A_{\text{strong}}$ if and only if $L^{p'(\cdot)}(\R^d)\in A_{\text{strong}}$ by Proposition~\ref{prop:bfspropertiesduality}, this result directly implies that
\[
M:L^{p(\cdot)}(\R^d)\to L^{p(\cdot)}(\R^d)\quad \Leftrightarrow \quad M:L^{p'(\cdot)}(\R^d)\to L^{p'(\cdot)}(\R^d),
\]
verifying Conjecture~\ref{con:mduality1} in this class of spaces.

As a matter of fact, Conjecture~\ref{con:mduality2} is also true in this setting. Indeed, \cite[Theorem~4.7.1.]{DHHR11} states that if $M:L^{p(\cdot)}(\R^d)\to L^{p(\cdot)}(\R^d)$, then we must have $p_->1$. Thus, since $L^{p(\cdot)}(\R^d)$ is $s$-concave for some $s<\infty$ if and only if $p_+<\infty$, we are in the setting of Theorem~\ref{thm:dieningduality}, proving that $M:L^{p'(\cdot)}(\R^d)\to L^{p'(\cdot)}(\R^d)$, as desired.

We note that a different characterization in terms of the exponent function $p(\cdot)$ was obtained by Lerner in \cite{Le23}, which yields an alternative proof of this duality result. The question of whether Theorem~\ref{thm:dieningduality} is valid even if $p_+=\infty$ is still open.

As shown in \cite[Theorem~5.3.4]{DHHR11}, it is not true that $L^{p(\cdot)}(\R^d)\in A$ implies that $M:L^{p(\cdot)}(\R^d)\to L^{p(\cdot)}(\R^d)$. However, this boundedness is true under certain regularity conditions on the exponent function $p:\R^d\to[1,\infty]$. Following \cite[Section~4.1]{DHHR11}, we say that $p$ is globally $\log$-H\"older continuous, if it is locally $\log$-H\"older continuous, i.e., there is a $C_1>0$ such that
\[
\Big|\frac{1}{p(x)}-\frac{1}{p(y)}\Big|\leq\frac{C_1}{\log(e+\tfrac{1}{|x-y|})}
\]
for all $x,y\in\R^d$, and it satisfies the $\log$-H\"older decay condition, i.e., there exist $C_2>0$, $p_\infty\in[1,\infty]$ such that
\[
\Big|\frac{1}{p(x)}-\frac{1}{p_\infty}\Big|\leq\frac{C_2}{\log(e+|x|)}
\]
for all $x\in\R^d$.
\begin{theorem}\label{thm:unweightedvariablelebesguelogholder}
Suppose $p:\R^d\to[1,\infty]$ is globally $\log$-H\"older continuous. Then the following assertions hold:
\begin{enumerate}[(a)]
    \item $L^{p(\cdot)}(\R^d)\in A_{\text{strong}}$ with $[X]_{A_{\text{strong}}}\lesssim_{d,C_1,C_2} 1$;
    \item\label{it:unweightedvariablelebesguelogholder2} If $p_->1$, then $M:L^{p(\cdot)}(\R^d)\to L^{p(\cdot)}(\R^d)$ with
\[
\|Mf\|_{L^{p(\cdot)}(\R^d)\to L^{p(\cdot)}(\R^d)}\lesssim_{d,C_1,C_2} (p_-)'\|f\|_{L^{p(\cdot)}(\R^d)}.
\]
\end{enumerate}
\end{theorem}
The first result is proved in  \cite[Theorem~4.4.8]{DHHR11}. Note that by Theorem~\ref{thm:ApropsHL}\ref{it:mmuckenhoupt2} this also implies that
\[
M:L^{p(\cdot)}(\R^d)\to L^{p(\cdot)}(\R^d)_{\text{weak}}.
\]
The second result was proved in \cite{CFN03}. See also \cite[Theorem~4.3.8]{DHHR11}, and the paragraph that follows it for the history of this problem.

In \cite{Ne04}, Nekvinda proved that Theorem~\ref{thm:unweightedvariablelebesguelogholder}\ref{it:unweightedvariablelebesguelogholder2} still holds when the the log-H\"older decay condition is relaxed to the weaker condition
\[
1\in L^{s(\cdot)}(\R^d),\quad \frac{1}{s(x)}:=\Big|\frac{1}{p(x)}-\frac{1}{p_\infty}\Big|.
\]
We will denote this condition by $L^{p(\cdot)}(\R^d)\in\mc{N}$.

Outside of $\log$-H\"older regularity, it was shown by Kopaliani in \cite{Ko07} that if $p:\R^d\to[1,\infty]$ satisfies $1<p_-\leq p_+<\infty$ and is constant outside of a bounded set, then actually $L^{p(\cdot)}(\R^d)\in A$ suffices to conclude that $M:L^{p(\cdot)}(\R^d)\to L^{p(\cdot)}(\R^d)$. This result was further improved in \cite[Theorem~4.52]{CF13} where they showed that it suffices to assume that $L^{p(\cdot)}(\R^d)\in A\cap\mc{N}$. In the very recent preprint \cite{ADK25} of Adamadze, Diening, and Kopaliani, it was shown that the condition $p_+<\infty$ can also be removed.
\begin{theorem}\label{thm:nekvindadecay}
Let $p:\R^d\to[1,\infty]$ with $L^{p(\cdot)}(\R^d)\in \mc{N}$. Then the following are equivalent:
\begin{enumerate}[(i)]
    \item $M:L^{p(\cdot)}(\R^d)\to L^{p(\cdot)}(\R^d)$;
    \item $p_->1$ and $L^{p(\cdot)}(\R^d)\in A$.
\end{enumerate}
\end{theorem}
In particular, using Theorem~\ref{thm:ApropsHL} we conclude that for an exponent function $p(\cdot)$ with $p_->1$ and $L^{p(\cdot)}(\R^d)\in \mc{N}$, the boundedness $M:L^{p(\cdot)}(\R^d)\to L^{p(\cdot)}(\R^d)$ is characterized by any of the following conditions:
\begin{itemize}
    \item $M:L^{p(\cdot)}(\R^d)\to L^{p(\cdot)}(\R^d)_{\text{weak}}$;
    \item $L^{p(\cdot)}(\R^d)\in A_{\text{strong}}$;
    \item $L^{p(\cdot)}(\R^d)\in A$.
\end{itemize}
And thus, by Proposition~\ref{prop:bfspropertiesduality}, also by any of these statements with $L^{p(\cdot)}(\R^d)$ replaced by the K\"othe dual $L^{p'(\cdot)}(\R^d)$.

The condition $L^{p(\cdot)}(\R^d)\in\mc{N}$ is sufficient, but not necessary for Theorem~\ref{thm:nekvindadecay} to hold. Even under the condition $p_+<\infty$, it is an open problem to find conditions that are easier to check than Diening's characterization $L^{p(\cdot)}(\R^d)\in A_{\text{strong}}$ from \cite{Di05}, or Lerner's characterization from \cite{Le23}.

Global $\log$-H\"older continuity implies the structural property $X\in\mc{G}$ of Theorem~\ref{thm:aequivastrong} related to the equivalence of the Muckenhoupt condition and the strong Muckenhoupt condition. Defining
\[
\langle f\rangle_{p(\cdot),Q}:=\frac{\|f\ind_Q\|_{L^{p(\cdot)}(\R^d)}}{\|\ind_Q\|_{L^{p(\cdot)}(\R^d)}},
\]
the following result follows from Theorem~\ref{thm:aequivastrong} and \cite[Theorem~7.3.22]{DHHR11}.
\begin{theorem}\label{thm:aequivstronglogholder}
Suppose $p:\R^d\to[1,\infty]$ is globally $\log$-H\"older continuous. Then for every pairwise disjoint collection of cubes $\mc{P}$ and any $f\in L^{p(\cdot)}(\R^d)$ supported in $\bigcup_{Q\in\mc{P}}Q$ we have
\[
\Big\|\sum_{Q\in\mc{P}}\langle f\rangle_{p(\cdot),Q}\ind_Q\Big\|_{L^{p(\cdot)}(\R^d)}\eqsim_{d,C_1,C_2}\|f\|_{L^{p(\cdot)}(\R^d)}.
\]
\end{theorem}

\subsection{Weighted variable Lebesgue spaces}\label{subsec:variablelebesgue}
Given a weight $w$, we define $L^{p(\cdot)}_w(\R^d)$ through
\[
\|f\|_{L^{p(\cdot)}_w(\R^d)}:=\|fw\|_{L^{p(\cdot)}(\R^d)}.
\]
It was shown in \cite{Le16b} that Diening's duality theorem still holds in weighted variable Lebesgue spaces under the assumption that $w^{p(\cdot)}\in A_{\text{FW}}$.
\begin{theorem}
Let $p:\R^d\to [1,\infty]$ with $1<p_-\leq p_+<\infty$, and let $w$ be a weight for which $w^{p(\cdot)}\in A_{\text{FW}}$. If $M:L^{p(\cdot)}_w(\R^d)\to L^{p(\cdot)}_w(\R^d)$, then $M:L^{p'(\cdot)}_{w^{-1}}(\R^d)\to L^{p'(\cdot)}_{w^{-1}}(\R^d)$.
\end{theorem}
This leaves open the question of whether we can remove the condition $w^{p(\cdot)}\in A_{\text{FW}}$ or not. Should Conjecture~\ref{con:mduality1} be true, then this condition can be removed. If the exponent function is globally $\log$-H\"older continuous, then it is shown in \cite[Lemma~3.4]{CFN12} that if $L^{p(\cdot)}_w(\R^d)\in A$, then $w^{p(\cdot)}\in A_{\text{FW}}$, and thus, in this case, this condition is superfluous.

In view of Theorem~\ref{thm:dieningduality}, we pose the following conjecture.
\begin{conjecture}
Let $p:\R^d\to [1,\infty]$ with $1<p_-\leq p_+<\infty$, and let $w$ be a weight for which $w^{p(\cdot)}\in A_{\text{FW}}$. Then the following are equivalent:
\begin{enumerate}[(i)]
    \item $M:L_w^{p(\cdot)}(\R^d)\to L_w^{p(\cdot)}(\R^d)$;
    \item $L_w^{p(\cdot)}(\R^d)\in A_{\text{strong}}$.
\end{enumerate}
\end{conjecture}

Global $\log$-H\"older continuity serves as a sufficient condition for the equivalence of $L^{p(\cdot)}_w(\R^d)\in A$ and boundedness of the maximal operator. In this setting, the Muckenhoupt condition is given by
\[
[L^{p(\cdot)}_w(\R^d)]_A=\sup_Q\|w\ind_Q\|_{L^{p(\cdot)}(\R^d)}\|w^{-1}\ind_Q\|_{L^{p'(\cdot)}(\R^d)},
\]
and is usually written as $w\in A_{p(\cdot)}$. The following result follows from a combination of Corollary~\ref{cor:aequivastrongweaktype}, and several observations and results from \cite{CDH11}.
\begin{theorem}\label{thm:logholderweightedequivalences}
Suppose $p:\R^d\to[1,\infty]$ is globally $\log$-H\"older continuous. Then the following are equivalent:
\begin{enumerate}[(i)]
    \item\label{it:logholderweightedvariablelebegue1} $M:L^{p(\cdot)}_w(\R^d)\to L^{p(\cdot)}_w(\R^d)_{\text{weak}}$;
    \item\label{it:logholderweightedvariablelebegue1.5} $L^{p(\cdot)}_w(\R^d)\in A_{\text{strong}}$;
    \item\label{it:logholderweightedvariablelebegue2}  $L^{p(\cdot)}_w(\R^d)\in A$.
\end{enumerate}
If $1<p_-\leq p_+<\infty$, then, additionally, the above conditions are equivalent to
\begin{enumerate}[(i)]\setcounter{enumi}{3}
\item $M:L^{p(\cdot)}_w(\R^d)\to L^{p(\cdot)}_w(\R^d)$.
\end{enumerate}
\end{theorem}
The equivalence of \ref{it:logholderweightedvariablelebegue1} and \ref{it:logholderweightedvariablelebegue2} was originally claimed in \cite[Theorem~1.5]{CFN12}, but as pointed out in the paragraph below \cite[Conjecture~1.4]{CC22}, there is a gap in that proof. We provide a correct proof here, settling this conjecture.
\begin{proof}[Proof of Theorem~\ref{thm:logholderweightedequivalences}]
The equivalence of the last statement with \ref{it:logholderweightedvariablelebegue2} is proved in \cite[Theorem~1.3]{CDH11}. For the equivalence of \ref{it:logholderweightedvariablelebegue1}-\ref{it:logholderweightedvariablelebegue2}, we first note that by \cite[Theorem~7.3.22]{DHHR11}, global $\log$-H\"older continuity implies that $L^{p(\cdot)}(\R^d)\in\mc{G}$. Following the observation of \cite[Lemma~2.1]{CDH11} (which notes that if $X\in\mc{G}$ holds in the unweighted space, then it automatically also hold in the weighted space by applying the unweighted estimate to $fw$, $gw^{-1}$ instead of $f$, $g$), we conclude that for any weight $w$ we have $L^{p(\cdot)}_w(\R^d)\in\mc{G}$ with 
\[
[L^{p(\cdot)}_w(\R^d)]_{\mc{G}}=[L^{p(\cdot)}(\R^d)]_{\mc{G}}.
\]
Thus, the result follows from Corollary~\ref{cor:aequivastrongweaktype}.
\end{proof}

\subsection{Musielak-Orlicz spaces}\label{subsec:musielakorlicz}
We call a function $\phi:\R^d\times[0,\infty)\to[0,\infty]$ a \emph{generalized $\Phi$-function} if
\begin{itemize}
    \item For a.e. $x\in\R^d$ the function $\phi(x,\cdot)$ is increasing, convex, left-continuous, and 
    \[
    \phi(x,0)=0, \quad
    \lim_{t\downarrow 0}\phi(x,t)=0,\quad \lim_{t\to\infty}\phi(x,t)=\infty;
    \]
    \item For any $t\geq 0$, the function $\phi(\cdot,t)$ is measurable.
\end{itemize}
Let $\phi$ be generalized $\Phi$-function and set
\[
\rho_{\phi(\cdot)}(f):=\int_{\R^d}\!\phi(x,|f(x)|)\,\mathrm{d}x
\]
for $f\in L^0(\R^d)$. We note that this is well-defined, since $x\mapsto \phi(x,|f(x)|)$ is measurable by \cite[Theorem~2.5.4]{HH19}. We define the \emph{Musielak-Orlicz space} $L^{\phi(\cdot)}(\R^d)$ (also known as a \emph{generalized Orlicz space}) as the space of $f\in L^0(\R^d)$ for which
\[
\|f\|_{L^{\phi(\cdot)}(\R^d)}:=\inf\{\lambda>0:\rho_{\phi(\cdot)}(\lambda^{-1}f)\leq 1\}<\infty.
\]
The space $L^{\phi(\cdot)}(\R^d)$ is a Banach function space with the Fatou property. The K\"othe dual of $L^{\phi(\cdot)}(\R^d)$ is given (up to an equivalent norm) by $L^{\phi^\ast(\cdot)}(\R^d)$, where $\phi^\ast$ is defined so that for a.e. $x\in\R^d$, $\phi^\ast(x,\cdot)$ is the convex conjugate of $\phi(x,\cdot)$, i.e.,
\[
\phi^\ast(x,t):=\sup_{u\geq 0}\{ut-\phi(x,u)\}.
\]
Given two generalized $\Phi$-functions $\phi,\psi$, we write $\phi\lesssim\psi$ if there are $\lambda,K>0$ and a $0\leq h\in L^1(\R^d)$ such that for a.e. $x\in\R^d$ and all $t\geq 0$
\[
\phi(x,\lambda t)\leq K\psi(x,t)+h(x).
\]
Moreover, we say that $\phi$ and $\psi$ are equivalent and write $\phi\sim\psi$, if both $\phi\lesssim\psi$ and $\psi\lesssim\phi$. This is the case precisely when $L^{\phi(\cdot)}(\R^d)=L^{\psi(\cdot)}(\R^d)$ with equivalent norm, see \cite[Theorem~3.2.6]{HH19}.

These properties and definitions in a more general setting can be found in the book \cite{HH19} by Harjulehto and H\"ast\"o. Moreover, we refer the reader to \cite{DHHR11}.
\begin{example}\label{ex:youngfunctionvariablelebesgue}
If $w$ is a weight and $p:\R^d\to[1,\infty]$ a measurable function, then
\[
\phi(x,t):=\phi_{p(x)}(w(x)t)=\tfrac{1}{p(x)}(w(x)t)^{p(x)}
\]
is a generalized $\Phi$-function satisfying
\[
\phi^\ast(x,t)=\phi_{p'(x)}(w(x)^{-1}t)=\tfrac{1}{p'(x)}(w(x)^{-1}t)^{p'(x)},
\]
and $L^{\phi}(\R^d)=L^{p(\cdot)}_w(\R^d)$, $L^{\phi^\ast}(\R^d)=L^{p'(\cdot)}_{w^{-1}}(\R^d)$. In particular, the class of Musielak-Orlicz spaces contains the class of weighted variable Lebesgue spaces. 
\end{example}
We say that a generalized $\Phi$-function $\phi$ satisfies the $\Delta_2$ condition (as was originally introduced in \cite{Mu83}), if there exists a constant $K>0$ and a function $0\leq h\in L^1(\R^d)$ such that for a.e. $x\in\R^d$ and all $t\geq 0$ we have
\[
\phi(x,2t)\leq K\phi(x,t)+h(x).
\]
Often, the $\Delta_2$ condition is defined through this inequality only when $h=0$. By passing to an equivalent generalized $\Phi$-function, we may assume without loss of generality that this is the case. Indeed, assuming $K>1$, we can set
\[
\psi(x,t):=(K-1)\phi(x,t)+h(x).
\]
Then $\phi\sim\psi$, and
\[
\psi(x,2t)\leq K(K-1)\phi(x,t)+Kh(x)=K\psi(x,t),
\]
as desired. If both $\phi$ and $\phi^\ast$ satisfy the $\Delta_2$ condition, we can also find a generalized $\Phi$-function $\psi\sim\phi$ for which $\psi$ and $\psi^\ast$ satisfy the $\Delta_2$ condition both with $h=0$, see \cite[Theorem~2.5.24]{HH19}.
\begin{example}
If $\phi$ is as in Example~\ref{ex:youngfunctionvariablelebesgue}, then $\phi$ satisfies the $\Delta_2$ condition precisely when $p_+<\infty$, in which case $K=2^{p_+}$. For a proof, see \cite[Example~2.3]{LVY18}. 
\end{example}
The above example suggests a connection between the $\Delta_2$ condition of $\phi$ and $s$-concavity of the space $L^{\phi(\cdot)}(\R^d)$. For $0<s<\infty$, we say that $\phi$ satisfies the $\Delta^s$ condition if there is a constant $K>0$ and a $0\leq h\in L^1(\R^d)$ for which for a.e. $x\in\R^d$, all $t\geq 0$, and all $\lambda\geq 1$ we have
\[
\lambda^{-s}\phi(x,\lambda t)\leq K\phi(x,t)+h(x).
\]
Then we have the following result.
\begin{theorem}\label{thm:delta2conditionequivalence}
Let $\phi$ be a generalized $\Phi$-function. Then the following are equivalent:
\begin{enumerate}[(i)]
    \item $\phi$ satisfies the $\Delta_2$ condition;
    \item $\phi$ satisfies the $\Delta^s$ condition for some $1\leq s<\infty$;
    \item $L^{\phi}(\R^d)$ is $s$-concave for some $1\leq s<\infty$.
\end{enumerate}
In particular, $L^{\phi(\cdot)}(\R^d)$ is $r$-convex and $s$-concave for some $1<r\leq s<\infty$ if and only if $\phi$ and $\phi^\ast$ satisfy the $\Delta_2$ condition.
\end{theorem}
\begin{proof}
It is shown in \cite[Proposition~1]{Ka98} that the $\Delta_2$ condition is equivalent to the $\Delta^s$ condition for some $s<\infty$. It is then shown in \cite[Proposition~2, Proposition~4]{Ka98} that the $\Delta^s$ condition is equivalent to $s$-concavity of $L^{\phi(\cdot)}(\R^d)$. This proves the first assertion. For the second one, note that $L^{\phi(\cdot)}(\R^d)$ is $r$-convex for $r>1$ if and only if $L^{\phi^\ast}(\R^d)$ is $r'$-concave. By the first result, this is equivalent to the $\Delta_2$ condition of $\phi^\ast$. The result follows.
\end{proof}

By the above theorem, Conjecture~\ref{con:mduality1} takes the following form.
\begin{conjecture}
Let $\phi$ be a generalized $\Phi$-function for which both $\phi$ and $\phi^\ast$ satisfy the $\Delta_2$ condition. Then $M:L^{\phi(\cdot)}(\R^d)\to L^{\phi(\cdot)}(\R^d)$ if and only if $M:L^{\phi^\ast(\cdot)}(\R^d)\to L^{\phi^\ast(\cdot)}(\R^d)$.
\end{conjecture}
An equivalent conjecture can be found as \cite[Question~4.3.7]{HH19} where it is asked if this statement holds under the condition that $\phi$ and $\phi^\ast$ satisfy an almost increasing condition. By \cite[Corollary~2.4.11]{HH19}, this almost increasing condition of $\phi$ is equivalent to the $\Delta_2$ condition of $\phi^\ast$, showing that these conjectures are indeed equivalent.

Several results in the generality of Musielak-Orlicz spaces related to the Muckenhoupt condition exist in the literature. We give an overview of several results here.

Given a generalized $\Phi$-function $\phi$, we set
\[
\phi^{-1}(x,t):=\inf\{u\geq 0:\phi(x,u)\geq t\}.
\]
We say that $\phi$ satisfies the $(A0)$ condition if there is a $0<\beta\leq 1$ for which
\begin{equation}\tag{A0}
\beta\leq\phi^{-1}(x,1)\leq\tfrac{1}{\beta}
\end{equation}
for a.e. $x\in\R^d$. We say that $\phi$ satisfies the $(A1)$ condition if there is a $0<\beta<1$ for which for every ball $B\subseteq\R^d$ with $|B|\leq 1$ we have
\begin{equation}\tag{A1}
\beta\phi^{-1}(x,t)\leq\phi^{-1}(y,t)
\end{equation}
for a.e. $x,y\in B$ for all $0\leq t\leq\tfrac{1}{|B|}$. We say that $\phi$ satisfies the $(A2)$ condition if for all $s>0$ there is a $0<\beta\leq 1$ and $h\in L^1(\R^d)\cap L^\infty(\R^d)$ such that
\begin{equation}\tag{A2}
\beta\phi^{-1}(x,t)\leq\phi^{-1}(y,t)
\end{equation}
for a.e. $x,y\in\R^d$ for all $h(x)+h(y)\leq t\leq s$.

It was recently shown in \cite{HHS24} that the conditions $(A0)$ and $(A2)$ together are equivalent to the condition that for all $s>0$ there is a $0<\beta\leq 1$ and $h\in L^1(\R^d)\cap L^\infty(\R^d)$ such that
\[
\beta\phi^{-1}(x,t)\leq\phi^{-1}(y,t+h(x)+h(y))
\]
for a.e. $x,y\in\R^d$ for all $0\leq t\leq s$.

To understand the conditions $(A0)$--$(A2)$, we remark on which conditions they correspond to in the variable Lebesgue case.
\begin{example}\label{ex:youngfunctionlogholder}
If $\phi$ is as in Example~\ref{ex:youngfunctionvariablelebesgue}, then
\[
\phi^{-1}(x,t)=p(x)^{\frac{1}{p(x)}} w(x)^{-1}t^{\frac{1}{p(x)}}\eqsim w(x)^{-1}t^{\frac{1}{p(x)}},
\]
so $\phi$ satisfies the $(A0)$ condition if and only if $w\eqsim 1$.

It is remarked in \cite[Section~4]{Ha15} that the $(A1)$ condition fulfills the role of local $\log$-H\"older regularity condition of $p(\cdot)$, and the $(A2)$ condition that of the Nekvinda decay condition $L^{p(\cdot)}(\R^d)\in\mc{N}$.
\end{example}

The following result is \cite[Theorem~4.3.4]{HH19}, which was first established in a slightly less general form by H\"ast\"o in \cite{Ha15}.
\begin{theorem}
Let $\phi$ be a generalized $\Phi$-function for which $\phi^\ast$ satisfies the $\Delta_2$ condition, and $\phi$ satisfies the conditions $(A0)$--$(A2)$. Then
\[
M:L^{\phi(\cdot)}(\R^d)\to L^{\phi(\cdot)}(\R^d).
\]
\end{theorem}
By Theorem~\ref{thm:delta2conditionequivalence}, the $\Delta_2$ condition of $\phi^\ast$ is equivalent to the assumption that $L^{\phi(\cdot)}(\R^d)$ is $r$-convex for some $r>1$. In the variable Lebesgue case, this corresponds to the condition $p_->1$, explaining its appearance in the above theorem. In view of Example~\ref{ex:youngfunctionlogholder}, this result should be compared to Theorem~\ref{thm:nekvindadecay} in the variable Lebesgue space setting. 

Comparing the result to Theorem~\ref{thm:logholderweightedequivalences}, one might wonder if a weighted variant of this result is true, i.e., if $M$ is bounded on $L^{\phi_w(\cdot)}(\R^d)$, where
\[
\phi_w(x,t):=\phi(x,w(x)t),
\]
if we additionally assume that $L^{\phi_w(\cdot)}(\R^d)\in A$. Note that now
\[
\|f\|_{L^{\phi_w(\cdot)}(\R^d)}=\|wf\|_{L^{\phi(\cdot)}(\R^d)}.
\]
We record this as a conjecture here.
\begin{conjecture}
Let $\phi$ be a generalized $\Phi$-function for which $\phi^\ast$ satisfies the $\Delta_2$ condition, and $\phi$ satisfies $(A0)$--$(A2)$. If $w$ is a weight for which $L^{\phi_w(\cdot)}(\R^d)\in A$, then
\[
M:L^{\phi_w(\cdot)}(\R^d)\to L^{\phi_w(\cdot)}(\R^d).
\]
\end{conjecture}
If true, the proof could be used to improve the final assertion of Theorem~\ref{thm:logholderweightedequivalences} by removing the condition $p_+<\infty$ and by replacing the $\log$-H\"older decay condition by the Nekvinda decay condition $L^{p(\cdot)}(\R^d)\in\mc{N}$. In a similar vein, one can ask the following question.
\begin{question}
Let $\phi$ be a generalized $\Phi$-function satisfying conditions $(A0)$--$(A2)$. Is it true that $L^{\phi(\cdot)}(\R^d)\in\mc{G}$?
\end{question}
If yes, this could also extend the known result for variable Lebesgue spaces to those with local $\log$-H\"older regularity and the Nekvinda decay condition. If no, perhaps these results are true under a stricter condition than $(A2)$, more closely related to the $\log$-H\"older decay condition.

To end this subsection, we comment on the results without the assumptions $(A0)$-$(A2)$ by Diening in \cite{Di05}. Just like in the variable Lebesgue case, he considered the condition $L^{\phi(\cdot)}(\R^d)\in A_{\text{strong}}$ for an \emph{$N$-function} $\phi$, i.e., a generalized $\Phi$-function satisfying the additional properties that for a.e. $x\in\R^d$ we have
\[
\lim_{t\downarrow 0}\tfrac{\phi(x,t)}{t}=0,\quad \lim_{t\to\infty}\tfrac{\phi(x,t)}{t}=\infty.
\]
These conditions exclude, e.g., the function $\phi(x,t):=\phi_p(t)$ for $p=1$ or $p=\infty$. While he was unable to prove a characterization of the boundedness of $M$ in $L^{\phi(\cdot)}(\R^d)$, he does present partial results. 

In \cite[Theorem~5.7]{Di05} he establishes that $L^{\phi(\cdot)}(\R^d)\in A_{\text{strong}}$ is equivalent to a certain reverse H\"older estimate that generalizes the reverse H\"older estimate for Muckenhoupt weights $w\in A_p$. This leaves us with the question of whether or not $L^{\phi(\cdot)}(\R^d)\in A_{\text{strong}}$ implies the existence of an $r>1$ for which also $L^{\phi(\cdot)}(\R^d)^r\in A_{\text{strong}}$, see Question~\ref{que:G}. We note here that
\[
L^{\phi(\cdot)}(\R^d)^r=L^{\phi_r(\cdot)}(\R^d),
\]
where $\phi_r(x,t):=\phi(x,t^{\frac{1}{r}})$. Moreover, in \cite[Theorem~6.4]{Di05} he shows that if a certain stronger reverse H\"older condition holds, then there is an $r>1$ for which $M$ is bounded on $L^{\phi(\cdot)}(\R^d)^r$.

\subsection{Morrey spaces}
For $1\leq p\leq q\leq\infty$ we define the Morrey space $M^{p,q}(\R^d)$ as the space of $f\in L^0(\R^d)$ for which there is a $C>0$ such that for all cubes $Q$ we have
\[
\Big(\int_Q\!|f|^p\,\mathrm{d}x\Big)^{\frac{1}{p}}\leq C|Q|^{\frac{1}{p}-\frac{1}{q}}.
\]
Moreover, the norm $\|f\|_{M^{p,q}(\R^d)}$ is defined as the smallest possible $C$. Then $M^{p,q}(\R^d)$ is a Banach function space with the Fatou property. When $p=q$ we have $M^{p,p}(\R^d)=L^p(\R^d)$, and when $q=\infty$ we have $M^{p,\infty}(\R^d)=L^\infty(\R^d)$ by the Lebesgue differentiation theorem. Moreover, by H\"older's inequality we have
\[
L^q(\R^d)\subseteq M^{p,q}(\R^d).
\]
The space $M^{p,q}(\R^d)$ is $p$-convex, but, if $q>p$, not $s$-concave for any $s<\infty$.

For $1\leq p\leq q\leq\infty$ we define the block space $B^{q,p}(\R^d)$ as the space of functions $g\in L^0(\R^d)$ for which there exists a sequence $\lambda\in\ell^1$ and a sequence of functions $(b_n)_{n\geq 1}$ with the property that for each $n\geq 1$ there is a cube $Q_n$ such that $\supp b_n\subseteq Q_n$ and
\[
|Q_n|^{\frac{1}{p}-\frac{1}{q}}\Big(\int_{Q_n}\!|b_n|^q\,\mathrm{d}x\Big)^{\frac{1}{q}}\leq 1,
\]
such that, pointwise a.e., we have
\[
g=\sum_{n=1}^\infty \lambda_n b_n.
\]
The norm $\|g\|_{B^{q,p}(\R^d)}$ is defined as the smallest possible value of $\|\lambda\|_{\ell^1}$ for which such a representation exists. The space $B^{q,p}(\R^d)$ is a Banach function space with the Fatou property. Moreover, it is $q$-concave, but, if $q>p$, not $r$-convex for any $r>1$.

We define weighted variants of Morrey and block spaces by setting
\[
\|f\|_{M_w^{p,q}(\R^d)}:=\|fw\|_{M^{p,q}(\R^d)},\quad \|g\|_{B_w^{q,p}(\R^d)}:=\|gw\|_{B^{q,p}(\R^d)}.
\]
These spaces relate to each other through the K\"othe duality
\[
M_w^{p,q}(\R^d)'=B_{w^{-1}}^{p',q'}(\R^d).
\]
\begin{remark}
Different weighted Morrey and block spaces can be obtained by changing the underlying measure space to a weighted one, which results in the Lebesgue measure $|Q|$ being replaced by $w(Q)$ in the definition. In this section we only consider the spaces where the weight is added as a multiplier, which results in the so-called weighted Morrey spaces of Samko type \cite{Sa09}.
\end{remark}

Conjecture~\ref{con:mduality2} in this setting takes the following form.
\begin{conjecture}
Let $1<p\leq q<\infty$ and let $w$ be a weight. If $M:B_{w^{-1}}^{p',q'}(\R^d)\to B_{w^{-1}}^{p',q'}(\R^d)$, then $M:M^{p,q}_w(\R^d)\to M^{p,q}_w(\R^d)$.
\end{conjecture}

The converse implication is false. The following counterexample was communicated to us by Andrei Lerner, and is adapted from \cite[Remark~1.12]{NS17}:
\begin{example}\label{ex:morreydoesnotimplyblock}
Let $1<p<q<\infty$. It is shown in \cite[Proposition~4.2]{Ta15} that $M$ is bounded on $M_w^{p,q}(\R^d)$ with power weights $w(x)=|x|^{\alpha d}$ if and only if $-\tfrac{1}{q}\leq\alpha<\tfrac{1}{q'}$. In particular, this holds in the case $\alpha=-\tfrac{1}{q}$, i.e., for
\[
w(x):=|x|^{-\frac{d}{q}}.
\]
However, since for this weight we have $\ind_{\R^d}\in M_w^{p,q}(\R^d)$, it follows from Corollary~\ref{cor:failuresufficientcondition} that we do not have $M:B_{w^{-1}}^{p',q'}(\R^d)\to B_{w^{-1}}^{p',q'}(\R^d)$.
\end{example}

As noted in the above example, Tanaka showed in \cite{Ta15} that for any power weight $w(x)=|x|^{\alpha d}$ with with $-\tfrac{1}{q}\leq \alpha<\tfrac{1}{q'}$, we have $M:M^{p,q}_w(\R^d)\to M^{p,q}_w(\R^d)$. Moreover, he showed that in this case, this is precisely the Muckenhoupt condition $M^{p,q}_w(\R^d)\in A$.

Example~\ref{ex:morreydoesnotimplyblock} shows that we don't have 
\begin{equation}\label{eq:weightedblock1}
M:B^{p',q'}_{w^{-1}}(\R^d)\to B^{p',q'}_{w^{-1}}(\R^d)
\end{equation}
when $w(x):=|x|^{-\frac{d}{q}}$, which leaves open the question of whether this bound is true for the power weights in the $A_q$ range, i.e., for $-\tfrac{1}{q}< \alpha<\tfrac{1}{q'}$. It was shown in \cite[Theorem~3.7]{Ni23} that \eqref{eq:weightedblock1} holds for any weight $w$ satisfying
\[
\sup_Q\langle w\rangle_{q,Q}\langle w^{-1}\rangle_{p',Q}<\infty,
\]
which is satisfied by the power weights $w(x)=|x|^{\alpha d}$ with
\[
-\frac{1}{q}<\alpha<\frac{1}{p'}.
\]
We will now show that the bound also holds in the remaining range $\tfrac{1}{p'}\leq \alpha<\tfrac{1}{q'}$. This can be deduced from Theorem~\ref{thm:Bbfs} and a result of Duoandikoetxea and Rosenthal from \cite{DR20}.
\begin{theorem}
Let $1<p\leq q<\infty$ and let $w(x):=|x|^{\alpha d}$ for $-\tfrac{1}{q}<\alpha<\frac{1}{q'}$. Then we have the following assertions:
\begin{enumerate}[(a)]
    \item\label{it:powerweightmorrey1} $T:M^{p,q}_w(\R^d)\to M^{p,q}_w(\R^d)$ for all Calder\'on-Zygmund operators $T$;
    \item\label{it:powerweightmorrey2} $M:M^{p,q}_w(\R^d)\to M^{p,q}_w(\R^d)$ and $M:B_{w^{-1}}^{p',q'}(\R^d)\to B_{w^{-1}}^{p',q'}(\R^d)$.
\end{enumerate}
\end{theorem}
\begin{proof}
The statement \ref{it:powerweightmorrey1} follows from \cite[Theorem~1.1]{DR20}, which asserts that this bound holds for any operator bounded on $L^2_w(\R^d)$ with respect to all $w\in A_2$. The second assertion \ref{it:powerweightmorrey2} follows from \ref{it:powerweightmorrey1} and Theorem~\ref{thm:Bbfs}.
\end{proof}
As can be deduced from \cite[Theorem~1.1]{DR20}, the above result actually holds for a more general class of weights.

For more general weights $w$ satisfying $M^{p,q}_w(\R^d)\in A$, Tanaka provides several sufficient conditions for the boundedness $M:M^{p,q}_w(\R^d)\to M^{p,q}_w(\R^d)$ in \cite[Theorem~4.1]{Ta15}. Moreover, Duoandikoetxea and Rosenthal show in \cite[Theorem~5.1]{DR21} that if $1<p\leq q<\infty$, and $M^{p,q}_w(\R^d)\in A$, then we have $M:M^{p,q}_w(\R^d)\to M^{p,q}_w(\R^d)$ if $w$ satisfies the condition $w(\cdot+a)\in A_{p,\text{loc}}$ for some $a\in\R^d$, where $w\in A_{p,\text{loc}}$ is defined through
\[
[w]_{p,\text{loc}}:=\sup_B\langle w\rangle_{p,B}\langle w^{-1}\rangle_{p',B},
\]
where the supremum is taken over all balls $B=B(x;r)$ with $r<\tfrac{1}{2}|x|$. A similar result also holds for $p=1$ when the boundedness of $M$ is replaced by the weak-type analogue. We remark here that the condition $w\in A_{p,\text{loc}}$ also appears in the study of weighted norm inequalities related to the Schr\"odinger operator, see \cite{BHS11}.

The following question is still open:
\begin{question}\label{question:morreychar}
Let $1< p\leq q\leq \infty$ and let $w$ be a weight. Is it true that $M^{p,q}_w(\R^d)\in A$ if and only if $M:M^{p,q}_w(\R^d)\to M^{p,q}_w(\R^d)$?
\end{question}
Partial results are given in \cite{DR22, Le22}. Given $1\leq p\leq q\leq\infty$ and a collection of cubes $\mc{P}$, we define the weighted Morrey space with $M^{p,q}_w(\R^d;\mc{P})$ analogous to $M^{p,q}_w(\R^d)$, but this time with the defining property only being required for all $Q\in\mc{P}$, and
\[
\|f\|_{M^{p,q}_w(\R^d;\mc{P})}:=\sup_{Q\in\mc{P}}|Q|^{-(\frac{1}{p}-\frac{1}{q})}
\Big(\int_Q\!|fw|^p\,\mathrm{d}x\Big)^{\frac{1}{p}}.
\]
Moreover, we set $B^{p',q'}_{w^{-1}}(\R^d;\mc{P}):=M^{p,q}_w(\R^d;\mc{P})'$. The following was shown by Lerner in \cite{Le22}.
\begin{theorem}\label{thm:localmorreylacunary}
Let $(x_n)_{n\geq 1}$ be a sequence of points in $\R^d$ with the property that there is a $\nu>1$ such that for all $n\neq m$ we have
\[
\max\{|x_n|,|x_m|\}\leq\nu|x_n-x_m|,
\]
and let $\mc{P}$ be the collection of cubes centered at the points $\{x_n:n\geq 1\}$.

If $1<p\leq q<\infty$ and $w$ is a weight, then the following are equivalent:
\begin{itemize}
    \item $M:M^{p,q}_w(\R^d;\mc{P})\to M^{p,q}_w(\R^d;\mc{P})$;
    \item $M^{p,q}_w(\R^d;\mc{P})\in A$.
\end{itemize}
\end{theorem}
When $\{x_n:n\geq 1\}=\{0\}$, the collection $\mc{P}=:\mc{P}_0$ is the collection of all cubes centered at $0$, and this result was already proved by Duoandikoetxea and Rosenthal in \cite[Theorem~6.1]{DR21}. As a matter of fact, in \cite[Theorem~4.1]{DR22} they showed that all Calder\'on-Zygmund operators are bounded on $M^{p,q}_w(\R^d;\mc{P}_0)$ if and only if 
\begin{equation}\label{eq:localmorreymuckenhoupt}
\sup_{Q}|Q|^{-1}\|w\ind_Q\|_{M^{p,q}(\R^d;\mc{P}_0)}\|w^{-1}M(\ind_Q)\|_{B^{p',q'}(\R^d;\mc{P}_0)}<\infty.
\end{equation}
Observing that
\[
M(\ind_Q)(x)\eqsim_d \frac{|Q|}{(\ell(Q)+|x-c_Q|)^d},
\]
where $\ell(Q)$ and $c_Q$ respectively denote the side length and the center of $Q$, one can recognize \eqref{eq:localmorreymuckenhoupt} as a condition appearing in two-weight boundedness characterizations of singular integrals from, e.g., the work of Muckenhoupt and Wheeden \cite{MW76} and subsequent works.

Combining \eqref{eq:localmorreymuckenhoupt} with Theorem~\ref{thm:Bbfs}, we obtain the following characterization.
\begin{theorem}\label{thm:localmorrey}
Let $\mc{P}_0$ denote the collection of cubes centered at $0$ and let $1<p\leq q<\infty$. Then the following are equivalent:
\begin{enumerate}[(i)]
    \item $T:M^{p,q}_w(\R^d;\mc{P}_0)\to M^{p,q}_w(\R^d;\mc{P}_0)$ for all Calder\'on-Zygmund operators $T$;
    \item $R_j:M^{p,q}_w(\R^d;\mc{P}_0)\to M^{p,q}_w(\R^d;\mc{P}_0)$ for all Riesz transforms $R_j$, $j=1,\ldots,d$;
    \item $M:M^{p,q}_w(\R^d;\mc{P}_0)\to M^{p,q}_w(\R^d;\mc{P}_0)$ and $M:B_{w^{-1}}^{p',q'}(\R^d;\mc{P}_0)\to B_{w^{-1}}^{p',q'}(\R^d;\mc{P}_0)$;
    \item\label{it:localmorrey4} $M^{p,q}_w(\R^d;\mc{P}_0)\in A$ and $M:B_{w^{-1}}^{p',q'}(\R^d;\mc{P}_0)\to B_{w^{-1}}^{p',q'}(\R^d;\mc{P}_0)$;
    \item\label{it:localmorrey5} \eqref{eq:localmorreymuckenhoupt} holds.
\end{enumerate}
\end{theorem}
\begin{proof}
By Theorem~\ref{thm:Bbfs} and \cite[Theorem~4.1]{DR22}, it suffices to prove \ref{it:localmorrey4}$\Rightarrow$\ref{it:localmorrey5}. For any cube $Q$ we have
\begin{align*}
\|w\ind_Q\|_{M^{p,q}_w(\R^d;\mc{P}_0)}&\|w^{-1}M(\ind_Q)\|_{B^{p',q'}(\R^d;\mc{P}_0)}
=\|\ind_Q\|_{M^{p,q}_w(\R^d;\mc{P}_0)}\|M(\ind_Q)\|_{B_{w^{-1}}^{p',q'}(\R^d;\mc{P}_0)}\\
&\leq\|M\|_{B_{w^{-1}}^{p',q'}(\R^d;\mc{P}_0)\to B_{w^{-1}}^{p',q'}(\R^d;\mc{P}_0)}[M^{p,q}_w(\R^d;\mc{P}_0)]_{A}|Q|,
\end{align*}
as desired.
\end{proof}
This leads us to the following question:
\begin{question}
Can we replace the collection $\mc{P}_0$ in Theorem~\ref{thm:localmorrey} by the more general collections of Theorem~\ref{thm:localmorreylacunary}?
\end{question}
Considering the local nature of the Muckenhoupt condition, we conjecture that to globalize the result one would actually require the stronger condition $M^{p,q}_w(\R^d)\in A_{\text{strong}}$.
\begin{conjecture}
Let $1<p\leq q<\infty$ and let $w$ be a weight. Then the following are equivalent:
\begin{enumerate}[(i)]
    \item\label{it:morreyconjecture1} $T:M^{p,q}_w(\R^d)\to M^{p,q}_w(\R^d)$ for all Calder\'on-Zygmund operators $T$;
     \item $R_j:M^{p,q}_w(\R^d)\to M^{p,q}_w(\R^d)$ for all Riesz transforms $R_j$, $j=1,\ldots,d$;
    \item\label{it:morreyconjecture3} $M:M^{p,q}_w(\R^d)\to M^{p,q}_w(\R^d)$ and $M:B_{w^{-1}}^{p',q'}(\R^d)\to B_{w^{-1}}^{p',q'}(\R^d)$;
    \item\label{it:morreyconjecture4} $M^{p,q}_w(\R^d)\in A_{\text{strong}}$.
\end{enumerate}
\end{conjecture}
By Theorem~\ref{thm:Bbfs}, statements \ref{it:morreyconjecture1}-\ref{it:morreyconjecture3} are equivalent, and \ref{it:morreyconjecture3}$\Rightarrow$\ref{it:morreyconjecture4} follows from Theorem~\ref{thm:ApropsHL}\ref{it:mmuckenhoupt3}. Hence, the conjectured statement is \ref{it:morreyconjecture4}$\Rightarrow$\ref{it:morreyconjecture1}.

\appendix

\section{Weak-type bounds for sparse operators}\label{app:A}
\begin{lemma}\label{lem:sparsedomofsparse}
Let $\mc{D}$ be a dyadic grid, let $0<\eta<1$, and let $\mc{S}\subseteq\mc{D}$ be a finite $\eta$-sparse collection. Then for all $0<\nu<1$ and all $f\in L^1_{\text{loc}}(\R^d)$ there exists a collection $\mc{E}\subseteq\mc{D}$ such that
\[
A_{\mc{S}}f\lesssim_d\frac{1}{\eta(1-\nu)} A_{\mc{E}}f
\]
and, if for each $Q\in\mc{E}$ we denote the collection of maximal cubes in $\mc{E}$ strictly contained in $Q$ by $\text{ch}_{\mc{E}}(Q)$, we have
\[
\sum_{Q'\in\text{ch}_{\mc{E}}(Q)}|Q'|\leq(1-\nu)|Q|.
\]
In particular, $\mc{E}$ is $\nu$-sparse with $E_Q:=Q\backslash\bigcup_{Q'\in\text{ch}_{\mc{E}}(Q)}Q'$.
\end{lemma}
\begin{proof}
Let $\mc{E}_0$ denote the maximal cubes in $\mc{S}$. Then we have
\[
A_{\mc{S}}f=\sum_{Q_0\in\mc{E}_0}A_{\mc{S}(Q_0)}f,
\]
where $\mc{S}(Q_0):=\{Q\in\mc{S}:Q\subseteq Q_0\}$. Fix $Q_0\in\mc{E}_0$ and let
\[
K:=\tfrac{1}{1-\nu}\|A_{\mc{S}}\|_{L^1(\R^d)\to L^{1,\infty}(\R^d)}
\]
so that
\[
E:=\{x\in Q_0:A_{\mc{S}(Q_0)}f(x)>K\langle f\rangle_{1,Q_0}\}
\]
satisfies $|E|\leq(1-\nu)|Q_0|$. Using, e.g., \cite[Lemma~4.4]{NSS24}, we obtain a pairwise disjoint collection of cubes $\text{ch}_{\mc{E}}(Q_0)$ of cubes in $\mc{D}$ contained in $Q_0$ for which
\[
E=\bigcup_{Q\in\text{ch}_{\mc{E}}(Q_0)} Q,
\]
and $\widehat{Q}\backslash E\neq\emptyset$ for all $Q\in\text{ch}_{\mc{E}}(Q_0)$, where $\widehat{Q}$ denotes the dyadic parent of $Q$. Then
\[
\sum_{Q\in\text{ch}_{\mc{E}}(Q_0)}|Q|=|E|\leq(1-\nu)|Q_0|,
\]
as desired.

Next, note that
\begin{equation}\label{eq:appendix1}
A_{\mc{S}(Q_0)}f\leq K\langle f\rangle_{1,Q_0}\ind_{Q_0}+\sum_{Q\in\text{ch}_{\mc{E}}(Q_0)}(A_{\mc{S}(Q_0)}f)\ind_Q
\end{equation}
Fix $Q\in\text{ch}_{\mc{E}}(Q_0)$ and pick $\widehat{x}\in\widehat{Q}\backslash E$. Then
\begin{align*}
\sum_{Q\in\text{ch}_{\mc{E}}(Q_0)}(A_{\mc{S}(Q_0)}f)\ind_Q
&=\sum_{Q\in\text{ch}_{\mc{E}}(Q_0)}\sum_{\substack{Q'\in\mc{S}(Q_0)\\ \widehat{Q}\subseteq Q'}}\langle f\rangle_{1,Q}\ind_Q+\sum_{Q\in\text{ch}_{\mc{E}}(Q_0)}A_{\mc{S}(Q)}f\\
&\leq \sum_{Q\in\text{ch}_{\mc{E}}(Q_0)}A_{\mc{S}(Q_0)}f(\widehat{x})\ind_Q+\sum_{Q\in\text{ch}_{\mc{E}}(Q_0)}A_{\mc{S}(Q)}f\\
&\leq K\langle f\rangle_{1,Q_0}\ind_{Q_0}+\sum_{Q\in\text{ch}_{\mc{E}}(Q_0)}A_{\mc{S}(Q)}f.
\end{align*}
Combining this with \eqref{eq:appendix1} yields
\[
A_{\mc{S}(Q_0)}f\leq 2K\langle f\rangle_{1,Q_0}\ind_{Q_0}+\sum_{Q\in\text{ch}_{\mc{E}}(Q_0)}A_{\mc{S}(Q)}f.
\]
Now define $\mc{E}_1:=\bigcup_{Q_0\in\mc{E}_0}\text{ch}_{\mc{E}}(Q_0)$. Iterating this procedure with $\mc{E}_0$ replaced by $\mc{E}_1$, we inductively obtain a collection $\mc{E}=\bigcup_{k=0}^\infty\mc{E}_k$ for which
\[
A_{\mc{S}}f\leq 2K\sum_{Q\in\mc{E}}\langle f\rangle_{1,Q}\ind_Q.
\]
Noting that $\|A_{\mc{S}}\|_{L^1(\R^d)\to L^{1,\infty}(\R^d)}\lesssim_d\tfrac{1}{\eta}$ now proves the result.
\end{proof}

The following lemma is extracted from the proof of the main result in \cite{DLR16}, and is used in the proof of Theorem~\ref{prop:D}.
\begin{lemma}\label{lem:sparsetoweak}
Let $\mc{D}$ be a dyadic grid, let $0<\nu<1$, and let $\mc{S}\subseteq\mc{D}$ be a finite collection satisfying
\begin{equation}\label{eq:lemsparsetomweak1}
\sum_{Q'\in\text{ch}_{\mc{S}}(Q)}|Q'|\leq(1-\nu)|Q|
\end{equation}
for all $Q\in\mc{S}$, where $\text{ch}_{\mc{S}}(Q)$ denotes the collection of the maximal cubes in $\mc{S}$ strictly contained in $Q$.
Setting
\[
\mc{S}_m:=\{Q\in\mc{S}:4^{-(m+1)}<\langle f\rangle_{1,Q}\leq 4^{-m}\}
\]
for $m\geq 1$, for each $Q\in\mc{S}_m$ there exits a subset $F_m(Q)\subseteq Q$ for which
\[
|F_m(Q)|\leq (1-\nu)^{2^m}|Q|,
\]
and for each $g\in L^1_{\loc}(\R^d)$ we have
\[
\int_{\{A_{\mc{S}}f>2\}\backslash\{ M^{\mc{D}}f>\frac{1}{4}\}}\!|g|\,\mathrm{d}x\leq \sum_{m=1}^\infty 4^{-m}\sum_{Q\in\mc{S}_m}\int_{F_m(Q)}\!|g|\,\mathrm{d}x.
\]
\end{lemma}
\begin{proof}
Set
\[
E:=\{A_{\mc{S}}f>2\}\backslash\{ M^{\mc{D}}f>\frac{1}{4}\}.
\]
If $x\in E$, then for any $Q\in\mc{S}$ with $x\in Q$ we have $\langle f\rangle_{1,Q}\leq M^\mc{D}f(x)\leq \tfrac{1}{4}$. Hence, we have
\[
A_{\mc{S}}f(x)=\sum_{m=1}^\infty A_{\mc{S}_m}f,
\]
where
\[
\mc{S}_m:=\{Q\in\mc{S}:4^{-(m+1)}<\langle f\rangle_{1,Q}\leq 4^{-m}\}
\]
for $m\geq 1$.

Now fix $m\geq 1$. We let $\mc{S}_{m,0}$ denote the maximal cubes (with respect to inclusion) in $\mc{S}_m$. Note that these exist, as $\mc{S}$ was assumed to be finite. Moreover, we iteratively define $\mc{S}_{m,n}$ as the maximal cubes in $\mc{S}_m\backslash\bigcup_{k=0}^{n-1}\mc{S}_{m,k}$. Again since $\mc{S}$ is finite, this means that there is an $N\geq 0$ for which
\[
\mc{S}_m=\bigcup_{n=0}^N\mc{S}_{m,n}.
\]
For $Q\in\mc{S}_{m,n}$ we set $E(Q):=Q\backslash\bigcup_{Q'\in\mc{S}_{m,n+1}}Q'$ so that the collection $(E(Q))_{Q\in\mc{S}_m}$ is pairwise disjoint. 

Fixing $n\geq 0$, for $Q\in\mc{S}_{m,n}$ we define
\[
F_m(Q):=\bigcup_{\substack{Q'\in\mc{S}_{m,n+2^m}\\ Q'\subseteq Q}}Q'
\]
so that, by \eqref{eq:lemsparsetomweak1},
\[
|F_m(Q)|\leq (1-\eta)\Big|\bigcup_{\substack{Q'\in\mc{S}_{m,n+2^m-1}}}Q'\Big|\leq\ldots\leq (1-\eta)^{2^m}|Q|
\]
and
\[
Q\backslash F_m(Q)=\bigcup_{k=0}^{2^m-1}\bigcup_{\substack{Q'\in\mc{S}_{m,n+k}\\ Q'\subseteq Q}} E(Q').
\]
Then we have
\begin{align*}
\sum_{Q\in\mc{S}_m}\int_{E\cap Q\backslash F_m(Q)}\!|g|\,\mathrm{d}x
&\leq\sum_{n=0}^N\sum_{Q\in\mc{S}_{m,n}}\sum_{k=0}^{2^m-1}\sum_{\substack{Q'\in\mc{S}_{m,n+k}\\ Q'\subseteq Q}} \int_{E\cap E(Q')}\!|g|\,\mathrm{d}x\\
&\leq 2^m\sum_{Q'\in\mc{S}_m}\int_{E\cap E(Q')}\!|g|\,\mathrm{d}x\\
&\leq 2^m \int_E\!|g|\,\mathrm{d}x
\end{align*}
which implies
\begin{align*}
\int_E\!|g|\,\mathrm{d}x
&\leq\frac{1}{2}\sum_{m=1}^\infty\int_E\!(A_{\mc{S}_m}f)|g|\,\mathrm{d}x\\
&\leq \frac{1}{2}\sum_{m=1}^\infty 4^{-m}\sum_{Q\in\mc{S}_m}\int_{E\cap Q}\!|g|\,\mathrm{d}x\\
&\leq\frac{1}{2}\sum_{m=1}^\infty 2^{-m} \int_E\!|g|\,\mathrm{d}x+\frac{1}{2}\sum_{m=1}^\infty 4^{-m}\sum_{Q\in\mc{S}_m}\int_{F_m(Q)}\!|g|\,\mathrm{d}x\\
&=\frac{1}{2}\int_E\!|g|\,\mathrm{d}x+\frac{1}{2}\sum_{m=1}^\infty 4^{-m}\sum_{Q\in\mc{S}_m}\int_{F_m(Q)}\!|g|\,\mathrm{d}x.
\end{align*}
The result follows.
\end{proof}

\section*{Acknowledgments}
The author wishes to thank Emiel Lorist for the numerous discussions on the duality conjecture for the Hardy-Littlewood maximal operator, for providing his comments on the text, and, in particular, for pointing out \cite[Theorem~2.3.1]{KLW23}, which is the key result used for proving Theorem~\ref{thm:A}.

Finally, I wish to thank the anonymous referee for their very thorough and insightful feedback, greatly improving the presentation and quality of this work.

\bibliography{bieb2}
\bibliographystyle{alpha}
\end{document}